\documentclass{article}

\usepackage{psfrag}
\usepackage[parfill]{parskip}
\usepackage{float, graphicx}
\usepackage[]{epsfig}
\usepackage{amsmath, amsthm, amssymb,}
\usepackage{epsfig}
\usepackage{verbatim}
\usepackage{multicol}
\usepackage{url}
\usepackage{subcaption}

\usepackage{latexsym}
\usepackage{mathrsfs}
\usepackage[colorlinks, bookmarks=true]{hyperref}
\usepackage{graphicx}
\usepackage{amsmath}
\usepackage{enumerate}
\usepackage[normalem]{ulem}
\usepackage{bm}
\usepackage{dsfont}
\usepackage{stmaryrd}
\usepackage{enumitem}
\usepackage[noadjust]{cite}
\usepackage{tikz}
\usepackage{indentfirst}
\usepackage{color}
\usepackage{xcolor}
\usepackage{hyperref}
\usepackage{mathtools}
\usepackage[margin=1in]{geometry}
\usepackage{cleveref}

\makeatletter

\def\@settitle{\begin{center}%
    \bfseries
 \normalfont\LARGE\@title
  \end{center}%
}
\def\@setauthors{\begin{center}%
 \normalsize\@author
  \end{center}%
}

\makeatother

\numberwithin{equation}{section}

\renewcommand{\cal}{\mathcal}
\newcommand\cA{{\mathcal A}}
\newcommand\cB{{\mathcal B}}

\newcommand{\cC}{{\cal C}}
\newcommand{\cD}{{\cal D}}

\newcommand{\cE}{{\cal E}}
\newcommand{\cF}{{\cal F}}
\newcommand{\cG}{{\cal G}}

\newcommand\cH{{\mathcal H}}
\newcommand{\cK}{{\cal K}}

\newcommand{\cM}{{\cal M}}
\newcommand{\cN}{{\cal N}}

\newcommand{\cP}{{\cal P}}

\newcommand{\cU}{{\mathcal U}}

\newcommand{\cX}{{\mathcal X}}
\newcommand{\cY}{{\mathcal Y}}

\newcommand{\sfJ}{{\mathsf J}}
\newcommand{\sfA}{{\mathsf A}}

\newcommand{\bfS}{{\bf S}}

\newcommand{\fa}{{\mathfrak a}}
\newcommand{\fb}{{\mathfrak b}}
\newcommand{\fc}{{\mathfrak c}}

\newcommand{\fo}{{\mathfrak o}}

\newcommand{\fp}{{\mathfrak p}}
\newcommand{\fq}{{\mathfrak q}}

\newcommand{\fg}{{\mathfrak g}}

\newcommand{\bmr}{{\bm{r}}}

\newcommand{\bmu}{{\bm{u}}}

\newcommand{\bmx}{{\bm{x}}}

\newcommand{\bfi}{{\bf i}}
\newcommand{\fC}{{\mathfrak C}}

\newcommand{\be}{\begin{equation}}
\newcommand{\ee}{\end{equation}}

\newcommand{\rd}{{\rm d}}

\newcommand{\ri}{\mathrm{i}}

\newcommand{\bC}{{\mathbb C}}
\newcommand{\bE}{\mathbb{E}}
\newcommand{\bH}{\mathbb{H}}

\newcommand{\bP}{\mathbb{P}}

\newcommand{\bR}{{\mathbb R}}
\newcommand{\bX}{{\mathbb X}}

\newcommand{\bT}{\mathbb T}
\newcommand{\bV}{\mathbb V}
\newcommand{\bW}{\mathbb W}

\newcommand{\bI}{\mathbb{I}}

\newcommand{\al}{\alpha}

\newcommand{\la}{\lambda}

\DeclareMathOperator{\Tr}{Tr}

\DeclareMathOperator{\dist}{dist}

\DeclareMathOperator{\Adm}{{Adm}}

\DeclareMathOperator{\OO}{O}
\DeclareMathOperator{\oo}{o}

\renewcommand{\Re}{\mathop{\mathrm{Re}}}
\renewcommand{\Im}{\mathop{\mathrm{Im}}}

\newcommand{\deq}{\mathrel{\mathop:}=} 
 
\renewcommand{\leq}{\leqslant}
\renewcommand{\geq}{\geqslant}

\newcommand{\del}{\partial}

\newcommand{\wh}{\widehat}
\newcommand{\wt}{\widetilde}

\newcommand{\qq}[1]{[\![{#1}]\!]}

\newcommand{\beq}{\begin{equation}}
\newcommand{\eeq}{\end{equation}}

\theoremstyle{plain} 
\newtheorem{theorem}{Theorem}[section]
\newtheorem*{theorem*}{Theorem}
\newtheorem{lemma}[theorem]{Lemma}
\newtheorem*{lemma*}{Lemma}
\newtheorem{corollary}[theorem]{Corollary}
\newtheorem*{corollary*}{Corollary}
\newtheorem{proposition}[theorem]{Proposition}
\newtheorem*{proposition*}{Proposition}

\newtheorem*{assumption*}{Assumption}
\newtheorem{claim}[theorem]{Claim}

\newtheorem{definition}[theorem]{Definition}
\newtheorem*{definition*}{Definition}
\newtheorem{example}[theorem]{Example}
\newtheorem*{example*}{Example}
\newtheorem{remark}[theorem]{Remark}

\newtheorem*{remark*}{Remark}
\newtheorem*{remarks*}{Remarks}

\newcommand{\col}{\vcentcolon}

\newcommand{\rhosc}{\rho_{\mathrm{sc}}}

\newcommand{\msc}{m_{\rm sc}}
\newcommand{\md}{m_d}

\newcommand{\sfS}{{\sf S}}

\newcommand{\sfF}{{\sf F}}

\def\author#1{\par
    {\centering{\authorfont#1}\par\vspace*{0.05in}}
}

\def\titlefont{\fontsize{13}{15}\bfseries\boldmath\selectfont\centering{}}
\def\authorfont{\fontsize{13}{15}}

\let\affiliationfont\rhfont

\def\address#1{\par
    {\centering{\affiliationfont#1\par}}\par\vspace*{11pt}
}

\def\body{
\setcounter{footnote}{0}
\def\thefootnote{\alph{footnote}}
\def\@makefnmark{{$^{\rm \@thefnmark}$}}
}

\def\title#1{
    \thispagestyle{plain}
    \vspace*{-14pt}
    \vskip 79pt
    {\centering{\titlefont #1\par}}%
    \vskip 1em
}

\def\rhosc{\rho_{\text{sc}}}

\usepackage{scalerel}
\newcommand{\cT}{{\mathcal T}}

\definecolor{forestgreen}{RGB}{34, 139, 34}

\newcommand{\GG}{{\mathcal G}}
\newcommand{\fR}{{\mathfrak R}}
\newcommand{\tG}{{\widetilde G}}
\newcommand{\tL}{{\widetilde L}}
\newcommand{\tcG}{{\widetilde \cG}}
\newcommand{\oOmega}{\overline{\Omega}}

\begin{document}

\title{Lecture Notes on Edge Universality for Random Regular Graphs}

\vspace{1.2cm}

\noindent \begin{minipage}[c]{0.45\textwidth}
 \author{Jiaoyang Huang}
\address{University of Pennsylvania\\
   huangjy@wharton.upenn.edu}
 \end{minipage}
\begin{minipage}[c]{0.45\textwidth}
 \author{Horng-Tzer Yau}
\address{Harvard University \\
   htyau@math.harvard.edu}

 \end{minipage}

\begin{abstract}

The purpose of this note is to explain the structure, general strategy, and main
ideas of the proof in the work of Huang, McKenzie, and Yau (2024) on the Ramanujan
property and edge universality of random regular graphs. The core of the argument
is the derivation of self-consistent equations and a microscopic version of the loop
equations for random $d$-regular graphs. We first recall the local law for random
$d$-regular graphs, and then illustrate the main ideas behind
the derivation of the self-consistent equations and the first loop equation.

\end{abstract}

\bigskip
{
\hypersetup{linkcolor=black}
\setcounter{tocdepth}{1}
\tableofcontents
}

\newpage

\section{Introduction}
We give an exposition of the ideas and proofs in \cite{huang2024ramanujan}. 
The overall strategy consists of three main components:
\begin{enumerate}
\item proving concentration for the self-consistent equations, via estimates of 
the expectations of high moments of these equations, leading to optimal 
concentration of eigenvalues;
\item proving a microscopic version of the loop equations, which can be viewed 
as a refinement of the self-consistent equations with an additional correction 
term;
\item deriving edge universality using the loop equations together with Dyson 
Brownian motion.
\end{enumerate}

We start in \Cref{s:Greenf} by recalling some useful identities for the Green's
function, which are standard tools in random matrix theory. In 
\Cref{s:local_law}, we recall the local law for random \(d\)-regular graphs and
related concepts from the previous work \cite{huang2024spectrum}. 

To illustrate the basic ideas, in this note we only prove estimates for the
expectation (first moment) of the self-consistent equations and derive the
first microscopic version of the loop equation. These are stated in
\Cref{t:recursion}. We then derive edge universality using the loop equations
in \Cref{s:universality}. The proof of \Cref{t:recursion} occupies the
remainder of the note.

The key tool is the local resampling procedure in \Cref{s:local_resampling},
which efficiently exploits the randomness of random \(d\)-regular graphs. Local
resampling produces an exchangeable pair of random \(d\)-regular graphs (before
and after resampling). In \Cref{s:exchangeable}, we explain the basic ideas of
proving concentration using exchangeable pairs, and in \Cref{s:expQ} we discuss
how to apply this method in the setting of random \(d\)-regular graphs. It
turns out, however, that the resulting error is too weak and far from optimal.

To derive self-consistent equations with optimal error, we need to:
\begin{enumerate}
\item develop an iteration mechanism, in which we repeatedly perform local 
resampling and the error improves at each step (see \Cref{s:proofoutline});
\item track the terms arising from the iteration by introducing admissible 
functions and forests in \Cref{s:forest};
\item for terms depending on the graph after local resampling, rewrite them as 
functions of the original graph using either the Schur complement formula (in 
\Cref{sec:schurlemma}) or the Woodbury formula (in \Cref{sec:fanalysis});
\item control the errors generated at each iteration step, as done in 
\Cref{s:change_est}.
\end{enumerate}

We remark that the derivation of the estimates for the self-consistent
equations and the loop equation are not independent. The loop equation (as
proved in \Cref{s:first_loop}) essentially follows from identifying the next
order correction to the self-consistent equations.

\subsection{Random $d$-regular graphs}

A  graph on $N$ vertices $\qq{N}=\{1,\dots,N\}$ is $d$-regular if every vertex has degree $d$.
Necessarily $1\leq d\leq N-1$ and $dN$ is even. In this article we mainly consider the \emph{uniform model} ${\cG_{N,d}}$:
Sample uniformly from the finite set of all \emph{simple} (no loops or multiple edges) $d$-regular graphs on $N$ vertices $\qq{N}$.

Besides the uniform model, several other models are considered in the literature. 
\paragraph{Permutation model.}
When $d$ is even, take $d/2$  independent uniformly distributed permutations $\sigma_1,\dots,\sigma_{d/2}$ on $\qq{N}$. 
The permutation model is the random graph on $N$ vertices obtained by adding an edge $\{i, \sigma_j(i)\}$ for each $i\in\qq{N}$
and $1\leq j\leq d/2$. 
\paragraph{Matching model.}
When $N$ is even, take $d$ independent uniformly distributed perfect matchings $\sigma_1,\dots,\sigma_{d}$ on $\qq{N}$. 
The matching model is the random graph on $N$ vertices obtained by adding an edge $\{i, \sigma_j(i)\}$ for each $i\in\qq{N}$
and $1\leq j\leq d$. 

These models yield a $d$-regular \emph{multigraph} with possible self-loops and
multiple edges; conditioning on simplicity gives the uniform model $\cG_{N,d}$. In this article we mainly discuss the uniform model, but many results also hold for other models.

We can identify a $d$-regular graph with its adjacency matrix $A$. 
Let \(\mathcal{G}\) be a simple, undirected random graph on the vertex set \(\qq{N}\).
Its adjacency matrix \(A=(A_{ij})\in\{0,1\}^{N\times N}\) is defined by \(A_{ij}=1\) if and only if \(i\) and \(j\) are adjacent. For $d$-regular graphs, it is immediate that $A$ has a trivial eigenvalue $d$ with associated eigenvector ${\bm e}=(1,1,\dots,1)^\top$. Moreover, by the Perron-Frobenius theorem, all other eigenvalues are bounded in absolute value by $d$. For convenience, we shall consider the normalized adjacency matrix
\begin{equation} \label{def_H}
H\deq A/\sqrt{d-1}.
\end{equation}

We denote the eigenvalues of the normalized adjacency matrix $H$ of a $d$-regular graph $\GG$ on $N$ vertices as $\lambda_1=d/\sqrt{d-1}\geq\lambda_2\geq \cdots\geq \lambda_N$. We also introduce the {Green's function} and the Stieltjes transform of the empirical eigenvalue distribution of the normalized adjacency matrix $H$ 
\begin{align}\label{e:G}
  G(z) \deq  (H-z)^{-1},\quad   m_N(z) \deq \frac{1}{N} \sum_i G_{ii}(z)=\frac1N\sum_{i=1}^N\frac{1}{\la_i-z},\quad z\in \bC^+.
\end{align}

By local weak convergence, the empirical eigenvalue density of random $d$-regular graphs converges to that of the infinite $d$-regular tree, which is known as the Kesten-McKay distribution; see \cite{kesten1959symmetric,mckay1981expected}. This density is given by 
\begin{align}\label{e:KMdistribution}
\varrho_d(x):=\mathbf1_{x\in [-2,2]} \left(1+\frac1{d-1}-\frac {x^2}d\right)^{-1}\frac{\sqrt{4-x^2}}{2\pi}.
\end{align}
Note that close to the spectral edge $\pm 2$, the Kesten-Mckay distribution has square root behavior:
\begin{align}\label{e:edge_behavior}
x\rightarrow\pm 2,\quad \varrho_d(x)=\frac{\cA\sqrt{2\mp x}}{\pi} +\OO(|2\mp x|), \quad \cA:=\frac{d(d-1)}{(d-2)^2}.
\end{align}
We denote by $\md(z)$ the Stieltjes transform of the Kesten--McKay distribution $\varrho_d(x)$,
\begin{align*}
    \md(z)=\int_\bR \frac{\varrho_d(x)\rd x}{x-z}=(d-1)\frac{-(d-2)z+ d\sqrt{z^2-4}}{2(d^2-(d-1)z^2)},\quad z\in \bC^+:=\{w\in \bC \col \Im[w]>0\}.
\end{align*}
We recall the semicircle distribution $\varrho_{\rm sc}(x)$ and its Stieltjes transform $\msc(z)$:
\begin{align}\begin{split}\label{e:msc_equation}
 \varrho_{\rm sc}(x)=\bm1_{x\in[-2,2]}\frac{\sqrt{4-x^2}}{2\pi},
 \quad 
 \msc(z)=\int_\bR \frac{\varrho_{\rm sc}(x)\rd x}{x-z}=\frac{-z+\sqrt{z^2-4}}{2}.
\end{split}\end{align}
Explicitly, the Stieltjes transform of the Kesten--McKay distribution $\md(z)$ can be expressed in terms of the Stieltjes transform $\msc(z)$:
\begin{align}\label{e:md_equation}
    \md(z)=\frac{1}{-z-\frac{d}{d-1}\msc(z)}.
\end{align}

\subsection{Main Results}

Our main result in \cite{huang2024ramanujan} verifies the edge universality conjecture for random $d$-regular graphs by Sarnak \cite{sarnak2004expander} and Miller, Novikoff and Sabelli \cite{miller2008distribution}.

\begin{theorem}[Edge eigenvalue universality \cite{huang2024ramanujan}]
\label{t:universality}
Fix $d\geq 3$, $k\geq 1$ and $s_1,s_2,\cdots, s_k \in \bR$, and let $\cA=d(d-1)/(d-2)^2$  from \eqref{e:edge_behavior}. There exists a small $\varepsilon>0$ such that the eigenvalues $\lambda_1 = {d}/{\sqrt{d-1}} \geq \lambda_2 \geq \cdots \geq \lambda_N$ of the normalized adjacency matrix $H$ of random $d$-regular graphs satisfy: 
\begin{align*}\begin{split}
 &\phantom{{}={}}\bP_{H}\left( (\cA N)^{2/3} ( \lambda_{2} - 2 )\geq s_1,\cdots, (\cA N)^{2/3} ( \lambda_{k+1} - 2 )\geq s_k \right)\\
 &= \bP_{\mathrm{GOE}}\left( N^{2/3} ( \mu_1 - 2  )\geq s_1,\cdots, N^{2/3} ( \mu_k - 2  )\geq s_k\right) +\OO(N^{-\varepsilon}),
\end{split}\end{align*}
where $\mu_1\geq \mu_2\geq \cdots \geq\mu_N$ are the eigenvalues of the GOE.
The analogous statement holds for the smallest eigenvalues $-\lambda_N,\ldots,-\lambda_{N-k+1}$.
\end{theorem}

When the degree $d$ grows with the size of the graph, edge universality for random $d$-regular graphs has been established previously for $N^{2/3 + \oo(1)} \leq d \leq N/2$, by He \cite{he2022spectral}, and for  $N^{\oo(1)} \leq d \leq N^{1/3 - \oo(1)}$, by the authors  of this paper \cite{huang2023edge}, which generalized a result for $N^{2/9 + \oo(1)} \leq d \leq N^{1/3 - \oo(1)}$, by Bauerschmidt, Knowles and the authors of this paper \cite{bauerschmidt2020edge}.

\Cref{t:universality} implies that the fluctuation of the second largest eigenvalue converges to the Tracy-Widom$_1$ distribution. The  Tracy-Widom$_1$ distribution  has  about $83\%$ of its mass on the set $\{x:x<0\}$\cite{carr_tw1_cdf_2025}.
Therefore \Cref{t:universality} implies $83\%$ of
$d$-regular graphs have the second eigenvalue \emph{less than} $2$.
The proof of \Cref{t:universality} can be extended to show 
that  the largest and smallest nontrivial eigenvalues converge in distribution to \emph{independent} Tracy-Widom$_1$ distributions.
As a consequence, we have the following result.

\begin{corollary}[\cite{huang2024ramanujan}]\label{c:rate}
Fix $d\geq 3$ and $N$ sufficiently large.  With probability approximately $69\%$, a randomly sampled $d$-regular graph has $\max\{\lambda_2,|\lambda_N|\}\leq 2$, and is therefore Ramanujan. 
\end{corollary}

\subsection{Parameters}

In this note we fix the parameters as follows
\begin{align}\label{e:parameters}
    0<\fo\ll \fb\ll  \fc\ll \fg\ll 1,
\end{align} 
set $\fR=(\fc/4)\log_{d-1}N$ and choose $\ell $ such that $\ell/\log_{d-1}N\ll \fb$.
Below, we describe their meanings and where they are introduced:
\begin{itemize}
    \item Many estimates involve bounds containing $N^\fo$ factors, which are harmless.
    \item $\fb$ relates to the concentration of Green's function entries, with errors bounded by $N^{-\fb}$, see \eqref{eq:infbound0} and \eqref{eq:infbound}.
    \item For the spectral parameter $z\in \bC^+$ in Green's functions and Stieltjes transforms, we restrict it to $\Im[z]\geq N^{-1+\fg}$, see \eqref{e:D}.
    \item $\ell$ comes from local resampling in \Cref{s:local_resampling}, we resample  boundary edges of balls with radius $\ell$.
    \item $\fc$ defines $\fR$, and with high probability, random $d$-regular graphs are tree-like within radius $\fR$ neighborhoods, see \Cref{def:omegabar}.
\end{itemize}

We restrict our analysis to the spectral domain 
\begin{align} \label{e:D}
    \mathbf D=\{z\in \bC^+:  N^{-1+\fg}\leq \Im[z]\leq N^{-\fo}, |\Re[z]|\leq 2+N^{-\fo}  \}.
\end{align}
In the spectral domain $\bf D$,  we impose the conditions $\Im[z]\leq N^{-\fo}$ and $|\Re[z]|\leq 2+N^{-\fo}$. These constraints ensure that  $|\msc(z)|$ is close to $1$, specifically satisfying $||\msc(z)|-1|\lesssim N^{-\fo/2}$.

\section{Properties of the Green's functions}
\label{s:Greenf}
Throughout this paper, we repeatedly use some (well-known) identities for Green's functions,
which we collect them here.

\subsection{Resolvent identity}

The following well-known identity is referred as resolvent identity:
for two invertible matrices $A$ and $B$ of the same size, we have
\begin{equation} \label{e:resolv}
  A^{-1} - B^{-1} = A^{-1}(B-A)B^{-1}=B^{-1}(B-A)A^{-1}.
\end{equation}

\subsection{Woodbury formula}
Let $A+UCV^\top$ be a rank $r$ perturbation of $A$. Namely, $U, V\in \bR^{N\times r}$ and $C\in \bR^{r\times r}$. Then, the Woodbury formula  gives us 
\begin{align}\label{e:woodbury}
(A+UCV^\top)^{-1}-A^{-1}=-A^{-1}U(C^{-1}+V^\top A^{-1} U)^{-1}V^\top A^{-1}.
\end{align}

\subsection{Schur complement formula}

Given an $N\times N$ matrix $H$ and an index set $\bT \subset \qq{N}$, recall that we denote by
$H|_\bT$ the $\bT \times \bT$-matrix obtained by restricting $H$ to $\bT$,
and that by $H^{(\bT)} = H|_{\bT^\complement}$ the matrix obtained by removing
the rows and columns corresponding to indices in $\bT$.
Thus, for any $\bT \subset \qq{N}$,
any symmetric matrix $H$ can be written (up to rearrangement of indices) in the block form
\begin{equation*}
  H = \begin{bmatrix} A& B^\top\\ B &D  \end{bmatrix},
\end{equation*}
with $A=H|_{\bT}$ and $D=H^{(\bT)}$.
The Schur complement formula asserts that, for any $z\in \bC^+$,
\begin{equation} \label{e:Schur}
 G=(H-z)^{-1}= \begin{bmatrix}
   (A-B^\top G^{(\bT)} B)^{-1} & -(A-B^\top G^{(\bT)} B)^{-1}B^\top G^{(\bT)} \\
   -G^{(\bT)} B(A-B^\top G^{(\bT)} B)^{-1} & G^{(\bT)}+G^{(\bT)} B(A-B^\top G^{(\bT)} B)^{-1}B^\top G^{(\bT)} 
 \end{bmatrix},
\end{equation}
where $G^{(\bT)}=(D-z)^{-1}$.
Throughout the paper, we often use the following special cases of \eqref{e:Schur}:
\begin{align} \begin{split}\label{e:Schur1}
  G|_{\bT} &= (A-B^\top G^{(\bT)} B)^{-1},\\
  G|_{\bT\bT^\complement}&=-G|_{\bT}B^\top G^{(\bT)},\\
   G|_{\bT^\complement}&=G^{(\bT)}+G|_{\bT^\complement\bT}(G|_{\bT})^{-1}G|_{\bT\bT^\complement}=G^{(\bT)}-G^{(\bT)}BG|_{\bT\bT^\complement},
  \end{split}
\end{align}
as well as the special case
\begin{equation} \label{e:Schurixj}
G_{ij}^{(k)} = G_{ij}-\frac{G_{ik}G_{kj}}{G_{kk}}
=G_{ij}+(G^{(k)}H)_{ik} G_{kj}.
\end{equation}

\subsection{Ward identity}

For any symmetric $N\times N$ matrix $H$, its Green's function $G(z)=(H-z)^{-1}$ satisfies
the Ward identity
\begin{equation} \label{e:Ward}
  \sum_{j=1}^{N} |G_{ij}(z)|^2= \frac{\Im G_{jj}(z)}{\eta},
\end{equation}
where $\eta=\Im [z]$. This can be deduced from using \eqref{e:resolv} on $G-G^*$. This identity provides a bound for the sum $\sum_{j=1}^{N} |G_{ij}(z)|^2$
in terms of the diagonal entries of the Green's function.

\section{Local law for random $d$-regular graphs}
\label{s:local_law}

In this section, we present the local law results from \cite{huang2024spectrum}. It states that the Green's function of random $d$-regular graphs are approximated by the Green's function extension with general weights, which are easy to compute. We introduce the concept of Green's function extension with general weights in \Cref{s:pre},  and state the local law in \Cref{s:local_law2}.

\subsection{Locally tree-like graphs}

Random $d$-regular graphs are locally tree-like in the sense that each radius $\OO(\log_{d-1}(N))$ neighborhoods is either a truncated tree (contains no cycles), or contains at most one cycle.

\begin{definition}\label{def:omegabar}
Fix $d\geq 3$ and a sufficiently small $0<\fc<1$, $\fR=(\fc /4)\log_{d-1}N$ as in \eqref{e:parameters}. We define the event $\oOmega$,  where the following occur: 
    \begin{enumerate}
        \item 
The number of vertices that do not have a  tree neighborhood of radius $\fR$ is at most $N^{\fc}$.
        \item 
        The radius $\fR$ neighborhood of each vertex has an excess (i.e., the number of independent cycles) of at most $1$. 
    \end{enumerate}
\end{definition}

The event $\oOmega$ is a typical event. The following proposition from \cite[Proposition 2.1]{huang2024spectrum} states that $\oOmega$ holds with high probability. 
\begin{proposition}[{\cite[Proposition 2.1]{huang2024spectrum}}]\label{lem:omega}
$\oOmega$ occurs with probability $1-\OO(N^{-(1-\fc)})$.
\end{proposition}

As we will see, for graphs $\cG\in \oOmega$,  their Green's functions can be approximated by tree extensions with overwhelmingly high probability. 
For the infinite $d$-regular tree and the infinite $(d-1)$-ary tree (trees where the root has degree $d-1$ and all other vertices have degree $d$),
the following proposition computes their Green's function explicitly.
\begin{proposition}[{\cite[Proposition 2.2]{huang2024spectrum}}]\label{greentree}
Let $\cX$ be the infinite $d$-regular tree.
For all $z \in \bC^+$, its Green's function is
\begin{equation} \label{e:Gtreemkm}
  G_{ij}(z)=m_{d}(z)\left(-\frac{\msc(z)}{\sqrt{d-1}}\right)^{\dist_{\cX}(i,j)},
\end{equation}
where $\dist_\cX(i,j)$ is the graph distance of the two vertices $i,j$ in $\cX$.
Let $\cY$ be the infinite $(d-1)$-ary tree with root vertex $o$.
Its Green's function is
\begin{equation} \label{e:Gtreemsc}
  G_{ij}(z)=m_{d}(z)\left(1-\left(-\frac{\msc(z)}{\sqrt{d-1}}\right)^{2{\rm anc}_{\cY}(i,j)+2}\right)\left(-\frac{\msc(z)}{\sqrt{d-1}}\right)^{\dist_{\cY}(i,j)},
\end{equation}
where ${\rm anc}_\cY(i,j)$ is the distance from the common ancestor of the vertices $i,j$ to the root $o$. 
In particular,
\begin{align}\label{e:Gtreemsc2}
G_{oi}(z)=\msc(z)\left(-\frac{\msc(z)}{\sqrt{d-1}}\right)^{\dist_{\cY}(o,i)}.
\end{align}
\end{proposition}

\subsection{Green's function extension with general weights}\label{s:pre}

We recall the notion of a Green’s function extension with general weight $\Delta$ from {\cite[Section 2.3]{huang2024spectrum}}. Roughly speaking, we consider the radius-$\ell$ neighborhood $\cT = \cB_\ell(o, \cG)$ of a vertex $o$ in $\cG$. Assume it is a tree, then the boundary vertices of $\cT$ are leaves (i.e., vertices of degree 1), to each of which we attach a weight $-\Delta$; see the left panel of \Cref{fig:three-trees}. The resulting Green’s function of $\cT$ with these boundary weights is referred to as the Green’s function extension with general weight $\Delta$. The following formal definition slightly extend this idea to allow cycles.

\begin{definition}\label{def:pdef}
    Fix degree $d\geq 3$, and a graph $\cT$ with degrees bounded by $d$. We define the function $P(\cT,z,\Delta)$ as follows. We denote $A(\cT)$  the adjacency matrix of $\cT$, $D(\cT)$ the diagonal matrix of degrees of $\cT$, and $\bI$ the diagonal matrix indexed by the vertex set $\bT$ of $\cT$. Then 
    \begin{align}\label{e:defP}
    P(\cT,z,\Delta):=\frac{1}{-z+A(\cT)/\sqrt{d-1}-(d\mathbb I -D(\cT))\Delta/(d-1)}.
    \end{align}
\end{definition}

The matrix $P(\cT,z,\Delta)$ is the Green's function of the matrix obtained from $A(\cT)/\sqrt{d-1}$ by attaching to each vertex $i\in \cT$ a weight $-(d-D_{ii}(\cT))\Delta/(d-1)$, see the middle panel of \Cref{fig:three-trees} ($d=3$ and one boundary vertex has degree $2$). When $\Delta=\msc(z)$, \eqref{e:defP} is the Green's function of the tree extension of $\cT$, i.e. extending $\cT$ by attaching copies of infinite $(d-1)$-ary trees to $\cT$ to make each vertex degree $d$. If $\cT$ is a tree, then in this case, the Green's function agrees with the Green's function of the infinite $d$-regular tree, as in \eqref{e:Gtreemkm}. 
For any vertex set $\bX$ in $\cT$, we define the following Green's function with vertex $\bX$ removed, see the right panel of \Cref{fig:three-trees}. Let $A^{(\bX)}(\cT)$ and $D^{(\bX)}(\cT)$ denote the matrices obtained from $A(\cT), D(\cT)$ by removing the row and column associated with vertex $\bX$, and $\bI$ the diagonal matrix indexed by the vertex set $\bT\setminus \bX$. The Green's function is then defined as:
 \begin{align}\label{e:defPi}
    P^{(\bX)}(\cT,z,\Delta):=\frac{1}{-z+A^{(\bX)}(\cT)/\sqrt{d-1}-(d\mathbb I-D^{(\bX)}(\cT))\Delta/(d-1)},
    \end{align}


\begin{figure}[t]
  \centering

  \begin{subfigure}{0.32\textwidth}
    \centering
    \resizebox{\linewidth}{!}{%
    \begin{tikzpicture}[
      every node/.style={circle,draw,inner sep=1.5pt},
      >=stealth
    ]

    \node[draw=none] at (-2.5,2) {\huge $\mathcal T$};

    \node (r) at (0,0) {};
    \node[draw=none, below] at (r) {\Large $o$};

    \draw[dashed] (0,0) circle (4cm);

    \node (a1) at (90:1.8)  {};
    \node[draw=none, right] at (a1) {\Large $i$};

    \node (a2) at (210:1.8) {};
    \node (a3) at (330:1.8) {};

    \draw (r) -- (a1);
    \draw (r) -- (a2);
    \draw (r) -- (a3);


    \node (a1b1) at (70:3.2)  {$-\Delta$};
    \node (a1b2) at (110:3.2) {$-\Delta$};
    \draw (a1) -- (a1b1);
    \draw (a1) -- (a1b2);

    \node (a2b1) at (190:3.2) {$-\Delta$};
    \node (a2b2) at (230:3.2) {$-\Delta$};
    \draw (a2) -- (a2b1);
    \draw (a2) -- (a2b2);

    \node (a3b1) at (310:3.2) {$-\Delta$};
    \node (a3b2) at (350:3.2) {$-\Delta$};
    \draw (a3) -- (a3b1);
    \draw (a3) -- (a3b2);

    \end{tikzpicture}%
    }
  \end{subfigure}
  \hfill
  \begin{subfigure}{0.32\textwidth}
    \centering
    \resizebox{\linewidth}{!}{%
    \begin{tikzpicture}[
      every node/.style={circle,draw,inner sep=1.5pt},
      >=stealth
    ]

    \node[draw=none] at (-2.5,2) {\huge $\mathcal T$};
    \node (r) at (0,0) {};
    \node[draw=none, below] at (r) {\Large $o$};
    \draw[dashed] (0,0) circle (4cm);
    \node (a1) at (90:1.8)  {};
    \node[draw=none, right] at (a1) {\Large $i$};
    \node (a2) at (210:1.8) {};
    \node (a3) at (330:1.8) {};
    \draw (r) -- (a1);
    \draw (r) -- (a2);
    \draw (r) -- (a3);
    \node (a1b1) at (70:3.2)  {$-\Delta$};
    \node (a1b2) at (110:3.2) {$-\Delta$};
    \draw (a1) -- (a1b1);
    \draw (a1) -- (a1b2);
    \node (a2b2) at (230:3.2) {$-\Delta$};

    \draw (a2) -- (a2b2);
    \node (a3b1) at (310:3.2) {${\color{red}-\Delta/2}$};
        \draw (a2) -- (a3b1);
    \node (a3b2) at (350:3.2) {$-\Delta$};
    \draw (a3) -- (a3b1);
    \draw (a3) -- (a3b2);

    \end{tikzpicture}%
    }
  \end{subfigure}
  \hfill
  \begin{subfigure}{0.32\textwidth}
    \centering
    \resizebox{\linewidth}{!}{%
    \begin{tikzpicture}[
      every node/.style={circle,draw,inner sep=1.5pt},
      >=stealth
    ]

    \node[draw=none] at (-2.5,2) {\huge $\mathcal T^{(i)}$};
    \node (r) at (0,0) {};
    \node[draw=none, below] at (r) {\Large $o$};
    \draw[dashed] (0,0) circle (4cm);
    \node (a1) at (90:1.8)  {};
    \node[draw=none, right] at (a1) {\Large $i$};
    \node (a2) at (210:1.8) {};
    \node (a3) at (330:1.8) {};
    \draw[dashed] (r) -- (a1);
    \draw (r) -- (a2);
    \draw (r) -- (a3);
    \node (a1b1) at (70:3.2)  {$-\Delta$};
    \node (a1b2) at (110:3.2) {$-\Delta$};
    \draw[dashed] (a1) -- (a1b1);
    \draw[dashed] (a1) -- (a1b2);
    \node (a2b1) at (190:3.2) {$-\Delta$};
    \node (a2b2) at (230:3.2) {$-\Delta$};
    \draw (a2) -- (a2b1);
    \draw (a2) -- (a2b2);
    \node (a3b1) at (310:3.2) {$-\Delta$};
    \node (a3b2) at (350:3.2) {$-\Delta$};
    \draw (a3) -- (a3b1);
    \draw (a3) -- (a3b2);

    \end{tikzpicture}%
    }
  \end{subfigure}

\caption{Left panel: for a truncated $d$-regular tree $\cT$, we attach a weight $-\Delta$ to each leaf vertex; middle panel: for a general graph $\cT$, we attach a weight $-(d-D_{ii}(\cT))/(d-1)$ to each vertex; right panel: remove the vertex set $\bX=\{i\}$.}

  \label{fig:three-trees}
\end{figure}
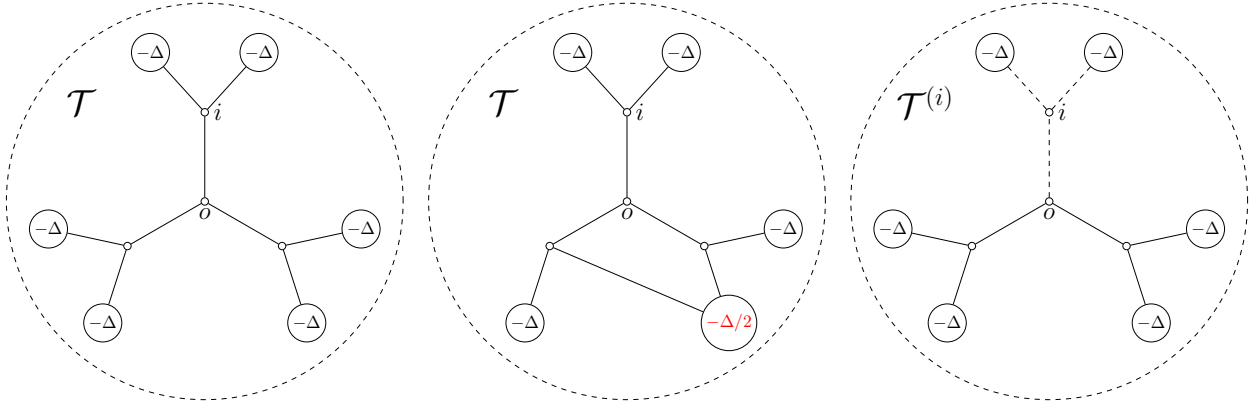

For any integer $\ell\geq 1$, we define the functions $X_\ell(\Delta,z), Y_\ell(\Delta,z)$ as
\begin{align}\label{def:Y}
X_\ell(\Delta,z)=P_{oo}(\cB_\ell(o,\cX),z,\Delta),\quad Y_\ell(\Delta,z)=P_{oo}(\cB_\ell(o,\cY),z,\Delta),
\end{align}
where $\cX$ is the infinite $d$-regular tree with root vertex $o$, and $\cY$ is the infinite $(d-1)$-ary tree with root vertex $o$. 
Then $\msc(z)$ is a fixed point of the function $Y_\ell$, i.e. $Y_\ell(\msc(z),z)=\msc(z)$. Also, $X_\ell(\msc(z),z)=\md(z)$. 
The following proposition states that if $\Delta$ is sufficiently close to $\msc(z)$, then $Y_\ell(\Delta,w)$ is close to $\msc(z)$, and $X_\ell(\Delta,w)$ is close to $\md(z)$.

\begin{proposition}\label{p:recurbound}
Given $z, \Delta, \Delta'\in \bC^+$ such that $\ell^2|\Delta-\msc(z)|, \ell^2|\Delta'-\msc(z)|\ll 1$, then
the functions $X_\ell(\Delta,z), Y_\ell(\Delta,z)$ 
satisfy
\begin{align}\begin{split}\label{e:recurbound}
Y_\ell(\Delta,z)-\msc(z)
&=\msc^{2\ell+2}(z)\md(z)\left(\frac{1-\msc^{2\ell+2}(z)}{d-1}+\frac{d-2}{d-1}\frac{1-\msc^{2\ell+2}(z)}{1-\msc^2(z)}\right)(\Delta-\msc(z))^2\\
&+\msc^{2\ell+2}(z)(\Delta-\msc(z))+\OO(\ell^5|\Delta-\msc(z)|^3).
\end{split}\end{align}
and
\begin{align}\begin{split}\label{e:Xrecurbound}
X_\ell(\Delta,z)-\md(z)
&=\frac{d}{d-1}\md^2(z)\msc^{2\ell}(z)(\Delta-\msc(z))+\OO\left(|\Delta-\msc(z)|^2\right).
\end{split}\end{align}
Moreover, the difference $Y_\ell(\Delta,z)-Y_\ell(\Delta',  z)$ satisfies
\begin{align}\label{e:Yl_derivative}
   |Y_\ell(\Delta,  z)-Y_\ell(\Delta',  z)|\lesssim \ell|\Delta-\Delta'|, 
\end{align}
\end{proposition}

\begin{proof}[Proof of Proposition \ref{p:recurbound}]
The proofs for $X_\ell$ and $Y_\ell$ are identical, so we will only provide the proof for $Y_\ell$.
We denote $\cH=\cB_\ell(o, \cY)$, which is the truncated $(d-1)$-ary tree at level $\ell$. We denote its vertex set as $\bH$ and normalized adjacency matrix as $H$. We denote $\bI$ and $\bI^\del$ the diagonal matrices, such that for $x,y\in \bH$,  $\bI_{xy}=\delta_{xy}$ and $\bI^\del_{xy}=\bm1(\dist_\cH(x,o)=\ell)\delta_{xy}$. Then \eqref{e:defP} gives that
\begin{align*}
Y_\ell( \Delta, z)=P(\cH,z,\Delta)
=\left(H-z-\Delta \mathbb I^\del\right)^{-1},\quad P=P(\cH,z,\msc(z))=(H-z-\msc(z)\bI^\del)^{-1}.
\end{align*}
In the rest of the proof, we will simply write $\msc=\msc(z)$. We can compute the Green's function $P(\cH, z, \Delta)$ by a perturbation argument,
\begin{align}\begin{split}\label{e:expansionGP}
P(\cH, z, \Delta)
&=\left(H-z-\Delta \mathbb I^\del\right)^{-1}=\left(H-z-\msc\mathbb I^\del-(\Delta-\msc)\mathbb I^\del)\right)^{-1}\\
&=P+\sum_{k\geq 1}(\Delta-\msc)^{k}P(\mathbb I^\del P)^k.
\end{split}\end{align}
With the explicit expression of $P$ as given in \eqref{e:Gtreemsc2}, we can compute
\begin{align}\label{e:PBPoo}
\left(P\mathbb I^\del P\right)_{oo}
=\sum_{l}P_{ol}P_{lo}=\sum_{l,l'}\msc^2 \left(-\frac{\msc}{\sqrt{d-1}}\right)^{2\ell}
&=\msc^{2\ell+2}
\end{align}
where the summation is over $l,l'$ such that $\dist_{\cH}(l,o)=\dist_{\cH}(l',o)=\ell$, and there are $(d-1)^{\ell}$ values of $l$. Moreover, for $k\geq 2$ we will show the following two relations for $P$,
\begin{align}
   \label{e:Pboundary} &\left(P\mathbb I^\del P\mathbb I^\del P\right)_{oo}=\msc^{2\ell+2}\md\left(\frac{1-\msc^{2\ell+2}}{d-1}+\frac{d-2}{d-1}\frac{1-\msc^{2\ell+2}}{1-\msc^2}\right) ,\\
    \label{e:Ptotalsum}&(|P|^k)_{oo}\lesssim (C\ell)^{2k-3},
\end{align}
where for each $i,j\in \bH$, $|P|_{ij}\deq |P_{ij}|$.

The relation \eqref{e:Pboundary} follows from explicit computation using \eqref{e:Gtreemsc} and \eqref{e:Gtreemsc2}
\begin{align*}
&\phantom{{}={}}\sum_{l, l'}P_{ol}P_{ll'}P_{l'o}=\sum_{l,l'}\msc^2 \left(-\frac{\msc}{\sqrt{d-1}}\right)^{2\ell}\md\left(1-\left(-\frac{\msc}{\sqrt{d-1}}\right)^{2+2{\rm anc}_\cH(l, l')}\right)\left(-\frac{\msc}{\sqrt{d-1}}\right)^{\dist_{\cH}(l,l')}\\
&= \frac{\msc^{2\ell+2}}{(d-1)^{\ell}}\sum_{l,l'} \md\left(1-\left(\frac{\msc}{\sqrt{d-1}}\right)^{2+2{\rm anc}_\cH(l, l')}\right)\left(-\frac{\msc}{\sqrt{d-1}}\right)^{\dist_{\cH}(l,l')}\\
&=\msc^{2\ell+2}\md\left(1-\left(\frac{\msc}{\sqrt{d-1}}\right)^{2+2\ell}
+\sum_{r=1}^\ell \left(1-\left(\frac{\msc}{\sqrt{d-1}}\right)^{2+2(\ell-r)}\right)\left(\frac{\msc}{\sqrt{d-1}}\right)^{2r}(d-2)(d-1)^{r-1}
\right)\\
&=\msc^{2\ell+2}\md\left(\frac{1-\msc^{2\ell+2}}{d-1}+\frac{d-2}{d-1}\frac{1-\msc^{2\ell+2}}{1-\msc^2}\right),
\end{align*}
where the summation in the first line is over $l,l'$ such that $\dist_{\cH}(l,o)=\dist_{\cH}(l',o)=\ell$; in the second to last line we used that there are $(d-1)^\ell$ choices of $l$; for a given $l$, there are $(d-2)(d-1)^{r-1}$ values of $l'$ such that $\dist_{\cH}(l, l')=2r, {\rm anc}_\cH(l, l')=\ell-r$ for $1\leq r\leq \ell$. When $l=l'$, it holds $\dist_{\cH}(l, l')=0, {\rm anc}_\cH(l, l')=\ell$.

To prove \eqref{e:Ptotalsum}, we show by induction that for any vertex $i$ such that $\dist_{\cH}(o,i)\leq \ell$,
\begin{align}\label{e:inductionPk}
    (|P|^k)_{o i}\lesssim \frac{(C\ell)^{2k-2}}{(d-1)^{\dist_{\cH}(o,i)/2}}.
\end{align}
The statement for $k=1$ follows from \eqref{e:Gtreemsc}. Assume the statement \eqref{e:inductionPk} holds for $k-1$, we prove it for $k$,
\begin{align}\label{e:inductionPk2}
   (|P|^k)_{o i}= \sum_{j\in \bH} (|P|^{k-1})_{o j}|P|_{ji}
   \lesssim\sum_{j\in \bH} \frac{(C\ell)^{2k-4}}{(d-1)^{\dist_{\cH}(o,j)/2}}\frac{1}{(d-1)^{\dist_{\cH}(j,i)/2}}.
\end{align}
We denote the path from $o$ to $i$ as $\cP$, then $|\{j\in \bH: \dist_{\cH}(j,\cP)=r\}|\leq \ell(d-1)^{r}$, and for $\dist_{\cH}(j,\cP)=r$, we have
\begin{align*}
    \frac{1}{(d-1)^{\dist_{\cH}(o,j)/2}}\frac{1}{(d-1)^{\dist_{\cH}(j,i)/2}}\lesssim \frac{1}{(d-1)^r}\frac{1}{(d-1)^{\dist_{\cH}(o,i)/2}}.
\end{align*}
In this way, we can reorganize the sum over $j$ in \eqref{e:inductionPk2} according to its distance to $\cP$,
\begin{align*}
     (|P|^k)_{o i}&= \sum_{r=0}^\ell \sum_{j\in \bH:\dist_{\cH}(j,\cP)=r }\frac{(C\ell)^{2k-4}}{(d-1)^{\dist_{\cH}(o,j)/2}}\frac{1}{(d-1)^{\dist_{\cH}(j,i)/2}}\\
     &\lesssim \sum_{r=0}^\ell|\{j: \dist_{\cH}(j,\cP)=r\}|\frac{(C\ell)^{2k-4}}{(d-1)^r}\frac{1}{(d-1)^{\dist_{\cH}(o,i)/2}}\lesssim \frac{(C\ell)^{2k-2}}{(d-1)^{\dist_{\cH}(o,i)/2}},
\end{align*}
which shows \eqref{e:inductionPk}. The claim \eqref{e:Ptotalsum} is a consequence of \eqref{e:inductionPk}
\begin{align*}
    (|P|^k)_{oo}
    =\sum_{i\in \bH} (|P|^{k-1})_{oi}|P|_{io}
    \lesssim \sum_{i\in \bH} \frac{(C\ell)^{2k-4}}{(d-1)^{\dist_{\cH}(o,i)/2}}\frac{1}{(d-1)^{\dist_{\cH}(i,o)/2}}
    \lesssim (C\ell)^{2k-3}.
\end{align*}

%
%
%

The claim \eqref{e:recurbound} follows from plugging \eqref{e:PBPoo}, \eqref{e:Pboundary} and \eqref{e:Ptotalsum} into \eqref{e:expansionGP}, and use $\ell^2|\Delta-\msc|\ll1$:
\begin{align}\begin{split}\label{e:expansionGP2}
Y_\ell(\Delta,z)-\msc(z)
&=\sum_{k\geq 1}(\Delta-\msc)^{k}(P(\mathbb I^\del P)^k)_{oo}\\
&=\msc^{2\ell+2}(z)\md(z)\left(\frac{1-\msc^{2\ell+2}(z)}{d-1}+\frac{d-2}{d-1}\frac{1-\msc^{2\ell+2}(z)}{1-\msc^2(z)}\right)(\Delta-\msc(z))^2\\
&+\msc^{2\ell+2}(z)(\Delta-\msc(z))+\sum_{k\geq 3}|\Delta-\msc(z)|^k(|P|^{k+1})_{oo}.
\end{split}\end{align}

The claim \eqref{e:Yl_derivative} follows from bounding the derivative
\begin{align}\label{e:derivative}
|\del_1 Y_\ell(\Delta,  z)|\lesssim \ell.
\end{align}
The claim \eqref{e:derivative} follows from taking derivative with respect to $\Delta$ on both sides of \eqref{e:expansionGP2}, and plugging in \eqref{e:PBPoo} and \eqref{e:Ptotalsum} 
\begin{align*}
|\del_1 Y_\ell(\Delta,  z)|
\lesssim  |(P\mathbb I^\del P)_{oo}|+\sum_{k\geq 1}|\Delta-\msc|^k (|P|^{k+2})_{oo}
\lesssim 1+C\ell \sum_{k\geq 1} (C\ell^2 |\Delta-\msc|)^k \lesssim \ell.
\end{align*}

\end{proof}

%
%
%
%

\subsection{Local law of $H$}
\label{s:local_law2}

To state the local law of $H$,  we introduce a quantity $Q(z)$ that was first defined in \cite{bauerschmidt2019local, huang2024spectrum}, and plays a crucial role in the proof of local law. The quantity is the average of $G_{jj}^{(i)}(z)$ over all pairs of adjacent vertices $i\sim j$:
\begin{align}\label{e:Qsum}
Q(z):=\frac{1}{Nd}\sum_{i\sim j}G_{jj}^{(i)}(z).
\end{align}


For any vertex set $\bX\subset \qq{N}$, and integer $r\geq 1$, we denote
$\cB_r(\bX,\cG)=\{i\in \qq{N}: \dist_\cG(i, \bX)\leq r\}$ the ball of radius-$r$ around vertices $\bX$ in $\GG$. 
The weak local law of $H$ has been proven in {\cite[Theorem 4.2]{huang2024spectrum}}, which is recalled below (by taking $(\fa, \fb,\fc, \mathfrak r)=(12,300,\fc,\fc/32)$). 

\begin{theorem}[{\cite[Theorem 4.2]{huang2024spectrum}}] \label{thm:prevthm0}
Fix any sufficiently small $0<\fb\ll\fc<1$, $\fR=(\fc/4)\log_{d-1}(N)$ and any $z\in \bC^+$, we define $\eta(z)=\Im[z], \kappa(z)=\min\{|\Re[z]-2|, |\Re[z]+2|\}$, and the error parameters
\begin{align}\label{e:defepsilon}
\varepsilon'(z):=(\log N)^{100}\left(N^{-10\fb} +\sqrt{\frac{\Im[m_d(z)]}{N\eta(z)}}+\frac{1}{(N\eta(z))^{2/3}}\right),\quad \varepsilon(z):=\frac{\varepsilon'(z)}{\sqrt{\kappa(z)+\eta(z)+\varepsilon'(z)}}.
\end{align}
For any $\fC\geq 1$ and $N$ large enough, with probability at least $1-\OO(N^{-\fC})$ with respect to the uniform measure on $\oOmega$, 
\be\label{eq:infbound0}
|G_{ij}(z)-P_{ij}(\cB_{\fR/100}(\{i,j\},\cG),z,\msc(z))|,\quad |Q(z)-\msc(z)|,\quad |m_N(z)-m_d(z)|\lesssim \varepsilon(z)
\ee
uniformly for every $i,j\in \qq{N}$, and any $z\in \bC^+$ with $\Im[z]\geq (\log N)^{300}/N$. 
\end{theorem}

\begin{definition}\label{def:omega}
We denote by \(\Omega\subset \overline{\Omega}\) the set of \(d\)-regular graphs
for which \eqref{eq:infbound0} holds.
\end{definition}

As a consequence of \Cref{thm:prevthm0} we have the following statement.
\begin{claim}
Uniformly for every $i,j\in \qq{N}$, and any $z\in \bC^+$ with $|z|\leq 1/\fg, \Im[z]\geq N^{-1+\fg}$, we have the following bounds: 
\begin{align}\begin{split}\label{eq:infbound}
&|G_{ij}(z)-P_{ij}(\cB_{\fR/100}(\{i,j\},\cG),z,\msc(z)|\lesssim N^{-2\fb} ,\\
 &|Q(z)-\msc(z)|,\quad |m_N(z)-m_d(z)|\lesssim N^{-2\fb} .
\end{split}\end{align}
and for any $i,j\in \qq{N}$
\begin{align}\label{e:Gsize}
    \Im[G_{ij}(z)]\lesssim N^{\fo/2} \Im[m_N(z)]
\end{align}
\end{claim}
\begin{proof}
    For $\Im[z]\geq N^{-1+\fg}$, thanks to \eqref{e:defepsilon}, \begin{align}\label{e:epsilon_bound}
    \varepsilon(z)\leq \sqrt{\varepsilon'(z)}\leq N^{-4\fb}.
\end{align}  
Thus, the claim \eqref{eq:infbound} follows from \eqref{eq:infbound0}.

We denote the eigenvectors of $H$ as $\bmu_1, \bmu_2, \cdots, \bmu_N$. We notice that the first statement in \eqref{eq:infbound0} implies that for $\Im[z]\geq (\log N)^{300}/N$, $\Im[G_{ii}(z)]\lesssim 1$. It follows that eigenvectors are delocalized \begin{align*}
    \|\bmu_\al\|^2_\infty\lesssim \max_{1\leq i\leq N}((\log N)^{300}/N)\Im[G_{ii}(\lambda_\alpha+\ri (\log N)^{300}/N)]\lesssim (\log N)^{300}/N\ll N^{\fo/2-1}.
\end{align*} 
Thus \eqref{e:Gsize} follows
\begin{align}\label{e:smallGij}
        |\Im[G_{xy}(z)]|=\left|\sum_{\al=1}^N\frac{\eta (\bmu_\al\bmu^\top_\al)_{xy}}{|\la_\al-z|^2}\right|
        \leq N^{\fo/2-1}\sum_{\al=1}^N\frac{\eta }{|\la_\al-z|^2}=N^{\fo/2} \Im[m_N(z)].
    \end{align}
\end{proof}

The following lemma states that after removing some vertices of $\cG$, the local law \eqref{eq:infbound} still holds (with possibly worse error).
\begin{lemma}\label{l:basicG}
Let $z\in \bC^+$ satisfy $|z|\leq 1/\fg$, and set $\eta:=\Im z\geq N^{-1+\fg}$.
Fix a $d$-regular graph $\cG\in \Omega$ (recall \Cref{def:omega}) and vertices
$\{i,o\}$ in $\cG$, and let $\cT=\cB_\ell(o,\cG)$ with vertex set $\bT$. For any $0\leq \mu\lesssim (d-1)^\ell$, take a family of directed edges
$\{(b_\alpha, c_\alpha)\}_{\alpha\in \qq{\mu}}$, set
$\bW=\{b_\alpha\}_{\alpha\in \qq{\mu}}$, and condition on the event that
\begin{align}\label{e:condition}
A_{io}\prod_{\alpha\in \qq{\mu}}A_{b_\alpha c_\alpha}
\prod_{x\neq y\in \{o, c_1, c_2,\dots, c_\mu\}}
\bm1\bigl(\cB_{\fR/2}(x,\cG) \text{ is a tree}\bigr)
\bm1\bigl(\dist_\cG(x,y)\geq \fR/2\bigr)=1.
\end{align}

Then, for any $\bX$ of the form
$\bX\in\{\bW,\,\bT\cup \bW\}$ and any $x,y\not\in \bX$, we have
\begin{align}\label{eq:local_law}
\bigl|G^{(\bX)}_{xy}(z)
- P^{(\bX)}_{xy}\bigl(\cB_{\fR/100}(\{x,y\}\cup \bX,\cG),z,\msc(z)\bigr)\bigr|
\lesssim N^{-\fb},\qquad
\bigl|\Im G^{(\bX)}_{xy}(z)\bigr|
   \lesssim N^\fo \Im[ m_N(z)],
\end{align}
and for any $x\not\in \bX$,
\begin{align}\label{e:Gest}
\frac{1}{N}\sum_{b\sim c\notin \bX}\bigl|G_{cx}^{(\bX)}(z)\bigr|^2
   \lesssim \frac{N^\fo \Im m_N(z)}{N\eta},\qquad
\frac{1}{Nd}\sum_{b\sim c\not\in \bX}
  \bigl|G_{c x}^{(b \bX)}(z)\bigr|^2
  \lesssim \frac{N^\fo \Im[m_N(z)]}{N\eta},
\end{align}
where the summation is over directed edges $(b,c)$ of $\cG$, and $b,c\not\in \bX$.
\end{lemma}

\begin{proof}
We will only prove the case that $\bX=\bT$, the other cases can be proven in the same way, so we omit. 
 From the Schur complement formula \eqref{e:Schur1}, we have
    \begin{align}\label{e:GT}
       G^{(\bT)}=G-G(G|_{\bT})^{-1}G.
    \end{align}

     We denote $\cT=\cB_\ell(o,\cG)$ and $P=P(\cB_{\fR/100}(\bT\cup \{x,y\},\cG), z, \msc(z))$. We recall that by \eqref{e:condition} the vertex \(o\) has a tree neighborhood of radius \(\fR/2\). The connected component component of \(o\) in
\(\cB_{\fR/100}(\bT\cup \{x,y\},\cG)\) is a truncated $d$-regular tree. 
As a consequence, when restricted to the connected component of \(o\) in
\(\cB_{\fR/100}(\bT\cup \{x,y\},\cG)\), \(P\) is the Green's function
of the \(d\)-regular tree.
     
     Then \eqref{e:defP} gives $(P|_\bT)^{-1}=H|_\bT-z-\msc(z)\bI^\del$, where the diagonal matrix $\bI_{ij}^{\del}=\bm1(\dist_{\cT}(o,i)=\ell)\delta_{ij}$ for $i,j\in \bT$. And for any $i \in \bT$,
\begin{equation} \label{e:sumPTSbd}
\sum_{j\in \bT} |(P|_{\bT})^{-1}_{ij}|
\lesssim \sum_{j\in \bT}|H_{ij}-z\delta_{ij}-\bI_{ij}^\del|\lesssim1.
\end{equation}

Moreover, if $x$ is in the connected component of \(o\) in
\(\cB_{\fR/100}(\bT\cup \{x,y\},\cG)\), then 
\eqref{eq:infbound} implies
\begin{align}\label{e:G-P}
|G_{ij}-P_{ij}|,\quad |G_{xj}-P_{xj}|\leq N^{-2\fb}, \qquad i,j\in \bT.
\end{align}
If \(x\) is not in the connected component of \(o\) in
\(\cB_{\fR/100}(\bT\cup \{x,y\},\cG)\), then \(\dist_\cG(x,j)> \fR/50\) for any
\(j\in \bT\). In this case \(P_{xj}=0\), and by \eqref{eq:infbound} we have
\(|G_{xj}|\leq N^{-2\fb}\). In this case, we also have
\begin{align}\label{e:G-P1}
|G_{ij}-P_{ij}|,\quad |G_{xj}-P_{xj}|\leq N^{-2\fb}, \qquad i,j\in \bT.
\end{align}

We can perform a resolvent expansion according to \eqref{e:resolv} and rewrite $(G|_{\bT})^{-1}_{ij}$ as: for some large enough $\fp\geq 1$
\begin{align}\begin{split}\label{e:sum_Ginverse0}
(G|_{\bT})^{-1}_{ij}-(P|_\bT)^{-1}_{ij}=\left(
\sum_{k=1}^\fp (P|_\bT)^{-1}\left((P|_\bT-G|_{\bT})(P|_\bT)^{-1}\right)^k
\right)_{ij}+\OO\left(\frac{1}{N}\right)\\
\lesssim \sum_{l\in \bT}|((P|_\bT)^{-1})_{il}| N^{-2\fb} \sum_{l\in \bT}|((P|_\bT)^{-1})_{lj}| +\OO\left(\frac{1}{N}\right)\lesssim N^{-2\fb},
\end{split}\end{align}
where in the second statement we used  \eqref{e:sumPTSbd} and \eqref{e:G-P1}. From the above expression, we get
\begin{align}\label{e:sum_Ginverse}
\sum_{j\in \bT}|(G|_\bT)^{-1}_{ij}| \lesssim \sum_{j\in \bT}|(P|_\bT)^{-1}_{ij}| +\OO(|\bT|N^{-2\fb} )\lesssim 1.   
\end{align}

Again, since the vertex $o$ has radius $\fR/2$ tree neighborhood,  for any  $0\leq r\leq \ell$, $|\{j\in \bT: \dist_\cG(x,j)=r\}|\lesssim (d-1)^{r}$. 
  As a consequence, for any $i\in \bT$, \eqref{e:G-P1} and \eqref{e:Gtreemkm} together imply that 
\begin{align}\begin{split}\label{e:sumGix}
&\sum_{j\in \bT}|P_{xj}|\lesssim 
\sum_{j\in \bT}\left(\frac{|\msc|}{\sqrt{d-1}}\right)^{\dist_\cT(x,j)}
 \lesssim \sum_{0\leq r\leq 2\ell}\frac{|\{j\in \bT: \dist(x,j)=r\}|}{(d-1)^{r/2}} \lesssim \ell,\\
&    \sum_{j\in \bT} |G_{xj}|\lesssim 
    \sum_{j\in \bT}|P_{xj}| +\OO(|\bT|N^{-2\fb})\lesssim \ell.
\end{split}\end{align}

%
 
For $P$,  Schur complement formula \eqref{e:Schur1} gives $P^{(\bT)}=P-P(P|_{\bT})^{-1}P$. By taking difference with \eqref{e:GT}, we get
\begin{align}\begin{split}\label{e:GTbbbbb}
 &\phantom{{}={}}|G^{(\bT)}_{xy}-P^{(\bT)}_{xy}|
 \lesssim |G_{xy}-P_{xy}|
 +\sum_{i,j\in \bT}|G_{xi}-P_{xi}||(G|_\bT)^{-1}_{ij}||G_{jy}|\\
&+\sum_{i,j\in \bT}|P_{xi}||(G|_\bT)^{-1}_{ij}-(P|_\bT)^{-1}_{ij}||G_{jy}|+
\sum_{i,j\in \bT}|P_{xi}||(P|_\bT)^{-1}_{ij}||G_{jy}-P_{jy}|\\
&\lesssim |G_{xy}-P_{xy}|+N^{-2\fb} \sum_{i,j\in \bT} \left(
|(G)^{-1}_{ij}||G_{jy}|+
|P_{xi}||G_{jy}|
+|P_{xi}||(P|_\bT)^{-1}_{ij}|\right)\\
&\lesssim \ell^2 N^{2\fb} \lesssim N^{-\fb}.
\end{split}
\end{align}
where in the second inequality, we used
\eqref{e:G-P1} and \eqref{e:sum_Ginverse0}; in the third inequality we used  \eqref{e:sumPTSbd}, \eqref{e:sumGix} and \eqref{e:sum_Ginverse}. This finishes the proof of the first statement in \eqref{eq:local_law}.

 By taking imaginary part on both sides of \eqref{e:GT}, we get
    \begin{align}\begin{split}\label{e:Gimaginary}
        \Im[G_{xy}^{(\bT)}]
        &=\Im[G_{xy}]-\Im[(G (G|_{\bT})^{-1}G)_{xy}]\lesssim \Im[G_{xy}]+\sum_{j\in \bT}\Im[G_{xj}] |(G|_{\bT})^{-1}G)_{jy}|\\
        &+\sum_{j\in \bT}|(G(G|_{\bT})^{-1})_{xj}||\Im[G_{jy}] |
        +\sum_{j,k\in \bT}|G_{xj}|\Im[(G|_{\bT})^{-1}_{jk}]|G_{ky}|.
    \end{split}\end{align}
In the following we show the following estimates
\begin{align}\label{e:sumG}
    \sum_{j\in \bT}(|G_{xj}|+|(G(G|_{\bT})^{-1})_{xj}|+|G_{jy}|+|(G(G|_{\bT})^{-1})_{jy}|)
    \lesssim \ell, \quad \max_{j,k\in \bT}\Im[(G|_{\bT})^{-1}_{jk}]\lesssim \ell^2 N^{\fo/2}\Im[m_N].
\end{align}
Then the second statement in \eqref{eq:local_law} follows from plugging \eqref{e:smallGij} and \eqref{e:sumG} into \eqref{e:Gimaginary}.

We start with the first statement in \eqref{e:sumG}.
  The first statement in \eqref{e:sumG} follows from the estimates \eqref{e:sum_Ginverse} and \eqref{e:sumGix}
\begin{align*}
     \sum_{j\in \bT}(|G_{xj}|+|(G(G|_{\bT})^{-1})_{xj}|)
     \lesssim \sum_{j\in \bT}|G_{xj}|\left(1+\sum_{i\in \bT}|(G|_{\bT})^{-1}_{ji}|\right)\lesssim \ell.
\end{align*}

    Next we prove the second statement in \eqref{e:sumG}. For the imaginary part of a symmetric matrix, we have the following identity
    \begin{align*}
        \Im[A^{-1}]= -A^{-1}\Im[A] \overline{A^{-1}}.
    \end{align*}
    Thus 
    \begin{align}\label{e:Im_Ginverse}
        \Im[(G|_\bT)^{-1}_{xy}]
        =-((G|_\bT)^{-1}\Im[G|_\bT] \overline{(G|_\bT)^{-1}})_{xy}.
    \end{align}
The second statement in \eqref{e:sumG} follows from plugging \eqref{e:smallGij} and \eqref{e:sum_Ginverse} into \eqref{e:Im_Ginverse}.

The first statement in \eqref{e:Gest} follows from a Ward identity \eqref{e:Ward}, and the second statement in \eqref{eq:local_law}:
\begin{align*}
    \frac{1}{N}\sum_{c\notin \bT}|G_{cx}^{(\bT)}|^2=\frac{\Im[G_{xx}^{(\bT)}]}{N\eta}\lesssim \frac{N^\fo\Im[m_N]}{N\eta}. 
\end{align*}

For the second statement in \eqref{e:Gest}, thanks to the Schur complement formula \eqref{e:Schurixj},
\begin{align}\label{e:Gbound}
    |G_{c x}^{(b  \bT)}|=\left|G_{c x}^{(  \bT)}-\frac{G_{cb}^{( \bT)}G_{b x}^{( \bT)}}{G_{bb}^{( \bT)}}\right|\lesssim |G_{cx}^{(  \bT)}|+|G_{b x}^{( \bT)}|,
\end{align}
where in the second statement we used \eqref{eq:local_law} to bound $|G_{cb}^{( \bT)}|\lesssim 1, |G_{bb}^{( \bT)}|\gtrsim 1$. By averaging \eqref{e:Gbound} we get
\begin{align*}
    \frac{1}{Nd}\sum_{b\sim c\not\in  \bT}|G_{c x}^{(b  \bT)}|^2\lesssim  \frac{1}{Nd}\sum_{b\sim c\not\in  \bT}|G_{cx}^{(  \bT)}|^2+|G_{bx}^{( \bT)}|^2\lesssim \frac{ N^\fo \Im[m_N]}{N\eta},
\end{align*}
where the second statement follows from the first statement in \eqref{e:Gest}.

\end{proof}

\section{Edge universality.}\label{s:universality}
We now outline the ideas behind proving edge universality \Cref{t:universality} for random \(d\)-regular graphs. Edge universality follows the standard three-step scheme \cite{erdHos2017dynamical}. The key deviations arise in steps (i) and (iii):
\begin{itemize}
  \item[(i)] In the sparse setting, the optimal concentration of eigenvalue locations is not with respect to the semicircle law as in the Wigner case. The self-consistent equations for the relevant Stieltjes transforms are therefore more intricate. New techniques---local resampling that more effectively exploit the randomness---are introduced to derive these equations and establish concentration; see \Cref{s:three_step} for details.
  \item[(iii)] For random \(d\)-regular graphs with fixed \(d\), standard Wigner-type comparison methods fail. A different mechanism based on loop equations plays a central role; see \Cref{s:loop1} for details.
\end{itemize}

\subsection{Three-step strategy.}\label{s:three_step}
In this section we recall the three-step strategy for proving edge universality for Wigner matrices (see, e.g., \cite{erdHos2017dynamical}) and highlight the key differences that arise in random $d$-regular graph setting. Let \(H\) denote either a Wigner matrix or the normalized adjacency matrix of a random \(d\)-regular graph. We write its eigenvalues in decreasing order as \(\lambda_1 \geq \lambda_2 \geq \cdots \geq \lambda_N\).

  \textbf{Step (i): local law and rigidity.}
The first step is to prove optimal eigenvalue rigidity up to the spectral edge (on the $N^{-2/3}$ scale), yielding precise control of the extreme eigenvalues. This is achieved by deriving a self-consistent equation for the Stieltjes transform $m_N(z)$ (or related spectral observables),
\begin{align*}
  m_N(z):=\frac{1}{N}\sum_{i=1}^N\frac{1}{\lambda_i-z}.
\end{align*}
For Wigner matrices, the self-consistent equation takes the form: with high probability
\begin{align}\label{e:self_eq1}
  \,1+z m_N(z)+m_N^2(z)\approx 0,
\end{align}
whose unique solution is the Stieltjes transform of the semicircle law. The self-consistent equation is proven through carefully estimating high moments $\bE[|1+z m_N(z)+m_N^2(z)|^{2p}]$.

For a random \(d\)-regular graph \(\cG\) on \(N\) vertices with fixed degree \(d\),
obtaining a closed equation for the Stieltjes transform requires a more refined
observable. We recall from \eqref{e:Qsum} the quantity  
\begin{align}\label{e:Qsum0}
Q(z)=\frac{1}{Nd}\sum_{i\sim j}G_{jj}^{(i)}(z),
\end{align}
and the functions \(Y_\ell(Q(z), z)\) and \(X_\ell(Q(z), z)\) from \eqref{def:Y},
which are the Green's functions at the root of a truncated \((d-1)\)-ary tree
and a truncated \(d\)-regular tree of depth \(\ell\asymp \log_{d-1} N\) with
boundary weights \(Q(z)\). It turns out that the quantity \(Q(z)\) and the
Stieltjes transform \(m_N(z)\) satisfy two self-consistent equations:
\begin{align}\label{e:self_consistent}
Q(z)-Y_\ell(Q(z),z)\approx 0,\quad m_N(z)-X_\ell(Q(z),z)\approx 0. 
\end{align}
One of the main results of \cite{huang2024ramanujan} is devoted to proving the
self-consistent equations \eqref{e:self_consistent} via an iteration scheme
utilizing local resampling. This note is devoted to explaining the derivation
of these self-consistent equations by showing that their expectations are small.

  \textbf{Step (ii): strong ergodicity of Dyson Brownian motion.}
We next study the normalized adjacency matrix perturbed by a small GOE matrix
\begin{align}\label{e:Ht}
  H(t)\;\deq\; H + \sqrt{t}\,Z,
\end{align}
where $Z$ is a GOE matrix for Wigner matrices; for random $d$-regular graphs, $Z$ is the GOE restricted to the subspace of matrices with vanishing row and column sums (so that the trivial eigenvector ${\bm e}=(1,\ldots,1)^\top$ is preserved).

Assume the eigenvalues of $H$ satisfy optimal edge rigidity with respect to a limiting density $\varrho$ exhibiting square-root edge behavior, which is the consequence of Step (i). In particular, for Wigner matrices $\varrho$ is the semicircle law \eqref{e:msc_equation}; for random $d$-regular graphs $\varrho$ is the Kesten--McKay law $\varrho_d$ from \eqref{e:KMdistribution}.
  Then, by \cite{landon2017edge,adhikari2020dyson}, after a time $t \geq N^{-1/3+\fc}$ (for arbitrarily small $\fc>0$) the edge statistics of $H(t)$ are universal: the extreme eigenvalues converge to the Airy$_1$ point process (the edge scaling limit of GOE). 

The asymptotic empirical eigenvalue distribution of the matrix $H(t)=H+\sqrt t Z$ can be described by the free additive convolution, from free probability theory \cite{MR1488333}. We denote the semicircle distribution of variance $t$ as $t^{-1/2}\rhosc(t^{-1/2}x)$.
Given a probability measure $\varrho$ on $\bR$, we denote its free convolution with a semicircle distribution of variance $t$ by $\varrho_t:=\varrho\boxplus t^{-1/2}\rhosc(t^{-1/2}x)$. Write the Stieltjes transforms of $\varrho$ and $\varrho_t$ as
\begin{align*}
m(z)=m(z;0)=\int (x-z)^{-1} \varrho(x)\rd x,\quad m(z;t)=\int (x-z)^{-1} \varrho_{t}(x)\rd x, \quad z\in \bC\setminus \bR.
\end{align*} 
The Stieltjes transform $m(z;t)$ solves the complex Burgers equation
\begin{align}\label{e:eq2}
\del_t m(z,t)=\del_z m(z,t) m(z,t)=\frac{1}{2}\del_z (m(z,t)^2).
\end{align}
We denote the right edge of the free convolution density $\varrho_t$ as $E_t$, then its evolution satisfies the following equation
\begin{align}\label{e:Etshift}
 \del_t E_t=m( E_t,t).
\end{align}

  \textbf{Step (iii): Green's function comparison.}
This final step removes the Gaussian component in \eqref{e:Ht} and, via a comparison argument, transfers edge universality from the Gaussian-divisible model to the original ensemble. In particular, it shows that the $k$-point correlation function at the edge is invariant along the flow \eqref{e:Ht}, with \(t\) interpreted as time. This is achieved by the following comparison theorem for multipoint correlation functions at the edge (equivalently, by comparing products of Stieltjes transforms on the microscopic scale). For details on how this theorem yields edge universality, see \cite[Section~17]{erdHos2017dynamical}.

 \begin{proposition}\label{p:small_change}
Fix a small $\fc>0$. We introduce the microscopic window:
 \begin{align}\label{e:micro_window}
     {\bf M}:=\{w\in \bC:N^{-2/3-\fc}\leq |\Im[w]|\leq N^{-2/3+\fc}, -N^{-2/3+\fc}\leq \Re[w]\leq N^{-2/3+\fc}\}.
 \end{align}
  Then for any $p\geq 1$, and $w_1, w_2,\cdots, w_{p}\in {\bf M}$, $z_j(t)=E_t+w_j$ for $1\leq j\leq p$, the following holds
 \begin{align}\label{e:product_s}
    \left.\bE\left[\prod_{j=1}^p N^{1/3}(m_N(z_j(t);t)-m(z_j(t);t))\right]\right|_{t=0}^{t=N^{-1/3+\fc}}=\OO\left( N^{-\fc}\right),
 \end{align}
 provided $N$ is sufficiently large.
 \end{proposition}

For Wigner matrices and random graphs with polynomially growing degree, \Cref{p:small_change} follows by integrating out the GOE perturbation $Z$ via Gaussian integration by parts, expanding the randomness of $H$ via cumulant expansion, and showing that the leading contributions cancel. For random $d$-regular graphs with fixed degree $d$, this argument breaks down. In the proof of \Cref{t:universality}, a key observation is that the time derivative of the multi-point correlation functions of the Stieltjes transform of $H(t)$ is governed by microscopic loop equations. In the next section, we state these loop equations and explain how they yield the Green's function comparison for fixed-degree $d$-regular graphs.

\begin{figure}

\begin{center}
\includegraphics[trim={0cm 0.5cm 0cm 4.5cm}, clip,scale=0.8]{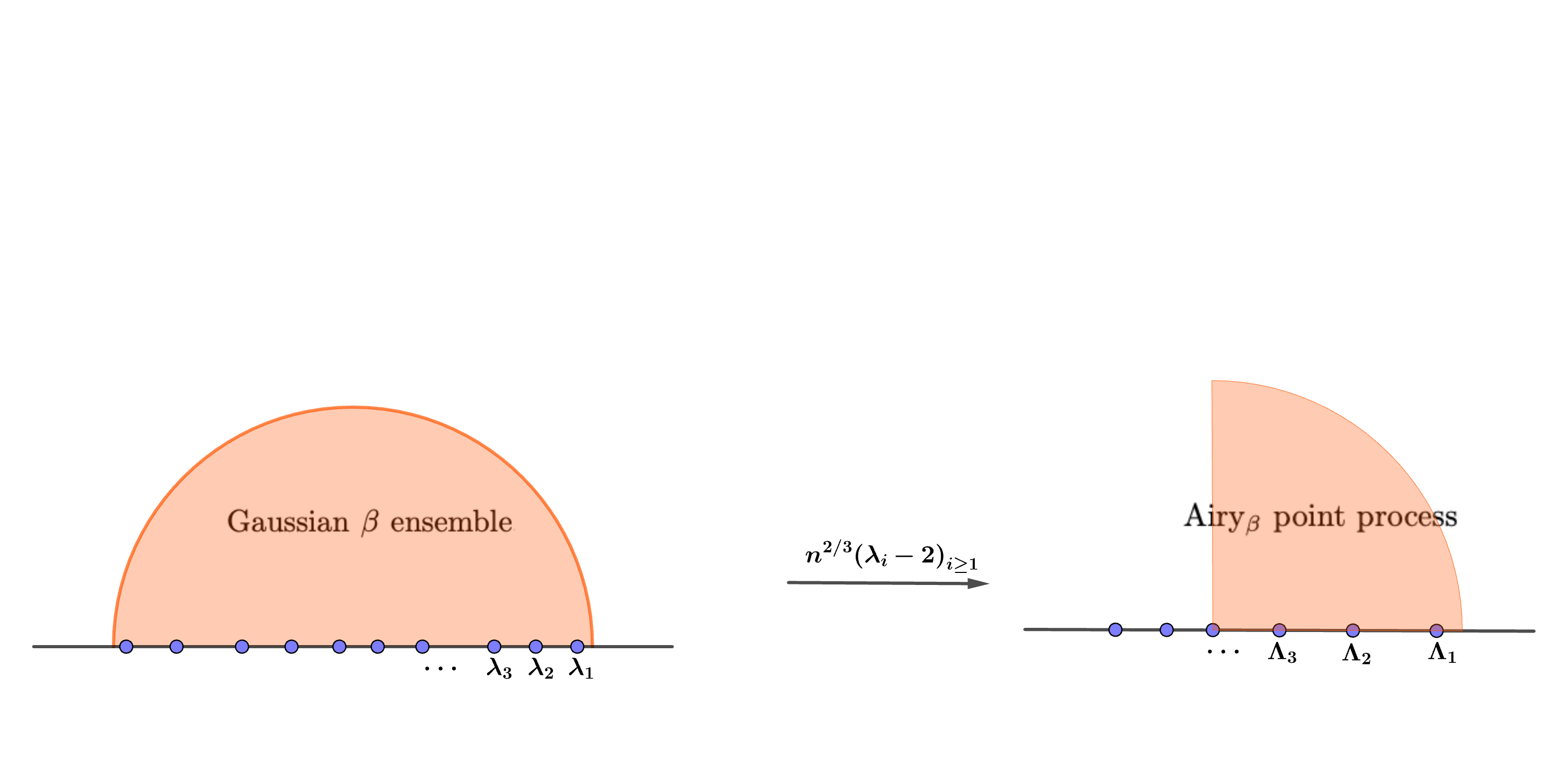}
 \caption{Extreme eigenvalues of Gaussian $\beta$-ensemble converge to the Airy$_\beta$ point process.   }
 \label{f:airy}
 \end{center}
 \end{figure}
 
\begin{center}

\end{center}

\subsection{Loop equations and Green's function comparison}\label{s:loop1}
Loop (or Dyson-Schwinger) equations were first used in the theoretical physics literature (e.g., in the work of Migdal \cite{migdal1983loop}) and were later introduced to the mathematical community by Johansson \cite{johansson1998fluctuations} to derive macroscopic central limit theorems for general $\beta$-ensembles of eigenvalues of random matrices, see also \cite{borot2013asymptotic,borot2024asymptotic, shcherbina2013fluctuations} and the book \cite{guionnet2019asymptotics}. For the Gaussian $\beta$-ensemble, the loop (or Dyson-Schwinger) equation describes a {recursive structure} satisfied by the Stieltjes transform (correlation functions) of the eigenvalues.
The first loop equation is given by: for $z\in \bC\setminus \bR$,
\begin{align}\label{e:loop0}
\bE\left[ m_N(z)^2+zm_N(z)+1+\frac{2-\beta}{\beta}\frac{\del_z m_N(z)}{N}\right]=0,\quad m_N(z):=\frac{1}{N}\sum_{i=1}^N\frac{1}{\lambda_i-z}.
\end{align}
More generally, for any $p\geq 1$, the loop equations at rank $p$ are given by : for $z, z_1, z_2,\cdots, z_{p-1}\in \bC\setminus \bR$
\begin{align}\begin{split}\label{e:loop1}
&\bE\left[\left(m_N(z)^2+zm_N(z)+1+\frac{2-\beta}{\beta}\frac{\del_z m_N(z)}{N}\right)\prod_{j=1}^{ p-1}m_N(z_j)\right]\\
&+\frac{2}{\beta N^2}\sum_{j=1}^{p-1}
    \bE\left[\del_{z_j} \frac{m_N(z)-m_N(z_j)}{z-z_j}\prod_{ i\neq j}m_N(z_i)\right]=0.
\end{split}\end{align}
These loop equations \eqref{e:loop0} and \eqref{e:loop1} arise from the invariance of the eigenvalue integral under infinitesimal reparametrizations of the variables; they can be derived either by integration by parts or, equivalently, by an infinitesimal change of variables.

More generally, one can derive the loop equations for $\beta$-ensembles with general potential $V$
\begin{align*}
P(\lambda_1, \lambda_2,\cdots, \lambda_N)=\frac{1}{Z_N^{\beta, V}} \prod_{i<j} |\lambda_i-\lambda_j|^\beta
\prod_{i=1}^N e^{-\beta N  V(\lambda_i)}.
\end{align*}
Loop (or Dyson-Schwinger) equations play a crucial role in proving {eigenvalue rigidity and universality} for $\beta$-ensembles with general potential \cite{bourgade2012bulk, bourgade2014universality, bourgade2014edge, shcherbina2014change, bekerman2015transport}. 

In general, loop  equations are model-dependent. For $\beta$-ensembles, they depend on the potential $V$. However, when zooming in near the spectral edge, the microscopic version of loop equations becomes {universal}. To introduce \emph{microscopic version of loop equations}, we recall that the extreme eigenvalues of GOE converge to the Airy$_1$ point process $\Lambda_1\geq \Lambda_2\geq \Lambda_3\geq \cdots $; see \Cref{f:airy}. We can introduce the following  normalized Stieltjes transform of the Airy$_1$ point process: 
\begin{align*}
 S(w)=\sum_{i\geq 1}\left(\frac{1}{\Lambda_i-w}-\frac{1}{\fa_i}\right)-\frac{{\rm Ai}'(0)}{{\rm Ai}(0)}, \quad w\in \bC\setminus \bR,
\end{align*}
where $\fa_1>\fa_2>\fa_3>\cdots$ are zeros of the Airy function ${\rm Ai}(w)$. Under the above normalization,  $ S(w)$ has square root behavior: $ S(w)\sim \sqrt w$ as $w\rightarrow \infty$. $ S(w)$ can be viewed as a random meromorphic function, with poles given by the Airy$_1$ point process.

Formally, taking the soft-edge scaling limit of the loop equations \eqref{e:loop0} and \eqref{e:loop1}, we obtain a hierarchy of equations for the Stieltjes transform $ S(w)$ of the Airy$_1$ point process. The rank-$p$ loop equations are given by: for any $w,w_1,\ldots,w_{p-1}\in\bC\setminus\bR$,
\begin{align}\begin{split}\label{e:Airy}
&\phantom{{}={}}\bE\left[\left({( S(w)-\sqrt{w})^2+ 2\sqrt{w}( S(w)-\sqrt w) +\del_w  S(w)}\right)\prod_{i=1}^{p-1} S(w_i)\right]\\
&+2\sum_{1\leq i\leq p-1}\bE\left[\del_{w_i}\left( \frac{ S(w)- S(w_i)}{w-w_i}\right)\prod_{j:j\neq i}  S(w_j)\right]=0.
\end{split}\end{align}
These loop equations \eqref{e:Airy} provide necessary conditions for the edge statistics to converge to the Airy point process.

A key intermediate step in \cite{huang2024ramanujan} toward the proof of \Cref{t:universality} is the following microscopic loop equation for the normalized adjacency matrix of a random $d$-regular graph perturbed by a small GOE. Its proof occupies the main part of \cite{huang2024ramanujan} and is analogous to the derivation of the self-consistent equations \eqref{e:self_consistent}. To make the error terms sufficiently small, we need to iterate the local resampling procedure many times and carefully track the accumulated errors.

\begin{proposition}\label{p:loop}
Fix a small $\fc>0$. Let $\Delta_t(z)\deq m_N(z;t)-m(z;t)$, and recall the microscopic window $\mathbf M$ from \eqref{e:micro_window}. Fix $p\ge 1$ and $w,w_1,\ldots,w_{p-1}\in\mathbf M$, and set $z(t)\deq E_t+w$ and $z_j(t)\deq E_t+w_j$ for $1\le j\le p-1$. Then for $N$ sufficiently large,
\begin{align}\begin{split}\label{e:micro_loop}
&\phantom{{}={}}\bE\!\left[\Big(\Delta_t\!\big(z(t)\big)^2 + 2\cA \sqrt{z(t)-E_t}\;\Delta_t\!\big(z(t)\big) + \frac{\partial_z m_N\!\big(z(t);t\big)}{N}\Big)\prod_{i=1}^{p-1}\Delta_t\!\big(z_i(t)\big)\right] \\
&\quad +\;\frac{2}{N^2}\sum_{1\le i\le p-1}\bE\!\left[\partial_{z}\left.\left(\frac{m_N\!\big(z;t\big)-m_N\!\big(z;t\big)}{z-z_i(t)}\right)\right|_{z=z_i(t)}\prod_{j: 1\leq j\neq i\leq p-1}\Delta_t\!\big(z_j(t)\big)\right]=\OO\left(N^{-(p+1)/3-10\fc}\right). 
\end{split}\end{align}
\end{proposition}

The {scaling limit $N^{1/3}\Delta_t(E_t+w/(\cA N)^{2/3})/\cA^{2/3}$} converges to $ S(w)-\sqrt w$ in the microscopic version of loop equations \eqref{e:Airy}. The error bound on the right-hand side of \eqref{e:micro_loop} is precisely calibrated so that, after scaling each factor of $\Delta_t$ by $N^{1/3}$, the total error still vanishes as $N\to\infty$.

In what follows, we show that the microscopic loop equations \eqref{e:micro_loop} imply the Green's function comparison in \Cref{p:small_change}, thereby completing the final step of the three-step strategy. For clarity, we prove \eqref{e:product_s} only in the case \(p=1\), writing $z(t)=E_t+w$ with $w\in \bf M$. We refer to \cite[Section 3.3]{huang2024ramanujan} for the general case.  In this case, \eqref{e:product_s} follows from integrating the following statement from $0$ to $t=N^{-1/3+\fc}$
   \begin{align}\label{e:small_change}
    | \del_t \bE\left[m_N(z(t);t)-m(z(t);t)\right]|\leq N^{-2\fc}.
 \end{align}

 The first loop equation corresponds to $p=1$ in \eqref{e:micro_loop}
\begin{align}\label{e:firsteqa}
\bE\left[(\Delta_t(z(t)) ^2+ 2\cA \sqrt{z(t)-2}\Delta_t(z(t)) +\frac{\del_z m_N(z(t);t)}{N}\right]=\OO\left( N^{-2/3-10\fc}\right),
\end{align}
where $z(t)=E_t+w$ with $w\in \bf M$. Next we show that \eqref{e:firsteqa} implies \eqref{e:small_change}.
By Gaussian integration by parts,
\begin{align}\label{e:eq1}
  \partial_t\,\bE\big[m_N(z;t)\big]
  \;=\; \frac{1}{N}\,\partial_t\,\bE\!\left[\Tr\,(H+\sqrt{t}\,Z-z)^{-1}\right]
  \;=\; \frac{1}{2}\,\bE\!\left[\partial_z\!\left(m^2_N(z;t)+\frac{\partial_z m_N(z;t)}{N}\right)\right].
\end{align}
Subtracting the complex Burgers equation \eqref{e:eq2} for $s(z;t)$ from \eqref{e:eq1} gives
\begin{align}\label{e:DBM_mt}
  \partial_t\,\bE\big[m_N(z;t)-m(z;t)\big]
  \;=\; \frac{1}{2}\,\bE\!\left[\partial_z\!\left(m^2_N(z;t)-m^2(z;t)+\frac{\partial_z m_N(z;t)}{N}\right)\right].
\end{align}

Plugging $z(t)=E_t+w$ with $w\in\mathbf M$ into \eqref{e:DBM_mt}, we obtain
\begin{align}\begin{split}\label{e:DBM_mt2}
  &\partial_t\,\bE\big[m_N(z(t);t)-m(z(t);t)\big]
  \;=\; \frac{1}{2}\,\bE\big[\partial_z F_t\big(z(t)\big)\big],\\
  &F_t(z)\;:=\;\Delta_t(z)^2 + 2\big(m(z;t)-m(E_t;t)\big)\Delta_t(z) + \frac{\partial_z m_N(z;t)}{N}.
\end{split}\end{align}
For $z$ near the spectral edge,  with $z-E_t\in\mathbf M$,  and using the square-root behavior of $\varrho_t$, we have
$
  m(z;t)-m(E_t;t)\approx \cA\,\sqrt{\,z-E_t\,}$.
Thus, neglecting the replacement error, the first loop equation \eqref{e:firsteqa} implies
\begin{align}\label{e:F_t}
  \bE\big[F_t(z)\big]=\OO\left( N^{-2/3-10\fc}\right).
\end{align}

Since $F_t$ is analytic on $\bC\setminus\bR$, Cauchy’s integral formula on a circle $\omega$ centered at $z(t)$ with radius $\Im z(t)/2\geq N^{-2/3-\fc}/2$ yields
\begin{align}\label{e:F_t2}
  \bE\big[\partial_z \left.F_t(z)\right|_{z=z(t)}\big]=\frac{1}{2\pi\ri}\oint_\omega \frac{\bE\big[F_t(w)\big]}{(w-z(t))^2}\rd w =\OO\left( N^{-8\fc}\right).
\end{align}
Combining \eqref{e:DBM_mt2} with \eqref{e:F_t2} proves \eqref{e:small_change}, and hence the Green's function comparison \Cref{p:small_change}.

\subsection{Self-consistent equation and the first loop equation}

To illustrate the main ideas, in these notes we present the proof of
the following \Cref{t:recursion}, which yields the self-consistent equations and the first
microscopic version of the loop equations at the edge. We refer to
\cite{huang2024ramanujan} for the complete proof of edge universality. 

We recall the functions \(Y_\ell, X_\ell\) and the quantity \(Q(z)\) from
\eqref{def:Y} and \eqref{e:Qsum}. For simplicity of notation, we write
\begin{align}\label{e:defYt}
 Y(z)=Y_\ell(Q(z), z),\quad X(z)=X_\ell(Q(z), z).
\end{align}
We also recall the sets  \(\overline{\Omega}\) and \(\Omega\) of \(d\)-regular
graphs from \Cref{def:omegabar} and \Cref{thm:prevthm0}. 

Next, we introduce some further error terms. We emphasize that they depend on
the graph \(\cG\) and are therefore random quantities.

\begin{definition}\label{def:phidef}
For any \(z\in \bC^+\) in the upper half–plane with \(\eta=\Im z\), we introduce
the control parameter 
\begin{align}\begin{split}\label{e:defPhi}
 \Phi(z)&:=\frac{\Im[m_N(z)]}{N\eta}+\frac{1}{N^{1-2\fc}},\quad 
 \Psi(z):= \bm1(\cG\in \Omega)
\left(\Phi(z)+\frac{|Q(z)-Y(z)|}{N^{\fb/8}}\right).
\end{split}\end{align}
\end{definition}

\begin{remark}
The quantity \(\Phi(z)\) is a standard control parameter in random matrix
theory. We will control the expectation of \(Q(z)-Y(z)\) via \(\Psi(z)\), see
\eqref{e:QY}. Although \(\Psi(z)\) itself contains \(|Q(z)-Y(z)|\) (but with an
extra factor \(N^{-\fb/8}\)), we can (ignoring the absolute value) move this
term to the left-hand side and conclude that the expectation of \(Q(z)-Y(z)\) is
controlled by \(\Phi(z)\).
\end{remark}

The following result states that the expectation of the self-consistent
equations is negligible, and it gives the first loop equation corresponding to
the case \(t=0\), \(p=1\) of \Cref{p:loop}. The proof of \Cref{t:recursion}
will be given in the remainder of this article.

\begin{theorem}\label{t:recursion}
Fix a spectral parameter \(z\in \bf D\) (recall from \eqref{e:D}), and recall the control parameter
\(\Psi(z)\) from \eqref{e:defPhi}. Then
\begin{align}
\label{e:QY}&\bE\left[\bm1(\cG\in \Omega)\bigl(Q(z)-Y(z)\bigr)\right]
   \lesssim (d-1)^{2\ell}\bE[\Psi(z)],\\
\label{e:mz} &\bE\left[\bm1(\cG\in \Omega)\bigl(m_N(z)-X(z)\bigr)\right]
   \lesssim (d-1)^{2\ell}\bE[\Psi(z)].
\end{align}
Moreover, if we further assume that \(|z-2|\leq N^{-\fg}\), then we have a
refined estimate for \(Q(z)-Y(z)\):
\begin{align}\begin{split}\label{e:Qrefined_bound}
  & \phantom{{}={}}\bE\left[\bm1(\cG\in \Omega)\left(
  \frac{\cA^2}{\ell+1}\bigl(Q(z)-Y(z)\bigr)+\frac{\partial_z m_N(z)}{N}
  \right)\right]
  =\OO\left(\frac{N^\fo\bE[\Psi(z)]}{(d-1)^{\ell/2}}\right),
\end{split}\end{align}
where the constant \(\cA=d(d-1)/(d-2)^2\) is as in \eqref{e:edge_behavior}. If
\(|z+2|\leq N^{-\fg}\), an analogous statement holds after multiplying the first
term in \eqref{e:Qrefined_bound} by \(-1\).
\end{theorem}

Close to the right spectral edge, i.e.\ for
\(|z-2|\lesssim N^{-2/3+\oo(1)}\) with \(\Im z\gtrsim N^{-2/3+\oo(1)}\), we
have
\begin{align}\begin{split}\label{e:mscmd0}
\msc(z )&=-1+\OO(\sqrt{|z -2|}), \quad 1-\msc^2(z)=2\sqrt{z-2}+\OO(|z-2|), \\
 \md(z )&=-\frac{d-1}{d-2}+\OO(\sqrt{|z -2|}),
\end{split}\end{align}
and the true sizes of the error terms are
\begin{align}\begin{split}\label{e:edge}
  \Im[m_N(z)]\lesssim \frac{N^{\oo(1)}}{N^{1/3}}, \quad  
  \Phi(z), \Psi(z)\lesssim \frac{N^{\oo(1)}}{N^{2/3}}.
\end{split}\end{align}
The optimal edge rigidity of eigenvalues in \cite{huang2024ramanujan} follows
from establishing analogous high-moment estimates to \eqref{e:QY} and
\eqref{e:mz}, which essentially lead to 
\begin{align}\label{e:QYmX}
|Q(z)-Y(z)|\lesssim \frac{N^{\oo(1)}}{N^{2/3}}, \quad 
|m_N(z)-X(z)|\lesssim \frac{N^{\oo(1)}}{N^{1/2}}.
\end{align}
We remark that, for technical reasons, the estimate for \(|m_N(z)-X(z)|\) is
weaker. Nevertheless, these estimates are sufficient to derive optimal bounds
for \(Q(z)-\msc(z)\) and \(m_N(z)-\md(z)\), as explained below.

The concentration of \(Q(z)\) and \(m_N(z)\) can be derived from
\eqref{e:QYmX}. We also recall the expansion of \(Y(z)=Y_\ell(Q(z),z)\) from
\eqref{e:recurbound} (by taking \(\Delta=Q(z)\)):
\begin{align}\begin{split}\label{e:Q-YQ6}
&\phantom{{}={}}Q-Y=(Q-\msc(z))-(Y-\msc(z))=(1-\msc^{2(\ell+1)}(z))(Q-\msc(z))\\
&\quad -\msc^{2\ell+2}(z)\md(z)\left(\frac{1-\msc^{2\ell+2}(z)}{d-1}
+\frac{d-2}{d-1}\frac{1-\msc^{2\ell+2}(z)}{1-\msc^2(z)}\right)(Q-\msc(z))^2
 +\OO\bigl(\ell^5|Q-\msc(z)|^3\bigr)\\
&= 
(1-\msc^{2}(z ))(\ell+1)(Q -\msc(z ))
+(\ell+1)(Q -\msc(z ))^2\\
&\quad +\OO\Bigl(\ell\bigl(|z -2||Q -\msc(z )|
 +\sqrt{|z -2|}|Q -\msc(z )|^2\bigr)+\ell^5|Q -\msc(z )|^3\Bigr)\\
&= 
(\ell+1)\Bigl(2\sqrt{z-2}\,(Q -\msc(z ))+(Q -\msc(z ))^2\Bigr)\\
&\quad +\OO\Bigl(\ell\bigl(|z -2||Q -\msc(z )|
 +\sqrt{|z -2|}|Q -\msc(z )|^2\bigr)+\ell^5|Q -\msc(z )|^3\Bigr),
\end{split}\end{align}
where we used \eqref{e:mscmd0} to replace \(\msc(z), \md(z), 1-\msc^2(z)\) by
\(-1, -(d-1)/(d-2), 2\sqrt{z-2}\), respectively. Given that \(Q-Y\) is small as
in \eqref{e:QYmX}, we can analyze \eqref{e:Q-YQ6} as a quadratic equation in
\(Q-\msc(z)\) with a small error. Such quadratic equations are ubiquitous in
random matrix theory; for example, for Wigner matrices. Using \eqref{e:QYmX},  the standard argument
shows that for \(|z-2|\lesssim N^{-2/3+\oo(1)}\), with high probability we have
\begin{align}\label{e:Q-m}
|Q-\msc(z)|\lesssim\frac{N^{\oo(1)}}{N^{1/3}},
\end{align}
Using \eqref{e:Q-m} we can simplify the error in \eqref{e:Q-YQ6}, and get
\begin{align}\begin{split}\label{e:Q-YQ7}
Q-Y=
(\ell+1)\Bigl(2\sqrt{z-2}\,(Q -\msc(z ))+(Q -\msc(z ))^2\Bigr)+\OO(N^{-1+\oo(1)}),
\end{split}\end{align}

Using the expansion of \(X(z)=X_\ell(Q(z),z)\) from \eqref{e:Xrecurbound} (with
\(\Delta=Q(z)\)), we have
\begin{align*}
X-\md(z)
&=\frac{d}{d-1}\md^2(z)\,\msc^{2\ell}(z)\,(Q-\msc(z))
   +\OO\bigl(|Q-\msc(z)|^2\bigr)\\
&= \cA(Q-\msc(z)) 
   +\OO\bigl(\ell\sqrt{|z-2|}\,|Q-\msc(z)|+|Q-\msc(z)|^2\bigr),
\end{align*}
where we again used \eqref{e:mscmd0} to replace \(\msc(z), \md(z)\) by
\(-1, -(d-1)/(d-2)\). It follows, by rearranging, that
\begin{align}\begin{split}\label{e:mQrelation}
m_N(z) -\md(z) &=m_N(z) -X+X-\md(z)\\
&=\cA(Q -\msc(z )) 
 +\OO\bigl(|m_N(z) -X|+\ell\sqrt{|z-2|}|Q-\msc(z)|+|Q-\msc(z)|^2\bigr).
\end{split}\end{align}
Combining \eqref{e:QYmX}, \eqref{e:Q-m} and \eqref{e:mQrelation}, we conclude
that for \(|z-2|\lesssim N^{-2/3+\oo(1)}\), with high probability
\begin{align}\label{e:m-md}
m_N(z)-\md(z)=\cA(Q-\msc(z))+ \OO\left(\frac{N^{\oo(1)}}{N^{1/2}}\right)
 =\OO\left(\frac{N^{\oo(1)}}{N^{1/3}}\right).
\end{align}

We remark that \eqref{e:Qrefined_bound} is slightly different from the loop
equation \eqref{e:micro_loop} (with \(t=0\) and \(p=1\)):
\begin{align}\begin{split}\label{e:micro_loop2}
&\phantom{{}={}}\bE\!\left[(m_N(z)-\md(z))^2 + 2\cA \sqrt{z-2}\,(m_N(z)-\md(z)) 
+ \frac{\partial_z m_N(z)}{N}\right]
 =\OO\left(N^{-2/3-10\fc}\right). 
\end{split}\end{align}

By plugging \eqref{e:m-md} into \eqref{e:Q-YQ7}, we conclude that for
\(|z-2|\lesssim N^{-2/3+\oo(1)}\), with high probability 
\begin{align}\label{e:AQ-Y}
    \frac{\cA^2(Q-Y)}{\ell+1}
    =(m_N(z)-\md(z))^2
    +2\cA \sqrt{z-2}\,(m_N(z)-\md(z))
    +\OO\left(\frac{N^{\oo(1)}}{N^{5/6}}\right).
\end{align}
The claim \eqref{e:micro_loop2} then follows from plugging \eqref{e:AQ-Y} into \eqref{e:Qrefined_bound}, and
the error term comes from
\begin{align*}
\frac{N^\fo\bE[\Psi(z)]}{(d-1)^{\ell/2}}
\lesssim \frac{N^{\fo+\oo(1)}}{(d-1)^{\ell/2}N^{2/3}}
\lesssim N^{-2/3-10\fc},
\end{align*}
where we used \eqref{e:edge}.

\section{Local resampling}  
\label{s:local_resampling}

In this section, we recall the local resampling and its properties. This gives us the framework to talk about resampling from the random regular graph distribution as a way to get an improvement in our estimates of the Green's function.

For any graph $\cG$, we denote the set of unoriented edges by $E(\cG)$,
and the set of oriented edges by $\vec{E}(\cG):=\{(u,v),(v,u):\{u,v\}\in E(\cG)\}$.
For a subset $\vec{S}\subset \vec{E}(\cG)$, we denote by $S$ the set of corresponding non-oriented edges.
For a subset $S\subset E(\cG)$ of edges we denote by $[S] \subset \qq{N}$ the set of vertices incident to any edge in $S$.
Moreover, for a subset $\bV\subset\qq{N}$ of vertices, we define $E(\cG)|_{\bV}$ to be the subgraph of $\cG$ induced by $\bV$.

\begin{figure}[t]
  \centering

  \begin{subfigure}{0.48\textwidth}
    \centering
    \resizebox{\linewidth}{!}{
    \begin{tikzpicture}[
  every node/.style={circle,draw,inner sep=1.5pt},
  >=stealth
]

\fill[blue!10] (0,0) ellipse (11cm and 9cm);

\node[draw=none] at (-2.5,8) {\Huge $\mathcal G$};

\node (r) at (0,0) {};
\node[draw=none, below] at (r) {\Large $o$};

\draw[dashed] (0,0) circle (4cm);

\node[draw=none] at (3,2) {\Huge $\mathcal T$};

\node (a1) at (90:1.8)  {};
\node[draw=none, right] at (a1) {\Large $i$};

\node (a2) at (210:1.8) {};
\node (a3) at (330:1.8) {};

\draw (r) -- (a1);
\draw (r) -- (a2);
\draw (r) -- (a3);


\node (a1b1) at (70:3.2)  {};
\node (a1b2) at (110:3.2) {};
\draw (a1) -- (a1b1);
\draw (a1) -- (a1b2);
\node[draw=none, left] at (a1b2) {\Large $l_\mu$};

\node (a2b1) at (190:3.2) {};
\node (a2b2) at (230:3.2) {};
\draw (a2) -- (a2b1);
\draw (a2) -- (a2b2);
\node[draw=none, yshift=15pt] at (a2b1) {\Large $l_1=l_2$};
\node[draw=none, right] at (a2b2) {\Large $l_3=l_4$};

\node (a3b1) at (310:3.2) {};
\node (a3b2) at (350:3.2) {};
\draw (a3) -- (a3b1);
\draw (a3) -- (a3b2);


\node (a1b1c1) at (60:4.6) {};
\node (a1b1c2) at (80:4.6) {};
\draw[blue] (a1b1) -- (a1b1c1);
\draw[blue] (a1b1) -- (a1b1c2);

\node (a1b2c1) at (100:4.6) {};
\node (a1b2c2) at (120:4.6) {};
\draw[blue] (a1b2) -- (a1b2c1);
\draw[blue] (a1b2) -- (a1b2c2);
\node[draw=none, left] at (a1b2c2) {\Large $a_\mu$};

\node (a2b1c1) at (180:4.6) {};
\node (a2b1c2) at (200:4.6) {};
\draw[blue] (a2b1) -- (a2b1c1);
\draw[blue] (a2b1) -- (a2b1c2);
\node[draw=none, left] at (a2b1c1) {\Large $a_1$};
\node[draw=none, left] at (a2b1c2) {\Large $a_2$};

\node (a2b2c1) at (220:4.6) {};
\node (a2b2c2) at (240:4.6) {};
\draw[blue] (a2b2) -- (a2b2c1);
\draw[blue] (a2b2) -- (a2b2c2);
\node[draw=none, left] at (a2b2c1) {\Large $a_3$};
\node[draw=none, left] at (a2b2c2) {\Large $a_4$};

\node (a3b1c1) at (300:4.6) {};
\node (a3b1c2) at (320:4.6) {};
\draw[blue] (a3b1) -- (a3b1c1);
\draw[blue] (a3b1) -- (a3b1c2);

\node (a3b2c1) at (340:4.6) {};
\node (a3b2c2) at (0:4.6)   {};
\draw[blue] (a3b2) -- (a3b2c1);
\draw[blue] (a3b2) -- (a3b2c2);

\foreach \a in {0,30,...,330} {
  \coordinate (P\a) at ({9*cos(\a)},{7*sin(\a)});
  \coordinate (Q\a) at ({9*cos(\a+10)},{7*sin(\a+10)});
  \node (na\a) at (P\a) {};
  \node (nb\a) at (Q\a) {};
  \draw[blue] (na\a) -- (nb\a);
}

\node[draw=none, left] at (P150) {\Large $b_\mu$};
\node[draw=none, left] at (Q150) {\Large $c_\mu$};

\node[draw=none, left] at (P180) {\Large $b_1$};
\node[draw=none, left] at (Q180) {\Large $c_1$};

\node[draw=none, left] at (P210) {\Large $b_2$};
\node[draw=none, left] at (Q210) {\Large $c_2$};

\node[draw=none, below] at (P240) {\Large $b_3$};
\node[draw=none, below] at (Q240) {\Large $c_3$};

\node[draw=none, below] at (P270) {\Large $b_4$};
\node[draw=none, below] at (Q270) {\Large $c_4$};

\end{tikzpicture}
}
  \end{subfigure}
  \hfill
  \begin{subfigure}{0.48\textwidth}
    \centering
    \resizebox{\linewidth}{!}{
       \begin{tikzpicture}[
  every node/.style={circle,draw,inner sep=1.5pt},
  >=stealth
]

\fill[blue!10] (0,0) ellipse (11cm and 9cm);

\node[draw=none] at (-2.5,8) {\Huge $\wt{\mathcal G}$};

\node (r) at (0,0) {};
\node[draw=none, below] at (r) {\Large $o$};

\draw[dashed] (0,0) circle (4cm);

\node[draw=none] at (3,2) {\Huge $\mathcal T$};

\node (a1) at (90:1.8)  {};
\node[draw=none, right] at (a1) {\Large $i$};

\node (a2) at (210:1.8) {};
\node (a3) at (330:1.8) {};

\draw (r) -- (a1);
\draw (r) -- (a2);
\draw (r) -- (a3);


\node (a1b1) at (70:3.2)  {};
\node (a1b2) at (110:3.2) {};
\draw (a1) -- (a1b1);
\draw (a1) -- (a1b2);
\node[draw=none, right] at (a1b2) {\Large $l_\mu$};

\node (a2b1) at (190:3.2) {};
\node (a2b2) at (230:3.2) {};
\draw (a2) -- (a2b1);
\draw (a2) -- (a2b2);
\node[draw=none, yshift=15pt] at (a2b1) {\Large $l_1=l_2$};
\node[draw=none, right] at (a2b2) {\Large $l_3=l_4$};

\node (a3b1) at (310:3.2) {};
\node (a3b2) at (350:3.2) {};
\draw (a3) -- (a3b1);
\draw (a3) -- (a3b2);

\foreach \a in {0,30,...,330} {
  \coordinate (P\a) at ({9*cos(\a)},{7*sin(\a)});
  \coordinate (Q\a) at ({9*cos(\a+10)},{7*sin(\a+10)});
  \node (na\a) at (P\a) {};
  \node (nb\a) at (Q\a) {};
}

\node[draw=none, left] at (P150) {\Large $b_\mu$};
\node[draw=none, left] at (Q150) {\Large $c_\mu$};

\node[draw=none, left] at (P180) {\Large $b_1$};
\node[draw=none, left] at (Q180) {\Large $c_1$};

\node[draw=none, left] at (P210) {\Large $b_2$};
\node[draw=none, left] at (Q210) {\Large $c_2$};

\node[draw=none, below] at (P240) {\Large $b_3$};
\node[draw=none, below] at (Q240) {\Large $c_3$};

\node[draw=none, below] at (P270) {\Large $b_4$};
\node[draw=none, below] at (Q270) {\Large $c_4$};


\node (a1b1c1) at (60:4.6) {};
\node (a1b1c2) at (80:4.6) {};

\draw[red] (Q60) -- (a1b1);
\draw[red] (P60) -- (a1b1c1);

\draw[red] (Q90) -- (a1b1);
\draw[red] (P90) -- (a1b1c2);

\node (a1b2c1) at (100:4.6) {};
\node (a1b2c2) at (120:4.6) {};

\draw[red] (Q120) -- (a1b2);
\draw[red] (P120) -- (a1b2c1);

\draw[red] (Q150) -- (a1b2);
\draw[red] (P150) -- (a1b2c2);

\node[draw=none, right] at (a1b2c2) {\Large $a_\mu$};

\node (a2b1c1) at (180:4.6) {};
\node (a2b1c2) at (200:4.6) {};

\draw[red] (Q180) -- (a2b1);
\draw[red] (P180) -- (a2b1c1);

\draw[red] (Q210) -- (a2b1);
\draw[red] (P210) -- (a2b1c2);

\node[draw=none, above] at (a2b1c1) {\Large $a_1$};
\node[draw=none, above] at (a2b1c2) {\Large $a_2$};

\node (a2b2c1) at (220:4.6) {};
\node (a2b2c2) at (240:4.6) {};

\draw[red] (Q240) -- (a2b2);
\draw[red] (P240) -- (a2b2c1);

\draw[red] (Q270) -- (a2b2);
\draw[red] (P270) -- (a2b2c2);

\node[draw=none, right] at (a2b2c1) {\Large $a_3$};
\node[draw=none, right] at (a2b2c2) {\Large $a_4$};

\node (a3b1c1) at (300:4.6) {};
\node (a3b1c2) at (320:4.6) {};

\draw[red] (Q300) -- (a3b1);
\draw[red] (P300) -- (a3b1c1);

\draw[red] (Q330) -- (a3b1);
\draw[red] (P330) -- (a3b1c2);

\node (a3b2c1) at (340:4.6) {};
\node (a3b2c2) at (0:4.6)   {};

\draw[red] (Q0) -- (a3b2);
\draw[red] (P0) -- (a3b2c1);

\draw[red] (Q30) -- (a3b2);
\draw[red] (P30) -- (a3b2c2);

\end{tikzpicture}}

  \end{subfigure}

   \caption{
    An example of the local resampling performed on the graph is as follows. We replace the blue edges, located on the boundary of the radius-$\ell$ neighborhood of a vertex $o$, with randomly chosen edges from the graph. Together, these edges constitute the resampling data, denoted by $\mathbf{S}$. This operation creates new red edges and establishes a new boundary.}
    \label{fig:switchingproc}
\end{figure}
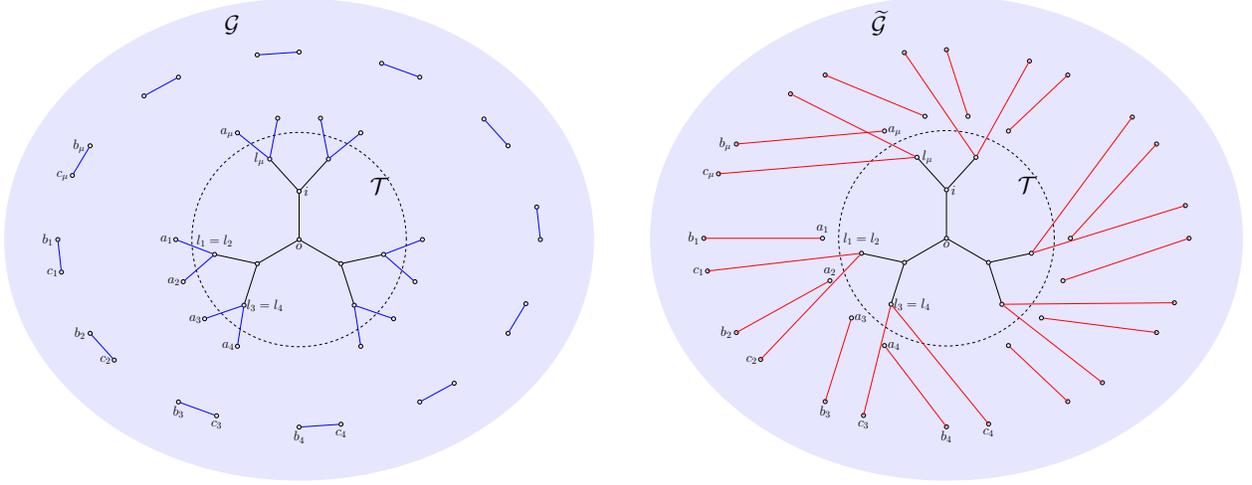

\begin{definition}
A (simple) switching is encoded by two oriented edges $\vec S=\{(v_1, v_2), (v_3, v_4)\} \subset \vec{E}$.
We assume that the two edges are disjoint, i.e.\ that $|\{v_1,v_2,v_3,v_4\}|=4$.
Then the switching consists of
replacing the edges $\{v_1,v_2\}, \{v_3,v_4\}$ with the edges $\{v_1,v_4\},\{v_2,v_3\}$.
We denote the graph after the switching $\vec S$ by $T_{\vec S}(\cG)$,
and the new edges $\vec S' = \{(v_1,v_4), (v_2,v_3)\}$ by
$
  T(\vec S) = \vec S'
$.
\end{definition}

The local resampling involves a fixed center vertex, which we now assume to be vertex $o$,
and a radius $\ell$.
Given a $d$-regular graph $\cG$, we write $\cT\deq\cB_{\ell}(o,\cG)$ to denote the radius-$\ell$ neighborhood of $o$ (which may not necessarily be a tree) and write $\bT$ for its vertex set.\index{$\cT, \bT$}
The edge boundary $\del_E \cT$ of $\cT$ consists of the edges in $\cG$ with one vertex in $\bT$ and the other vertex in $\qq{N}\setminus\bT$.
We enumerate the edges of $\del_E \cT$ as $ \del_E \cT = \{e_1,e_2,\dots, e_\mu\}$, where $e_\al=\{l_\al, a_\al\}$ with $l_\al \in \bT$ and $a_\al \in \qq{N} \setminus \bT$. We orient the edges $e_\al$ by defining $\vec{e}_\al=(l_\al, a_\al)$.
We notice that $\mu$ and the edges $e_1,e_2, \dots, e_\mu$ depend on $\cG$. The edges $e_\al$ are distinct, but
the vertices $a_\al$ are not necessarily distinct and neither are the vertices $l_\al$. Our local resampling switches the edge boundary of $\cT$ with randomly chosen edges in $\cG^{(\bT)}$
if the switching is admissible (see below), and leaves them in place otherwise.
To perform our local resampling, see \Cref{fig:switchingproc}, we choose $(b_1,c_1), \dots, (b_\mu,c_\mu)$ to be independent, uniformly chosen oriented edges from the graph $\cG^{(\bT)}$, i.e.,
the oriented edges of $\cG$ that are not incident to $\bT$,
and define 
\begin{equation}\label{e:defSa}
  \vec{S}_\al= \{\vec{e}_\al, (b_\al,c_\al)\},
  \qquad
  {\bf S}=(\vec S_1, \vec S_2,\dots, \vec S_\mu).
\end{equation}
The sets $\bf S$ will be called the \emph{resampling data} for $\cG$. We remark that repetitions are allowed in the data $(b_1, c_1), (b_2, c_2),\cdots, (b_\mu, c_\mu)$.
We define an indicator that will be crucial to the definition of the switch.

\begin{definition}
For $\al\in\qq{\mu}$,
we define the indicator functions
$I_\al \equiv I_\al(\cG,{\bf S})=1$\index{$I_\alpha$}  if
\begin{enumerate}
\item
 the subgraph $\cB_{\fR/4}(\{a_\al, b_\al, c_\al\}, \cG^{(\bT)})$ after adding the edge $\{a_\al, b_\al\}$ is a tree;
\item 
and $\dist_{\cG^{(\bT)}}(\{a_\al,b_\al,c_\al\}, \{a_\beta,b_\beta,c_\beta\})> {\fR/4}$ for all $\beta\in \qq{\mu}\setminus \{\al\}$.
\end{enumerate}
\end{definition}
 The indicator function $I_\alpha$ imposes two conditions. The first one is a ``tree" condition, which ensures that 
 $a_\al$ and $\{b_\al, c_\al\}$ are far away from each other, and their neighborhoods are trees. 
The second one imposes an ``isolation" condition, which ensures that we only perform simple switching when the switching pair is far away from other switching pairs. In this way, we do not need to keep track of the interaction between different simple switchings. 

We define the \emph{admissible set}
\begin{align}\label{Wdef}
{\mathsf W}_{\bf S}:=\{\al\in \qq{\mu}: I_\al(\cG,{\bf S}) \}.
\end{align}
We say that the index $\al \in \qq{\mu}$ is \emph{switchable} if $\al\in {\mathsf W}_{\bf S}$. We denote the set $\bW_{\bf S}=\{b_\al:\al\in {\mathsf W}_{\bf S}\}$\index{$\bW_{\bf S}$}. Let $\nu:=|{\mathsf W}_{\bf S}|$ be the number of admissible switchings and $\al_1,\al_2,\dots, \al_{\nu}$
be an arbitrary enumeration of ${\mathsf W}_{\bf S}$.
Then we define the switched graph by
\begin{equation} \label{e:Tdef1}
T_{\bf S}(\cG) := \left(T_{\vec S_{\al_1}}\circ \cdots \circ T_{\vec S_{\al_\nu}}\right)(\cG),
\end{equation}
and the resampling data by
\begin{equation} \label{e:Tdef2}
  T({\bf S}) := (T_1(\vec{S}_1), \dots, T_\mu(\vec{S}_\mu)),
  \quad
  T_\al(\vec{S}_\al) \deq
  \begin{cases}
    T(\vec{S}_\al) & (\al \in {\mathsf W}_{\bf S}),\\
    \vec{S}_\al & (\al \not\in {\mathsf W}_{\bf S}).
  \end{cases}
\end{equation}

To make the structure more clear, we introduce an enlarged probability space.
Equivalent to the definition above, the sets $\vec{S}_\al$ as defined in \eqref{e:defSa} are uniformly distributed over 
\begin{align*}
{\sf S}_{\al}(\cG)=\{\vec S\subset \vec{E}: \vec S=\{\vec e_\al, \vec e\}, \text{$\vec{e}$ is not incident to $\cT$}\},
\end{align*}
i.e., the set of pairs of oriented edges in $\vec{E}$ containing $\vec{e}_\al$ and another oriented edge in $\cG^{(\bT)}$.
Therefore ${\bf S}=(\vec S_1,\vec S_2,\dots, \vec S_\mu)$ is uniformly distributed over the set
${\sf S}(\cG)=\sf S_1(\cG)\times \cdots \times \sf S_\mu(\cG)$.

We introduce the following notation on the probability and expectation with respect to the randomness of the $\bfS\in \sf S(\cG)$.
\begin{definition}\label{def:PS}
    Given any $d$-regular graph $\cG$, we 
    denote $\bP_\bfS(\cdot)$ the uniform probability measure on ${\sf S}(\cG)$;
 and $\bE_\bfS[\cdot]$ the expectation  over the choice of $\bfS$ according to $\bP_\bfS$. 
\end{definition}

 The following claim from \cite[Lemma 7.3]{huang2024spectrum} states that this switch is invariant under the random regular graph distribution.

\begin{lemma}[{\cite[Lemma 7.3]{huang2024spectrum}}] \label{lem:exchangeablepair}
Fix $d\geq 3$. We recall the operator $T_\bfS$ from \eqref{e:Tdef1}. Let $\cG$ be a random $d$-regular graph  and $\bfS$ uniformly distributed over $\sfS(\cG)$, then the graph pair $(\cG, T_{\bf S}(\cG))$ forms an exchangeable pair:
\begin{align*}
(\cG, T_{\bf S}(\cG))\stackrel{law}{=}(T_{\bf S}(\cG), \cG).
\end{align*}
\end{lemma}

\subsection{Green's function of resampled graph}

Fix an edge $\{i,o\}\subset \cG$, we recall the resampling data ${\bf S}=\{(l_\al, a_\al), (b_\al, c_\al)\}_{\al\in\qq{\mu}}$ around $o$ from \Cref{s:local_resampling}, denote  $\wt \cG=T_\bfS \cG$. In the remainder of the paper, we denote by $\wt H$ the normalized adjacency matrix of $\tcG$. Then
\begin{align}\label{e:H-H}
   \wt H-H=-\sum_{\al\in\qq{\mu} }\xi_\al,\quad  \xi_\al:=\frac{1}{\sqrt{d-1}}\left(\Delta_{l_\al a_\al}
    +\Delta_{b_\al c_\al}
    -\Delta_{l_\al c_\al}-\Delta_{a_\al b_\al}\right).
\end{align}

Its Green's function and the Stieltjes transform of its empirical eigenvalue distribution are denoted as follows:
\begin{align}\label{e:tGtm}
    \wt G(z)=(\widetilde H-z)^{-1}, \quad \wt m_N(z)=\frac{1}{N}\Tr\wt G(z).
\end{align}

In the rest of this section we collect some basic estimates of the Green's functions of $ \wt G(z)$. Their proofs follow from analyzing Green's functions using the resolvent identity formula \eqref{e:resolv}.

\begin{lemma}
    \label{c:expectationbound}
Let $z\in \bC^+$ satisfy $|z|\leq 1/\fg$, and set $\eta:=\Im z\geq N^{-1+\fg}$.
Fix a $d$-regular graph $\cG\in \Omega$ (recall \Cref{def:omega}) and vertices
$\{i,o\}$ in $\cG$, and let $\cT=\cB_\ell(o,\cG)$ with vertex set $\bT$.
We denote the resampling data ${\bf S}=\{(l_\al, a_\al), (b_\al, c_\al)\}_{\al\in\qq{\mu}}$ around $o$ (recall from \Cref{s:local_resampling}), and  $\wt \cG=T_\bfS \cG$. Moreover, we also assume that $\wt \cG\in\Omega$.

Then, for any $\bX$ of the form
$\bX\in\{\emptyset, \bT\}$ and any $x,y\not\in \bX$, we have
\begin{align}\label{eq:local_law1}
\bigl|\wt G^{(\bX)}_{xy}(z)
- P^{(\bX)}_{xy}\bigl(\cB_{\fR/100}(\{x,y\}\cup \bX,\wt \cG),z,\msc(z)\bigr)\bigr|
\lesssim N^{-\fb}.
\end{align}
Moreover, the following holds
\begin{align}\label{e:tmmdiff}
    |\wt m_N(z)-m_N(z)|\lesssim \frac{(d-1)^\ell N^\fo \Im[m_N(z)]}{N\eta},
\end{align}
\end{lemma}

\begin{proof}
The first statement \eqref{eq:local_law1} follows from our assumption that $\wt \cG\in\Omega$ and \eqref{eq:local_law} (with $\bX=\bT$).

For the second statement, we recall the notations from \eqref{e:H-H}
  \begin{align}\begin{split}\label{e:tmmdiff2}
      &\phantom{{}={}}|\wt m_N(z)-m_N(z)|
      =\frac{1}{N}\Tr (\wt H-z)^{-1} - \frac{1}{N}\Tr (H-z)^{-1}
      =\frac{1}{N}\Tr (\wt H-z)^{-1} (H-\wt H) (H-z)^{-1}\\
      &=\frac{1}{N}\sum_{\al\in\qq{\mu}}\Tr (\wt H-z)^{-1} \xi_\al (H-z)^{-1}\\
      &\leq \frac{1}{\sqrt{d-1}N}\sum_{i=1}^N\sum_{\al\in\qq{\mu}}(|\wt G_{il_\al}|+|\wt G_{ia_\al}|+|\wt G_{ib_\al}|+|\wt G_{ic_\al}|)( |G_{a_\al i}|+|G_{b_\al i}|
      +|G_{c_\al i}|+|G_{l_\al i}|)\\
      &\lesssim 
      \frac{1}{\sqrt{d-1}N}\sum_{i=1}^N\sum_{\al\in\qq{\mu}}(|\wt G_{il_\al}|^2+|\wt G_{ia_\al}|^2+|\wt G_{ib_\al}|^2+|\wt G_{ic_\al}|^2
      +|G_{a_\al i}|^2+|G_{b_\al i}|^2
      +|G_{c_\al i}|^2+|G_{l_\al i}|^2)\\
      &\lesssim \frac{(d-1)^\ell N^\fo \Im[\wt m_N(z)]}{N\eta}+\frac{(d-1)^\ell N^\fo \Im[m_N(z)]}{N\eta}
  \end{split}\end{align}  
  where the first two statements is from resolvent identity \eqref{e:resolv}; the third and fourth statements follows from \eqref{e:H-H}; the fifth statement follows from Cauchy-Schwarz inequality; the sixth statement follows from the Ward identity \eqref{e:Ward}, \eqref{e:Gsize} and $\mu=\OO((d-1)^\ell)$.
Since $N\eta\geq N^\fg\gg (d-1)^\ell N^\fo$, the claim \eqref{e:tmmdiff} follows from rearranging \eqref{e:tmmdiff2}. 

\end{proof}

\section{Concentration via exchangeable pair}\label{s:exchangeable}
In this section, we explain the basic ideas of proving concentration using
exchangeable pairs, and discuss how to apply this approach in the setting of
random \(d\)-regular graphs.

\subsection{Sum of independent Bernoulli random variables}
Let $X_1,X_2, \dots,X_n$ be independent $\mathrm{Bernoulli}(p_i)$ random variables and define
\[
W = \sum_{i=1}^n X_i, \qquad \mu = \mathbb{E}W = \sum_{i=1}^n p_i.
\]
In this section, we explain how to use exchangeable pair to show that $W$ concentrates.

Choose $I \sim \mathrm{Unif}\{1,\dots,n\}$ independent of everything and let
$X_I, X_I' \sim \mathrm{Bernoulli}(p_I)$ be independent copies.
Set
\[
W' = W - X_I + X_I'.
\]
Then $(W,W')$ is an exchangeable pair and satisfies $|W'-W|\le 1$. Moreover,
\begin{align}\label{e:deff}
\mathbb{E}[W'-W \mid X_1,\dots,X_n]
  = \frac{1}{n}\sum_{i=1}^n (p_i - X_i)
  = -\frac{1}{n}(W - \mu)=:-f(W).
\end{align}

For any smooth $g$,
\[
\mathbb{E}[f(W)g(W)] =\bE[(W-W') g(W)]= \frac{1}{2}\mathbb{E}\!\big[(W-W')(g(W)-g(W'))\big].
\]
Take $g(w) =    f^{2p-1}(w)$ with $f(w)=(w-\mu)/n$ from \eqref{e:deff}. Then 
\begin{align}\begin{split}\label{e:moment}
\mathbb{E}[f^{2p}(W)]&= \frac{1}{2}\mathbb{E}\!\big[(W-W')(f^{2p-1}(W)-f^{2p-1}(W'))\big]\\
&\leq \frac{2p-1}{4}\mathbb{E}\!\big[(W-W')(f(W)-f(W')(f^{2p-2}(W)+f^{2p-2}(W'))\big]\\
&= \frac{2p-1}{2n}\mathbb{E}\!\big[(W-W')^2 f^{2p-2}(W)\big]
= \frac{2p-1}{2n}\bE[\mathbb{E}\!\big[(W-W')^2|W] f^{2p-2}(W)\big].
\end{split}\end{align}
We can then bound the inner expectation as
\begin{align}\label{e:2moment}
\mathbb{E}[(W-W')^2 \mid W]
          = \frac{1}{n}\sum_{i=1}^n \mathbb{E}[(X_i - X_i')^2 \mid W]\leq  \frac{1}{n}\sum_{i=1}^n \mathbb{E}[(X_i + X_i') \mid W]=\frac{W+\mu}{n}.
\end{align}
By plugging \eqref{e:2moment} into \eqref{e:moment}, we arrive at the following moment bound
\begin{align}\begin{split}\label{e:moment2}
\mathbb{E}[f^{2p}(W)]
\leq \frac{2p-1}{2n^2}\bE[(W+\mu) f^{2p-2}(W)\big]
= \frac{2p-1}{2n}\bE[f^{2p-1}(W)\big]
+\frac{(2p-1)\mu}{n^2}\bE[f^{2p-2}(W)\big].
\end{split}\end{align}
We can upper bound the odd moments by even moments using H{\" o}lder's inequality. Then recursively, we can bound the $2p$-th moment as
\begin{align*}
\bE\left[\left(\frac{\sum_i (X_i-p_i)}{n}\right)^{2p}\right]=\mathbb{E}[f^{2p}(W)]\leq \frac{C_p\mu^p}{n^{2p}}
\end{align*}

\subsection{Random d-regular graphs}
In our setting, we choose a directed edge $(i,o)$ uniformly at random, independently of everything else, and let $\wt{\cG}$ be the graph obtained by locally resampling the neighborhood of the vertex $o$. Then $(\cG,\tcG)$ form an exchangeable pair. In next section we will show that
\begin{align*}
\bE\!\left[\widetilde G_{oo}^{(i)}\,\big|\,\cG\right]
=\bE_{\mathbf S}\!\left[\widetilde G_{oo}^{(i)}\right]
= Y+\cE_{(i,o)}(\cG),
\end{align*}
where we recall $Y$ from \eqref{e:defYt}, and $\cE_{(i,o)}(\cG)$ is small (here $\bE_{\mathbf S}$ denotes expectation over the resampling randomness). This implies
\begin{align*}
\bE\!\left[\wt Q- Q\,\big|\,\cG\right]
= \frac{1}{Nd} \sum_{i\sim o}\bE_{\mathbf S}\!\left[\widetilde G_{oo}^{(i)}\right]-Q
= Y-Q+\cE(\cG),
\qquad 
\cE(\cG):=\frac{1}{Nd}\sum_{i\sim o}\cE_{(i,o)}(\cG).
\end{align*}

An argument analogous to the exchangeable-pairs proof of concentration for sums of Bernoulli random variables yields high-moment bounds for $\wt Q-Q$, and hence concentration of $\wt Q-Q$. Although the individual errors $\cE_{(i,o)}(\cG)$ are small—by \eqref{eq:infbound} and \Cref{l:basicG} one obtains bounds of order $\OO(N^{-\fb})$—these are not sufficient to obtain the optimal eigenvalue concentration.  For example in \Cref{t:recursion}, near the spectral edge $|z\pm 2|=N^{-2/3+\oo(1)}$, the optimal error is $N^{-2/3+\oo(1)}$. To attain this optimal rate, we perform additional local resampling steps to capture and control the fluctuations of the error terms $\cE_{(i,o)}(\cG)$ more precisely. This is outlined in \Cref{s:forest} and \Cref{s:proofoutline}.

\section{Expectation of $Q(z)$.}\label{s:expQ}

In this section we take the first step computing the expectation of $Q(z)-Y(z)$.
\begin{align}\begin{split}\label{e:maint0}
\bE\left[{\bm1(\cG\in \Omega)}(Q-Y)\right]
&=\frac{1}{Nd}\sum_{o,i} \bE\left[A_{oi}{\bm1(\cG\in \Omega)}(G_{oo}^{(i)}-Y)\right].
\end{split}\end{align}
For $\cG\in \oOmega$ from \Cref{def:omegabar}, the number of vertices that do not have a tree neighborhood of radius $\fR$ is at most $N^\fc$. We can restrict the righthand side of \eqref{e:maint0} to the sum over vertices $o$ which has radius $\fR$ tree neighborhood: 
\begin{align}\begin{split}\label{e:maint1}
&\frac{1}{Nd}\sum_{o,i} \bE\left[A_{oi}{\bm1(\cG\in \Omega)}(G_{oo}^{(i)}-Y)\right] =\frac{1}{Nd}\sum_{o,i} \bE\left[ I(\{o,i\},\cG){\bm1(\cG\in \Omega)}(G_{oo}^{(i)}-Y)\right] 
+\OO\left(\frac{1}{N^{1-\fc}}\right),
\end{split}\end{align}
where 
\begin{align}\label{e:defineI}
   I(\{o,i\},\cG)= A_{oi}\bm1(\cB_{\fR}(o, \cG) \text{ is a tree}).
\end{align}

To compute the righthand side of \eqref{e:maint0}, condition on $I(\{o,i\},\cG)=1$ we perform a local resampling around $(i, o) \in \cF$ using the resampling data ${\bf S}=\{(l_\al, a_\al), (b_\al, c_\al)\}_{\al\in\qq{\mu}}$, where $\mu := d(d-1)^{\ell}$ is the number of boundary edges of the $d$-regular tree truncated at depth $\ell+1$. We denote the new graph as $\widetilde \cG = T_\bfS(\cG)$, with its corresponding Green's function $\widetilde G$. 

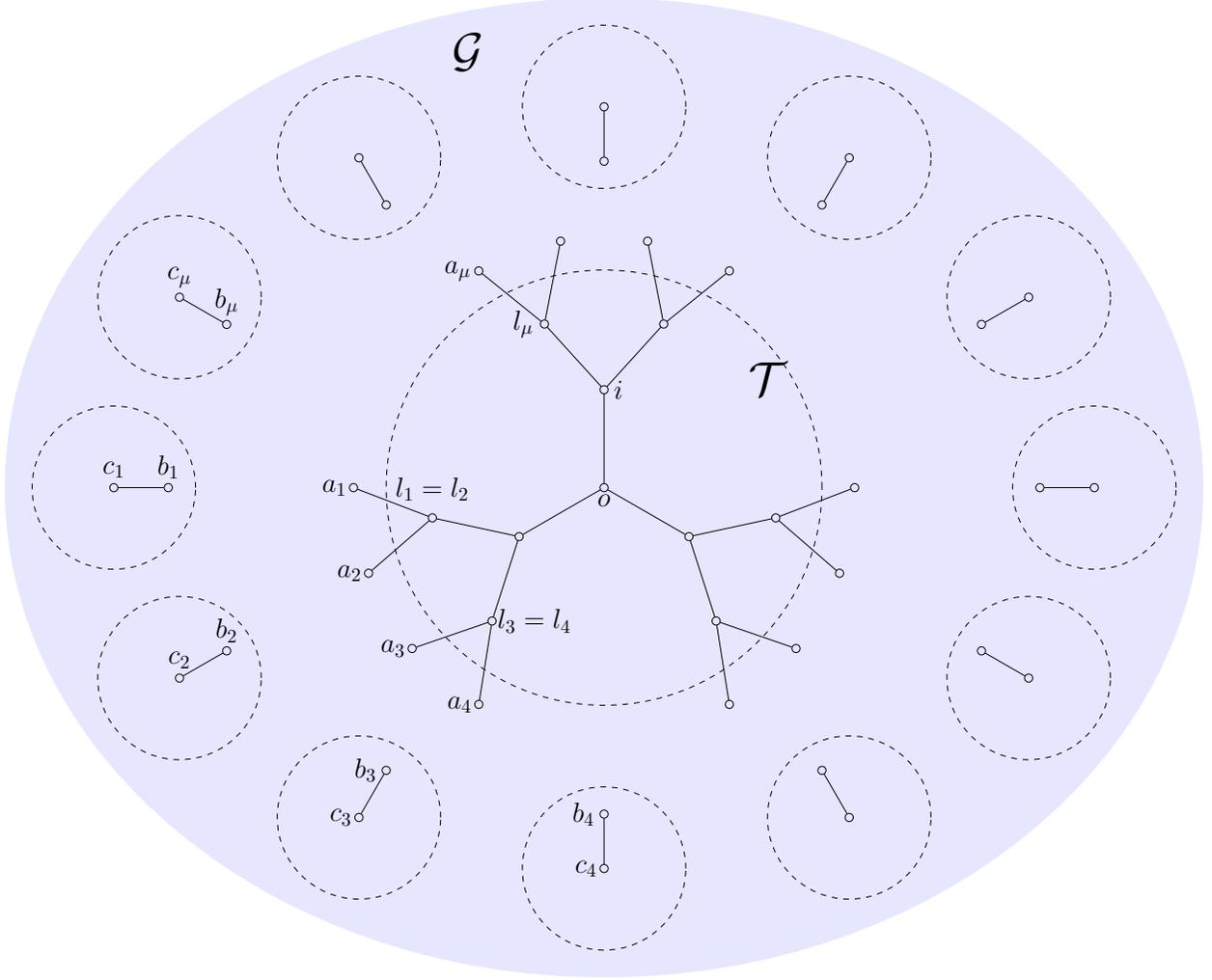
\begin{figure}[t]
\centering
\resizebox{\textwidth}{!}{%
\begin{tikzpicture}[
  every node/.style={circle,draw,inner sep=1.5pt},
  >=stealth
]

\fill[blue!10] (0,0) ellipse (11cm and 9cm);

\node[draw=none] at (-2.5,8) {\Huge $\mathcal G$};

\node (r) at (0,0) {};
\node[draw=none, below] at (r) {\Large $o$};

\draw[dashed] (0,0) circle (4cm);

\node[draw=none] at (3,2) {\Huge $\mathcal T$};

\node (a1) at (90:1.8)  {};
\node[draw=none, right] at (a1) {\Large $i$};

\node (a2) at (210:1.8) {};
\node (a3) at (330:1.8) {};

\draw (r) -- (a1);
\draw (r) -- (a2);
\draw (r) -- (a3);


\node (a1b1) at (70:3.2)  {};
\node (a1b2) at (110:3.2) {};
\draw (a1) -- (a1b1);
\draw (a1) -- (a1b2);
\node[draw=none, left] at (a1b2) {\Large $l_\mu$};

\node (a2b1) at (190:3.2) {};
\node (a2b2) at (230:3.2) {};
\draw (a2) -- (a2b1);
\draw (a2) -- (a2b2);
\node[draw=none, yshift=15pt] at (a2b1) {\Large $l_1=l_2$};
\node[draw=none, right] at (a2b2) {\Large $l_3=l_4$};

\node (a3b1) at (310:3.2) {};
\node (a3b2) at (350:3.2) {};
\draw (a3) -- (a3b1);
\draw (a3) -- (a3b2);


\node (a1b1c1) at (60:4.6) {};
\node (a1b1c2) at (80:4.6) {};
\draw (a1b1) -- (a1b1c1);
\draw (a1b1) -- (a1b1c2);

\node (a1b2c1) at (100:4.6) {};
\node (a1b2c2) at (120:4.6) {};
\draw (a1b2) -- (a1b2c1);
\draw (a1b2) -- (a1b2c2);
\node[draw=none, left] at (a1b2c2) {\Large $a_\mu$};

\node (a2b1c1) at (180:4.6) {};
\node (a2b1c2) at (200:4.6) {};
\draw (a2b1) -- (a2b1c1);
\draw (a2b1) -- (a2b1c2);
\node[draw=none, left] at (a2b1c1) {\Large $a_1$};
\node[draw=none, left] at (a2b1c2) {\Large $a_2$};

\node (a2b2c1) at (220:4.6) {};
\node (a2b2c2) at (240:4.6) {};
\draw (a2b2) -- (a2b2c1);
\draw (a2b2) -- (a2b2c2);
\node[draw=none, left] at (a2b2c1) {\Large $a_3$};
\node[draw=none, left] at (a2b2c2) {\Large $a_4$};

\node (a3b1c1) at (300:4.6) {};
\node (a3b1c2) at (320:4.6) {};
\draw (a3b1) -- (a3b1c1);
\draw (a3b1) -- (a3b1c2);

\node (a3b2c1) at (340:4.6) {};
\node (a3b2c2) at (0:4.6)   {};
\draw (a3b2) -- (a3b2c1);
\draw (a3b2) -- (a3b2c2);

\foreach \a in {0,30,...,330} {
  \coordinate (P\a) at ({8*cos(\a)},{6*sin(\a)});
  \coordinate (Q\a) at ({9*cos(\a)},{7*sin(\a)});
  \node (na\a) at (P\a) {};
  \node (nb\a) at (Q\a) {};
  \draw (na\a) -- (nb\a);
  \draw[dashed] (Q\a) circle (1.5cm);
}

\node[draw=none, above] at (P150) {\Large $b_\mu$};
\node[draw=none, above] at (Q150) {\Large $c_\mu$};

\node[draw=none, above] at (P180) {\Large $b_1$};
\node[draw=none, above] at (Q180) {\Large $c_1$};

\node[draw=none, above] at (P210) {\Large $b_2$};
\node[draw=none, above] at (Q210) {\Large $c_2$};

\node[draw=none, left] at (P240) {\Large $b_3$};
\node[draw=none, left] at (Q240) {\Large $c_3$};

\node[draw=none, left] at (P270) {\Large $b_4$};
\node[draw=none, left] at (Q270) {\Large $c_4$};

\end{tikzpicture}%
}
\caption{We view \(\cF\) (corresponding to the solid edges) as an embedded subgraph of 
\(\cG\), which contains all the switching edges. The indicator 
\(I(\cF,\cG)\) signifies that \(o, c_1, c_2, \ldots, c_\mu\) have tree 
neighborhoods and are well separated.}
\label{fig:FG}
\end{figure}

We introduce the following subgraph of $\cG$
\begin{equation}\label{e:buildF}
  \cF =B_{\ell}(o, \cG)\cup \{(l_\al, a_\al), (b_\al, c_\al)\}_{\al\in\qq{\mu}}=B_{\ell+1}(o, \cG)\cup \{ (b_\al, c_\al)\}_{\al\in\qq{\mu}}=(\bfi , E ),
\end{equation}
which contains all the switching edges, see \Cref{fig:FG}. To ensure switching edges are well separated and lie in large tree neighborhoods, we
use the following indicator.
\begin{align}\label{e:defIFG}
I(\cF ,\cG)
:= \prod_{\{x,y\}\in E } A_{xy}
   \;\prod_{c\in \{o,c_1,\cdots, c_\mu\}}\;\prod_{x\in \cB_\ell(c;\cG)}
      \bm{1}\!\big(\cB_{\fR}(x;\cG)\ \text{is a tree}\big)
   \; \prod_{\substack{c\neq c'\in \{o,c_1,\cdots, c_\mu\}}}      \bm{1}\!\big(\dist_\cG(c,c')\ge 3\fR\big).
\end{align}

Condition on \(I(\cF,\cG)=1\). Ignoring vertex labels, the graph \(\cF\) is simply 
a forest consisting of a \(d\)-regular tree truncated at depth \(\ell+1\) together 
with \(\mu\) disjoint edges. Thus the averaging over the edge \((i,o)\) in 
\eqref{e:maint0}, together with the randomness in the resampling data 
\(\{(b_\alpha,c_\alpha)\}_{\alpha\in\qq{\mu}}\), can be viewed as an average 
over embeddings of \(\cF\) into \(\cG\).

To formalize this, let \(\cF=(V,E)\) denote the \emph{template} whose vertices 
\(V\) are formal symbols, consisting of the truncated \(d\)-regular tree of depth 
\(\ell+1\) and the \(\mu\) disjoint edges. The vertices become concrete vertices 
of \(\cG\) only after an embedding is chosen. An embedding is specified by an 
assignment \(\bfi\in\qq{N}^{|V|}\); different choices of \(\bfi\) yield 
different embeddings. Once an embedding \(\bfi\) is fixed, we write 
\(\cF=(\bfi,E)\), and the indicator \(I(\cF,\cG)\) in \eqref{e:defIFG} is 
well defined.

\begin{claim}\label{c:resample}
We perform a local resampling around vertex $o$, and denote the resampled graph by $\wt \cG$, then
\begin{align}\begin{split}\label{e:maint02}
\bE\left[{\bm1(\cG\in \Omega)}(Q-Y)\right]
&=\frac{1}{Nd}\sum_{o,i} \bE\left[A_{oi}{\bm1(\cG\in \Omega)}(G_{oo}^{(i)}-Y)\right] \\
&=\frac{1}{Z_{\cF }}\sum_{ \bfi }\bE\left[I(\cF ,\cG)\bm1(\cG,\tcG\in \Omega) (\wt G_{oo}^{(i)}-Y )\right]+\OO(N^{-\fb/2}\bE[\Psi]),
\end{split}\end{align}
where the normalization constant $Z_\cF$ is not random, and for any $\cG\in \Omega$, it satisfies
\begin{align}\label{e:ZF}
\sum_{ \bfi }I(\cF ,\cG)=Z_{\cF}\left(1+\OO\left(\frac{1}{N^{1-2\fc}}\right)\right).
\end{align}
\end{claim}

The relation \eqref{e:ZF} follows from the fact that most vertices of $\cG$ has large tree neighborhood, so the number of ``good" embeddings of $\cF$ into $\cG$ concentrates.

If we temporarily ignore the indicator and the averaging over embeddings, the
claim reduces to the simple symmetry
\[
 \bE\!\left[A_{oi}\,(G^{(i)}_{oo}-Y)\right]
= \bE\!\left[A_{oi}\,(\wt G^{(i)}_{oo}-Y)\right],
\]
which follows from that $(\cG,\widetilde{\cG})$ are identically distributed under
the local resampling. The indicator $I(\cF,\cG)$ just restrict attention to “good” placements where we have good estimates for the Green's function. We omit the proof of \Cref{c:resample}.

In the following proposition, we show that the expectation \eqref{e:maint02} breaks down into an $\OO(1)$-weighted sum of terms in the same form, in the following sense.
\begin{definition}\label{def:O1sum}
    We say that $\mathcal U$ is an $\mathcal O(1)$-weighted sum of elements of $\mathcal R$ if there exist finitely many terms $R_1,\cdots,R_m\in\mathcal R$ (with $m$ possibly depending on $N$) and coefficients $\fc_1,\cdots,\fc_m$ such that
\[
\mathcal U=\sum_{j=1}^m \fc_j\,R_j,
\qquad\text{and}\qquad
\sum_{j=1}^m |\fc_j|=\mathcal O(1).
\]
\end{definition}

\begin{proposition}\label{p:iteration}
We recall $\cF=(\bfi, E)$ from \eqref{e:buildF}, and view it as embedded  in $\cG$ with vertices given by $\bfi$.  Then 
    \begin{align}\label{e:IFIF}
        \frac{1}{Z_{\cF}}\sum_{\bfi}\bE\left[I(\cF,\cG)\bm1(\cG,\tcG\in \Omega)(\widetilde G_{oo}^{(i)}-Y) \right]
        =I_1+I_2+\cE,
    \end{align}
    where $|\cE|=\OO((d-1)^{2\ell}\bE[\Psi])$, and   
  \begin{enumerate}
  \item  $I_1$ is an $\OO(1)$-weighted sum of terms of the following form
    \begin{align}\label{e:case1}
       \frac{1}{(d-1)^{\fq \ell/2}Z_{\cF}}\sum_{\bfi}  \bE[\bm1(\cG\in \Omega)I(\cF, \cG)(G_{c_\al c_\al}^{(b_\al)}-Y)(G_{c_\al c_\al}^{(b_\al)}-Q)],
    \end{align}
    where $\fq \geq 0$; 
        
      \item $I_2$ is an $\OO(1)$-weighted sum of terms of the following form
    \begin{align}\label{e:case3}
       \frac{(d-1)^{3r\ell}}{(d-1)^{\fq\ell/2}Z_{\cF}}\sum_{\bfi}  \bE[\bm1(\cG\in \Omega)I(\cF, \cG)(G_{c_\al c_\al}^{(b_\al)}-Y)R_{\bfi}].
    \end{align}
Here, $\fq\geq 0$ and $r\geq 2$, and the function $R_{\bfi}$ contains $r$ factors of the form 
\begin{align}\begin{split}\label{e:rdefcE1}
    &\{(G_{c_\al c_\al}^{(b_\al)}-Q)\}_{\al\in\qq{\mu}}, \quad  \{ G_{c_\al c_\beta}^{(b_\al b_\beta)}, G_{b_\al b_\beta}\}_{\al\neq\beta\in \qq{\mu}},\quad (Q-\msc(z)).
\end{split}\end{align}

  \end{enumerate}
\end{proposition}

\subsection{Switching using 
the Schur complement formula}

  To prove \Cref{p:iteration}, we need to rewrite $\wt G_{oo}^{(i)}-Y $ which depends on the Green's functions of the switched graph $\tcG$, in terms of the Green's functions of the original graph 
$\cG$, see \Cref{fig:switchingproc}. This is achieved as follows.

\begin{lemma}\label{l:coefficient}
Assume that $\cG, \widetilde \cG\in \Omega$ and $I(\cF,\cG)=1$ from \eqref{e:defIFG}, and define the index set $\sfA_i := \{ \alpha \in \qq{\mu} : \dist_{\cT}(i, l_\al) = \ell+1 \}$, see \Cref{fig:Ai}. The following holds:
  $\widetilde G_{oo}^{(i)}-Y$ can be rewritten as a weighted sum  
\begin{align}\begin{split}\label{e:easy_G-Y}
   \widetilde G_{oo}^{(i)}-Y
   &= \frac{\msc(z )^{2\ell+2}} {(d-1)^{\ell+1}}\sum_{\al\in\sfA_i}(G^{(b_\alpha)}_{c_\alpha c_\alpha}-Q)+\frac{\msc(z )^{2\ell+2}} {(d-1)^{\ell+1}}\sum_{\al\neq \beta\in\sfA_i}G^{(b_\alpha b_\beta)}_{c_\alpha c_\beta}
   + \cU+\cE.
\end{split}\end{align}
where 
$\cU$ is an $\OO(1)$-weighted sum of terms in the form $(d-1)^{3(r-1)\ell}R_r$, where $r\geq 2$, and $R_r$ contains $r$ factors of the form 
 \begin{align}\label{e:error_factor}
     (G^{(b_\al)}_{c_\al c_\al}-Q),  \quad G_{c_\al c_\beta}^{(b_\al b_\beta)},\quad Q-\msc(z), \quad \al\neq \beta\in \qq{\mu},
 \end{align}
 and at least one of them is $G_{c_\al c_\al}^{(b_\al)}-Q$ or $G_{c_\al c_\beta}^{(b_\al b_\beta)}$; and the error $\cE$ is given by
\begin{align}\begin{split}\label{e:defCE}
 \cE&= \frac{\msc^{2(\ell+1)}(z)} {(d-1)^{\ell+1}}\sum_{\al,\beta\in\sfA_i}(\wt G^{(\bT)}_{c_\alpha c_\beta}-G^{(b_\alpha b_\beta)}_{c_\alpha c_\beta})+\OO\left(\frac{1}{N^{\fb/2} } \sum_{\al, \beta\in \qq{\mu}} |\wt G^{(\bT)}_{c_\al c_\beta}-G^{(b_\al b_\beta)}_{c_\al c_\beta}| +\frac{1}{N^2}\right)\\
 &=\OO\left(\sum_{\al, \beta\in \qq{\mu}} |\wt G^{(\bT)}_{c_\al c_\beta}-G^{(b_\al b_\beta)}_{c_\al c_\beta}| +\frac{1}{N^2}\right).
\end{split}\end{align}
\end{lemma}

\begin{figure}[t]
\centering
\resizebox{\textwidth}{!}{%
\begin{tikzpicture}[
  every node/.style={circle,draw,inner sep=1.5pt},
  >=stealth
]

\fill[blue!10] (0,0) ellipse (9cm and 6cm);

\node[draw=none] at (-5,4) {\Huge $\mathcal G$};

\node (r) at (0,0) {};
\node[draw=none, below] at (r) {\Large $o$};

\draw[dashed] (0,0) circle (4cm);

\node[draw=none] at (3,2) {\Huge $\mathcal T$};

\node (a1) at (90:1.8)  {};
\node[draw=none, right] at (a1) {\Large $i$};

\node (a2) at (210:1.8) {};
\node (a3) at (330:1.8) {};

\draw (r) -- (a1);
\draw (r) -- (a2);
\draw (r) -- (a3);


\node (a1b1) at (70:3.2)  {};
\node (a1b2) at (110:3.2) {};
\draw (a1) -- (a1b1);
\draw (a1) -- (a1b2);
\node[draw=none, left] at (a1b2) {\Large $l_\mu$};

\node[fill=red] (a2b1) at (190:3.2) {};
\node[fill=red] (a2b2) at (230:3.2) {};
\draw (a2) -- (a2b1);
\draw (a2) -- (a2b2);
\node[draw=none, red, yshift=15pt] at (a2b1) {\Large $l_1=l_2$};
\node[draw=none, red, right] at (a2b2) {\Large $l_3=l_4$};

\node[fill=red] (a3b1) at (310:3.2) {};
\node[fill=red] (a3b2) at (350:3.2) {};
\draw (a3) -- (a3b1);
\draw (a3) -- (a3b2);
\node[draw=none, red, xshift=40] at (a3b2) {\Large $l_{2\nu-1}=l_{2\nu}$};

\node (a1b1c1) at (60:4.6) {};
\node (a1b1c2) at (80:4.6) {};
\draw (a1b1) -- (a1b1c1);
\draw (a1b1) -- (a1b1c2);

\node (a1b2c1) at (100:4.6) {};
\node (a1b2c2) at (120:4.6) {};
\draw (a1b2) -- (a1b2c1);
\draw (a1b2) -- (a1b2c2);
\node[draw=none, left] at (a1b2c2) {\Large $a_\mu$};

\node (a2b1c1) at (180:4.6) {};
\node (a2b1c2) at (200:4.6) {};
\draw (a2b1) -- (a2b1c1);
\draw (a2b1) -- (a2b1c2);
\node[draw=none, left] at (a2b1c1) {\Large $a_1$};
\node[draw=none, left] at (a2b1c2) {\Large $a_2$};

\node  (a2b2c1) at (220:4.6) {};
\node  (a2b2c2) at (240:4.6) {};
\draw (a2b2) -- (a2b2c1);
\draw (a2b2) -- (a2b2c2);
\node[draw=none, left] at (a2b2c1) {\Large $a_3$};
\node[draw=none,left] at (a2b2c2) {\Large $a_4$};

\node (a3b1c1) at (300:4.6) {};
\node (a3b1c2) at (320:4.6) {};
\draw (a3b1) -- (a3b1c1);
\draw (a3b1) -- (a3b1c2);

\node  (a3b2c1) at (340:4.6) {};
\node  (a3b2c2) at (0:4.6)   {};
\draw (a3b2) -- (a3b2c1);
\draw (a3b2) -- (a3b2c2);

\end{tikzpicture}%
}
\caption{The index set $\sfA_i := \{ \alpha \in \qq{\mu} : \dist_{\cT}(i, l_\al) = \ell+1 \}$ is given by $\{1,2,\cdots, 2\nu\}$, corresponding to these red nodes.}
\label{fig:Ai}
\end{figure}
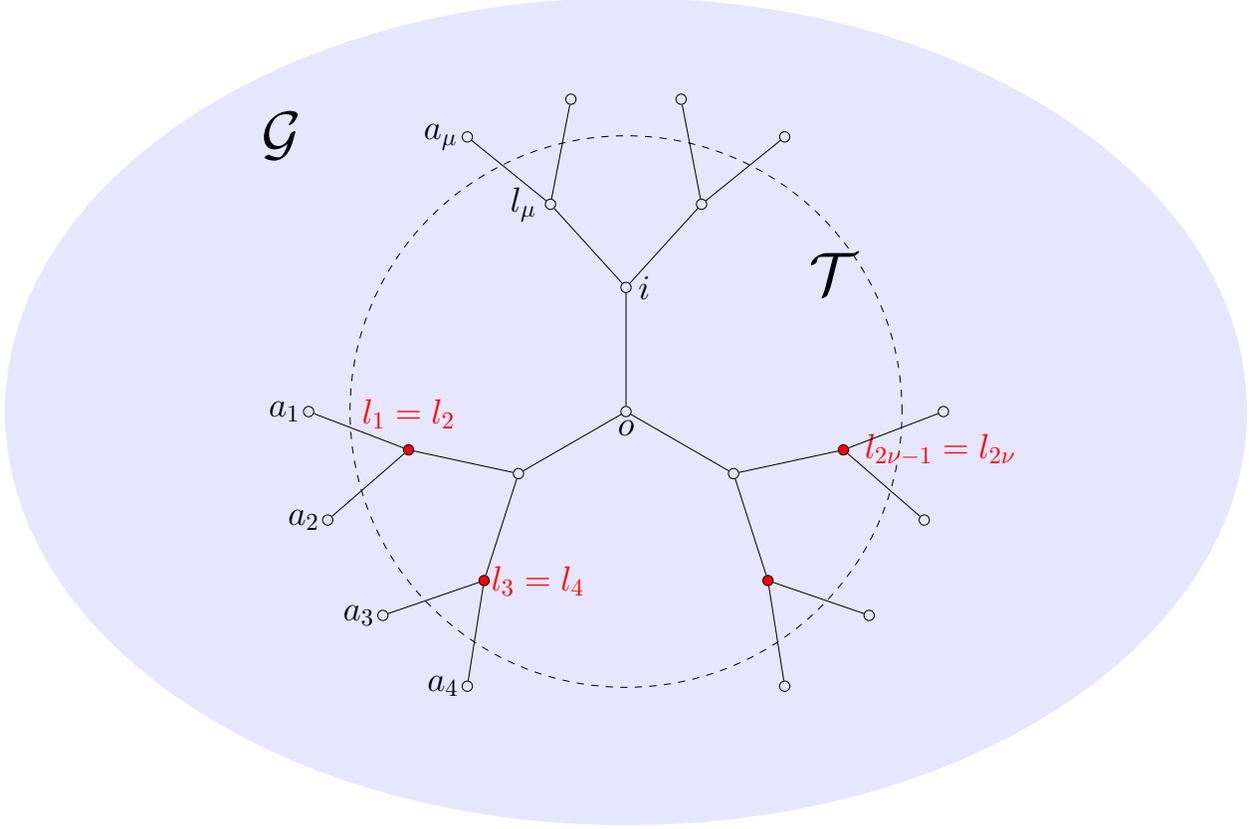

\begin{remark}
The expansion \eqref{e:easy_G-Y} decomposes \(\widetilde G_{oo}^{(i)}-Y\) into two leading contributions and a higher–order remainder $\cU$, which together comprise weighted sums of terms built from the factors in \eqref{e:error_factor}.
By \Cref{l:basicG}, each factor in \eqref{e:error_factor} is bounded by \(N^{-\fb}\) with high probability. Consequently, any product \(R_r\) of \(r\) such factors satisfies \(|R_r|\lesssim N^{-\fb r}\) with high probability. The term \(\cU\) collects an \(\OO(1)\)-weighted sum of contributions of the form \((d-1)^{3(r-1)\ell}R_r\) with \(r\geq 2\), so that
\[
(d-1)^{3(r-1)\ell}\,|R_r| \lesssim (d-1)^{3(r-1)\ell}N^{-\fb r}.
\]
Therefore, every contribution to \(\cU\) is higher–order (by at least one extra factor \(N^{-\fb}\)) compared to the two leading sums in \eqref{e:easy_G-Y}.

\end{remark}

\begin{remark}
More explicitly, the term $\cU$ in \eqref{e:easy_G-Y} is given by
    \begin{align}\begin{split}\label{e:Uterm}
        &\cU=\frac{\msc^{2(\ell+1)}(z)}{(d-1)^{\ell+2}}\sum_{\al \in \sfA_i}L^{(i)}_{l_\al l_\al}(G_{c_\al c_\al}^{(b_\al)}-Q)^2
        + \frac{\msc^{2(\ell+1)}(z)}{(d-1)^{\ell+2}}\sum_{ \al\in \sfA_i,\beta \in \qq{\mu}\atop \al\neq \beta}(L^{(i)}_{l_\beta l_\beta}+L^{(i)}_{l_\al l_\beta})(G_{c_\al c_\beta}^{(b_\al b_\beta)})^2\\
        &+
     \sum_{\al\neq\beta\in\qq{\mu}}  \fc_1(\al,\beta)  (G_{c_\al c_\al}^{(b_\al)}-Q)(G_{c_{\beta} c_{\beta}}^{(b_{\beta})}-Q)
    +
   \sum_{\al\in\qq{\mu}\atop \al'\neq\beta'\in\qq{\mu}}  \fc_2(\al,\al', \beta')  (G_{c_\al c_\al}^{(b_\al)}-Q)G_{c_{\al'} c_{\beta'}}^{(b_{\al'} b_{\beta'})}\\
    &+\sum_{\al\neq \beta\in \qq{\mu},\al'\neq \beta'\in \qq{\mu}\atop \{\al,\beta\}\neq \{\al',\beta'\}}\fc_3(\al,\beta,\al',\beta')  G_{c_\al c_\beta}^{(b_\al b_\beta)}G_{c_{\al'} c_{\beta'}}^{(b_{\al'} b_{\beta'})}+ \sum_{\al\in\qq{\mu}}\fc_4(\al)  (G_{c_\al c_\al}^{(b_\al)}-Q)(Q-\msc(z))\\
    &+ \sum_{\al\neq \beta\in\qq{\mu}}\fc_5(\al,\beta)  G_{c_\al c_\beta}^{(b_\al b_\beta)}(Q-\msc(z))
     +\textnormal{terms of the form }\{(d-1)^{3(r-1)\ell} R_r\}_{r\geq 3},
    \end{split}\end{align}
    where the total sum of the coefficients $\fc_1, \fc_2, \ldots, \fc_5$ is bounded by $\OO(1)$, and 
    \begin{align}\label{e:2defPi}
        L^{(i)}=P^{(i)}(\cT, z, \msc(z))=\frac{1}{H_{\bT}^{(i)}-z- \msc(z)\mathbb I^{\del}},\quad \mathbb I^{\del}_{xy}=\bm1(\dist_\cT(x,o)=\ell)\delta_{xy},\text{ for } x,y\in \bT, 
    \end{align}
    which is the Green's function of the $(d-1)$-ary tree.
\end{remark}

In the rest we prove \Cref{l:coefficient}. We start with the Schur complement formulas which will be used to prove \Cref{l:coefficient}. Let $\wt B$ be the normalized adjacency matrix of the directed edges $\{(c_\al, l_\al)\}_{\al \in \qq{\mu}}$. Then the adjacency matrices $\widetilde H^{(i)}$ is in the block form
\begin{align*}
    \widetilde H^{(i)}=
    \left[
    \begin{array}{cc}
        H^{(i)}_{\bT} & \wt B^\top\\
        \wt B & \widetilde H_{\bT^\complement}
    \end{array}
    \right].
\end{align*}
We also denote the Green's function of $\cG^{(\bT)}$ and $\wt\cG^{(\bT)}$ as $ G^{(\bT)}$ and $\wt G^{(\bT)}$ respectively.  

We collect some estimates below, which will be used later. Recall from \Cref{greentree}, for $x,y\in \bT\setminus\{i\}$, $|L_{xy}^{(i)}|\lesssim (d-1)^{-\dist_\cT(x,y)/2}$. It follows that
\begin{align}\begin{split}\label{e:sum_Pbound}
    &\sum_{x\in \bT\setminus\{i\}} |L^{(i)}_{ox}|\lesssim \sum_{r=0}^{\ell}(d-1)^{r/2}\lesssim (d-1)^{\ell/2},\\
    &\sum_{x,y\in \bT\setminus\{i\}} |L^{(i)}_{xy}|\lesssim
    \sum_{r=0}^{\ell}(d-1)^{\ell-r}\left(\sum_{r'=0}^r (d-1)^{r'/2}+\sum_{r'=r+1}^{2\ell-r}(d-1)^{r/2}\right)\lesssim \ell(d-1)^{\ell},
\end{split}\end{align}
where for the first sum, we used that $|\{x\in \bT: \dist_\cT(o,x)=r\}|=\OO((d-1)^{r})$ for $0\leq r\leq \ell$. For the second sum, we consider $\{x\in \bT: \dist_\cT(o,x)=\ell-r\}$. Note that there are $\OO((d-1)^{\ell-r})$ such vertices. Given such $x$, for any $y\in \bT$, $0\leq \dist_\cT(x,y)\leq 2\ell-r$.
There are two cases depending on $\dist_\cT(x,y)=r'$. If $0\leq r'\leq r$, we have $|\{y\in \bT: \dist_\cT(x,y)=r'\}|=\OO((d-1)^{r'})$, and if $r+1\leq r'\leq 2\ell-r$ we have $|\{y\in \bT: \dist_\cT(x,y)=r'\}|=\OO((d-1)^{(r+r')/2})$. The second statement in \eqref{e:sum_Pbound} follows from summing over $r,r'$ and using $|L_{xy}^{(i)}|\lesssim (d-1)^{-r'/2}$.

\begin{proof}[Proof of \Cref{l:coefficient}]

Thanks to the Schur complement formula \eqref{e:Schur1}, we have 
\begin{align*}
\widetilde G_{oo}^{(i)}=\left(\frac{1}{H_{\bT}^{(i)}-z-{\widetilde  B}^\top \widetilde G^{(\bT)}{\widetilde  B} }\right)_{oo}.
\end{align*}
Since $o$ has a radius $\fR$ tree neighborhood, $H^{(i)}_\bT$ is the normalized adjacency matrix of a truncated $(d-1)$-ary tree, and $L^{(i)}$ from \eqref{e:2defPi} agrees with the Green's function of $(d-1)$-ary tree (see \eqref{e:Gtreemsc2}),
\begin{align*}
    &\msc(z)=L^{(i)}_{oo}=P_{oo}^{(i)}(\cT, z, \msc(z))=\left(\frac{1}{H_{\bT}^{(i)}-z- \msc(z)\mathbb I^{\partial}}\right)_{oo}=\left(\frac{1}{H_{\bT}^{(i)}-z-\widetilde B^\top \msc(z)\widetilde B}\right)_{oo}, 
\end{align*}
By taking the difference of the two above expressions, we have
\begin{equation*}
\widetilde G_{oo}^{(i)}-\msc(z)=\left(\frac{1}{H_{\bT}^{(i)}-z-\widetilde B^\top \msc(z)\widetilde B-\cD}-\frac{1}{H_{\bT}^{(i)}-z-\widetilde B^\top \msc(z)\widetilde B}\right)_{oo},
\end{equation*}
where 
\begin{align}\begin{split}\label{e:defcE}
\cD
=\widetilde  B^\top(Q-\msc(z))  \widetilde  B+\cD_1, \quad
\cD_1= \wt B^\top(\widetilde G^{(\bT)}-Q  )\widetilde  B,
\end{split}\end{align}
are matrices indexed by $(\bT\setminus\{i\})\times(\bT\setminus\{i\})$.

By our assumption $\cG, \wt \cG\in \Omega$. Thanks to \eqref{eq:infbound} and \Cref{l:basicG}, we have 
\begin{align*}
     |(\cD_1)_{xy}|, |\cD_{xy}|\leq N^{-\fb}, \text{ for }x,y\in \bT\setminus\{i\}.
\end{align*} 
Thus, for some sufficiently large constant $\fp$, we have
\begin{align}\begin{split}\label{eq:resolventexp}
    \widetilde G_{oo}^{(i)}-\msc(z)&=\left(\frac{1}{H_{\bT}^{(i)}-z-\widetilde B^\top \msc(z)  \widetilde B-\cD}-\frac{1}{H_{\bT}^{(i)}-z-\widetilde B^\top \msc(z)  \widetilde B}\right)_{oo}\\
    &=\left(L^{(i)}\sum_{k=1}^\fp \left(\cD L^{(i)}\right)^k\right)_{oo}+\OO(N^{-2}).
\end{split}
\end{align}
Recall $Y=Y_\ell(Q, z)$ from \eqref{def:Y} and \eqref{e:defYt}:
\begin{align*}
Y=\left(\frac{1}{H_{\bT}^{(i)}-z-\widetilde B^\top Q  \widetilde B}\right)_{oo}.
\end{align*}
By the same argument as in \eqref{eq:resolventexp} we also have that
\begin{align}\label{e:YQ-m}
    Y-\msc(z)=\left(L^{(i)}\sum_{k=1}^\fp \left(\widetilde B^\top (Q-\msc(z))\widetilde B L^{(i)}\right)^k\right)_{oo}+\OO(N^{-2}).
\end{align}
By taking the difference of \eqref{eq:resolventexp} and \eqref{e:YQ-m}, up to error $\OO(N^{-2})$, we get that the difference $\widetilde G_{oo}^{(i)}-Y$ is given as
\begin{align}\begin{split}\label{e:GooY}
(L^{(i)} \cD_1L^{(i)})_{oo}
&+ \left(L^{(i)}\sum_{k=2}^\fp \left(\cD L^{(i)}\right)^k\right)_{oo}\\
&-\left(L^{(i)}\sum_{k=2}^\fp \left((\widetilde B^\top (Q-\msc(z)) \widetilde B L^{(i)}\right)^k\right)_{oo}.
\end{split}\end{align}

If $\al \in \sfA_i$, then $\dist_\cT(i, l_\al)=\ell+1$, and \Cref{greentree} gives
\begin{align}\label{e:Piol}
  L^{(i)}_{ol_\al}=\msc(z)\left(-\frac{{\msc(z)}}{\sqrt{d-1}}\right)^{\dist_\cT(o,l_\al)},\quad  |L^{(i)}_{ol_\al}|\lesssim (d-1)^{-\dist_{\cT}(o,l_\al)/2}=(d-1)^{-\ell/2}.
\end{align}
Otherwise if $\al \in\qq{\mu}\setminus \sfA_i$, then $o,l_\al$ are in different connected components of $\cT^{(i)}$, and $|L^{(i)}_{ol_\al}|=0$.
Thus the first term in \eqref{e:GooY} can be computed as, 
\begin{align}\label{e:G-Y2}
    (L^{(i)}\cD_1 L^{(i)})_{oo}=\frac{\msc(z)^{2\ell+2}} {(d-1)^{\ell+1}}\sum_{\al\in\sfA_i}(\wt G^{(\bT)}_{c_\alpha c_\alpha}-Q )
   +\frac{\msc(z)^{2\ell+2}} {(d-1)^{\ell+1}}\sum_{\al\neq \beta\in\sfA_i} \wt G^{(\bT)}_{c_\alpha c_\beta},
\end{align}

We obtain the first two terms in \eqref{e:easy_G-Y}, after replacing $\widetilde G_{c_\al c_\al}^{(\bT)}-Q ,\widetilde G_{c_\al c_\beta}^{(\bT)}$ in \eqref{e:G-Y2} by $G_{c_\al c_\al}^{(b_\al)}, G_{c_\al c_\beta}^{(b_\al b_\beta)}$. We collect the difference in the error term $\cE$ (as in \eqref{e:defCE}).

For the terms $k\geq 2$, in general for any matrix $V$ defined on $(\bT\setminus\{i\})\times (\bT\setminus\{i\})$, 
$L^{(i)}(V L^{(i)})^k$ is given as a sum of terms in the following form
\begin{align}\label{e:PUP}
   \sum_{x_1, x_2, \cdots, x_{2k}\in \bT\setminus\{i\}}L^{(i)}_{o x_1} V_{x_1 x_2} L^{(i)}_{x_2 x_3} V_{x_3 x_4} 
   L^{(i)}_{x_4 x_5}\cdots V_{x_{2k-1} x_{2k}}L^{(i)}_{x_{2k} o}.
\end{align}
We can reorganize \eqref{e:PUP} in the following way
\begin{align}\begin{split}\label{e:totalsum0}
    &\phantom{{}={}}\sum_{x_1, x_2, \cdots, x_{2k}\in \bT\setminus\{i\}}L^{(i)}_{o x_1} L^{(i)}_{x_2 x_3}\cdots L^{(i)}_{x_{2k} o}  V_{x_1 x_2}  V_{x_3 x_4} 
   \cdots V_{x_{2k-1} x_{2k}}.\\
   &=(d-1)^{3(k-1)\ell}\sum_{x_1, x_2, \cdots, x_{2k}\in \bT\setminus\{i\}}(d-1)^{-3(k-1)\ell}L^{(i)}_{o x_1} L^{(i)}_{x_2 x_3}\cdots L^{(i)}_{x_{2k} o} V_{x_1x_2} V_{x_3 x_4} 
   \cdots  V_{x_{2k-1}x_{2k}}\\
   &=:(d-1)^{3(k-1)\ell}\sum_{x_1, x_2, \cdots, x_{2k}\in \bT\setminus\{i\}}\fc_{\bmx}V_{x_1x_2}V_{x_3 x_4}
   \cdots  V_{x_{2k-1}x_{2k}},
\end{split}\end{align}
where the weights $\fc_{\bmx}=(d-1)^{-3(k-1)\ell}L^{(i)}_{o x_1} \cdots L^{(i)}_{x_{2k} o} $, and the total weights are bounded as
\begin{align*}
\sum_{x_1, x_2, \cdots, x_{2k}\in \bT\setminus\{i\}} |\fc_{\bmx}|
&=(d-1)^{-3(k-1)\ell}\sum_{x_1\in \bT\setminus\{i\}}|L^{(i)}_{o x_1}| \sum_{x_2, x_3\in \bT\setminus\{i\}}|L^{(i)}_{x_2 x_3}|\cdots \sum_{x_{2k}\in \bT\setminus\{i\}}|L^{(i)}_{x_{2k} o}|\\
&\lesssim (d-1)^{-3(k-1)\ell} \ell^{k-1} (d-1)^{k\ell}= (d-1)^{-(2k-3)\ell} \ell^{k-1}\lesssim 1,
\end{align*}
where to get the second line we used \eqref{e:sum_Pbound}; in the last inequality, we used that $k\geq 2$.

To compute the difference for $k\geq 2$ in \eqref{e:GooY}, we consider two possible forms for $V$:
$V=\widetilde B^\top (Q-\msc(z))\widetilde B $ or $V=\widetilde B^\top (Q-\msc(z))\widetilde B +\cD_1 $. As discussed above (see \eqref{e:totalsum0}), terms in \eqref{e:GooY} with $k\geq 2$ break down to an $\OO(1)$-weighted sum of terms in the form $(d-1)^{3(k-1)\ell}\wt R_k$. Here $\wt R_k$ is a product of $k$ factors,  each taking  one of the following:
\begin{align*}
    (\wt G_{c_\al c_\al}^{(\bT)}-Q), \quad  \wt G_{c_\al c_{\beta}}^{(\bT)},\quad  (Q-\msc(z)), \quad \al\neq \beta\in \qq{\mu}.
\end{align*} 
Moreover, $\wt R_k$ contains at least one factor of the form $\{\wt G_{c_\al c_\al}^{(\bT)}-Q, \wt G_{c_\al c_\beta}^{(\bT)}\}_{\al\neq \beta\in \qq{\mu}}$ (arising from  $\cD_1$). Otherwise, the terms from the difference in \eqref{e:GooY} cancel out.

For each $\wt R_k$ terms, we get $R_k$ by replacing $\widetilde G_{c_\al c_\al}^{(\bT)}-Q,\widetilde G_{c_\al c_\beta}^{(\bT)}$ with $G_{c_\al c_\al}^{(b_\al)}-Q, G_{c_\al c_\beta}^{(b_\al b_\beta)}$, respectively. As $k\geq 2$  and each factor of $R_k, \wt R_k$ is bounded by $N^{-\fb} $ (by \Cref{l:basicG} and the assumption that $\cG,\widetilde\cG\in\Omega$), the replacement error is bounded by
\begin{align*}
    |\wt R_k-R_k|\lesssim \frac{1}{N^{\fb/2} } \sum_{\al, \beta\in \qq{\mu}} |\wt G^{(\bT)}_{c_\al c_\beta}-G^{(b_\al b_\beta)}_{c_\al c_\beta}|.
\end{align*}
We collect the above error in $\cE$ (as in \eqref{e:defCE}).

We denote the $\OO(1)$-weighted sum of terms in the form $(d-1)^{3(k-1)\ell} R_k$ as $\cU$. This finishes  the proof of \Cref{l:coefficient}.

\end{proof}

\subsection{Proof of \Cref{p:iteration}}

By plugging \eqref{e:easy_G-Y} into \eqref{e:IFIF}, we get the following four terms
\begin{align}
&\label{e:e1}\frac{\msc^{2(\ell+1)}(z)}{(d-1)^{\ell+1}} \sum_{\al\in \sfA_i}   \sum_{\bfi} \frac{1}{Z_{\cF}} \bE\left[I(\cF,\cG)\bm1(\cG\in \Omega)(G_{c_\al c_\al}^{(b_\al)}-Q)\right],\\
&\label{e:e2}\frac{\msc^{2(\ell+1)}(z)}{(d-1)^{\ell+1}}\sum_{\al\neq \beta\in \sfA_i}\sum_{\bfi}  \frac{1}{Z_{\cF}}\bE\left[I(\cF,\cG)\bm1(\cG\in \Omega) G_{c_\al c_\beta}^{(b_\al b_\beta)}\right],\\
&\label{e:e3}\frac{1}{Z_{\cF}}\sum_{\bfi} \bE\left[I(\cF,\cG)\bm1(\cG\in \Omega)\cU \right],\\
&\label{e:e4}\frac{1}{Z_{\cF}}\sum_{\bfi} \bE\left[I(\cF,\cG)\bm1(\cG\in \Omega)\cE \right].
\end{align}

\paragraph{First term \eqref{e:e1}.} The first term \eqref{e:e1} is negligible
\begin{align}\label{e:ftt1}
&\frac{\OO(1)}{(d-1)^{\ell}} \sum_{\al\in \sfA_i}   \sum_{\bfi} \frac{1}{Z_{\cF}} \bE\left[I(\cF,\cG)\bm1(\cG\in \Omega)(G_{c_\al c_\al}^{(b_\al)}-Q)\right]=\OO(N^{-\fb/2}\bE[\Psi]).
\end{align}

If we temporarily ignore the indicator and the averaging over embeddings, the
claim \eqref{e:ftt1} reduces to the definition of $Q$
\begin{align}\label{e:core1}
\frac{1}{Nd}\sum_{b_\al, c_\al \in \qq{N}} A_{c_\al b_\al}\,(G^{(b_\al)}_{c_\al c_\al}-Q)=0.
\end{align}
which follows from that $(\cG,\widetilde{\cG})$ are identically distributed under
the local resampling. In \eqref{e:ftt1}, for the average over embeddings, we can first sum over the indices $b_\al, c_\al$ using \eqref{e:core1}, and then average over other indices $\bfi\setminus\{b_\al,c_\al\}$.

\paragraph{Second term  \eqref{e:e2}.}
     Assume the following estimate for the second term \eqref{e:e2}
\begin{align}\begin{split}\label{e:ftt2}
&\phantom{{}={}}\frac{\OO(1)}{(d-1)^{\ell}} \sum_{\al\neq \beta\in \sfA_i}   \sum_{\bfi}  \frac{1}{Z_{\cF}}\bE\left[I(\cF,\cG)\bm1(\cG\in \Omega) G_{c_\al c_\beta}^{(b_\al b_\beta)}\right]\\
&=\frac{\OO(1)}{(d-1)^{\ell}} \sum_{\al\neq \beta\in \sfA_i} \sum_{\bfi}  \frac{1}{Z_{\cF}}\bE\left[I(\cF,\cG)\bm1(\cG\in \Omega) G_{b_{\al} b_{{\beta}}}(G_{c_{\al} c_{\al}}^{(b_{\al})}-Q ) (G_{c_{{\beta}} c_{{\beta}}}^{(b_{{\beta}})}-Q )\right]+\OO(N^{-\fb/4} \bE[  \Psi]).
\end{split}\end{align}
Then we can further change $(G_{c_\al c_\al}^{(b_\al)}-Q )$ in \eqref{e:ftt2} to $(G_{c_\al c_\al}^{(b_\al)}-Y )$ and get
\begin{align}\begin{split}
\label{e:chQY}
&\frac{\OO(1)}{(d-1)^{\ell}} \sum_{\al\neq \beta\in \sfA_i} \sum_{\bfi}  \frac{1}{Z_{\cF}}\bE\left[I(\cF,\cG)\bm1(\cG\in \Omega)(G_{c_{\al} c_{\al}}^{(b_{\al})}-Y )\times ( G_{b_{\al} b_{{\beta}}} (G_{c_{{\beta}} c_{{\beta}}}^{(b_{{\beta}})}-Q ))\right]\\
+
&\frac{\OO(1)}{(d-1)^{\ell}} \sum_{\al\neq \beta\in \sfA_i} \sum_{\bfi}  \frac{1}{Z_{\cF}}\bE\left[I(\cF,\cG)\bm1(\cG\in \Omega)|Y-Q| | G_{b_{\al} b_{{\beta}}} (G_{c_{{\beta}} c_{{\beta}}}^{(b_{{\beta}})}-Q )|\right].
\end{split}\end{align}
The first term on the righthand side of \eqref{e:chQY} is in the form of \eqref{e:case3}.
The second term on the righthand side of \eqref{e:chQY} can be further bounded as
\begin{align*}
\begin{split}
\frac{\OO(1)}{(d-1)^{\ell}} \sum_{\al\neq \beta\in \sfA_i} \sum_{\bfi}  \frac{1}{Z_{\cF}}\bE\left[I(\cF,\cG)\bm1(\cG\in \Omega)|Y-Q| N^{-2\fb}\right]
=\OO\left(N^{-\fb/4} \bE[\Psi]\right).
\end{split}\end{align*}

Next we outline the proof of \eqref{e:ftt2}. If we temporarily ignore the indicator and the averaging over embeddings, the
claim \eqref{e:ftt2} boils down to compute the following quantity
\begin{align}\label{e:core2}
\frac{1}{(Nd)^2}\sum_{b_\al, c_\al, b_\beta, c_\beta \in \qq{N}} A_{c_\al b_\al}A_{b_\beta c_\beta}G^{(b_\al b_\beta)}_{c_\al c_\beta}.
\end{align}
For random $d$-regular graphs, the adjacency matrices have a trivial eigenvector $(1,1,\cdots, 1)^\top$. As a consequence the row and column sums of the Green's function is small: 
\begin{align}\label{e:av1}
\sum_{i\in \qq{N}} G_{ij}=\sum_{j\in \qq{N}} G_{ij}=1/(d/\sqrt{d-1}-z)=\OO(1).
\end{align}
Thus if $G^{(b_\al b_\beta)}_{c_\al c_\beta}$ in \eqref{e:core2} is replaced by $G_{c_\al c_\beta}$ (without removing the vertices $b_\al, b_\beta$), then the average is very small $\OO(1/N)$. To compute \eqref{e:core2}, we need  to express  $G^{(b_\al b_\beta)}_{c_\al c_\beta}$ back to  $G^{(b_\al)}, G^{(b_\beta)}, G$ using Schur complement formula by carefully adding vertices $b_\al, b_\beta$ back. And it turns out we have nice leading term, and all other terms are negligible
\begin{align}\begin{split}
\label{e:core21}
&\phantom{{}={}}\frac{1}{(Nd)^2}\sum_{b_\al, c_\al, b_\beta, c_\beta \in \qq{N}} A_{c_\al b_\al}A_{b_\beta c_\beta}G^{(b_\al b_\beta)}_{c_\al c_\beta}\\
&=
\frac{1}{(Nd)^2}\sum_{b_\al, c_\al, b_\beta, c_\beta \in \qq{N}} A_{c_\al b_\al}A_{b_\beta c_\beta}\frac{G_{b_{\al} b_{{\beta}}}}{d-1}(G_{c_{\al} c_{\al}}^{(b_{\al})}-Q ) (G_{c_{{\beta}} c_{{\beta}}}^{(b_{{\beta}})}-Q )+\text{``negligible error"}.
\end{split}\end{align}

The claim \eqref{e:ftt2} follows by first summing over the indices $b_\alpha,c_\alpha,b_\beta,c_\beta$ using \eqref{e:core21}, and then averaging over the remaining indices in $\bfi\setminus\{b_\alpha,c_\alpha,b_\beta,c_\beta\}$.

\begin{proof}[Proof of \eqref{e:core21}]
For simplicity of notation, we write $b_\al,c_\al, b_\beta, c_\beta$ as $b,c,b',c'$.
First we notice that using \eqref{e:av1}, the average of Green's functions with only one vertex removed is also small. More precisely, the Schur complement formula \eqref{e:Schurixj} gives
\begin{align}\begin{split}\label{e:rGbs2}
   &\phantom{{}={}}\frac{1}{Nd}\sum_{b \sim c } G_{c c'}^{(b')}  =\frac{1}{Nd}\sum_{c \sim b }\left(G_{c c'}-\frac{G_{c  b'}G_{b'c'}}{G_{b'b'}}\right)=\OO\left(\frac{1}{N}\right), \quad
   \frac{1}{Nd}\sum_{b \sim c } G_{b c'}^{(b')}  =\OO\left(\frac{1}{N}\right),
\end{split}\end{align}
where in the first statement we average over the free index $c$; in the second statement we average over the free index $b$.

To show \eqref{e:core21}, thanks to the Schur complement formula \eqref{e:Schurixj}, we have
\begin{align}\begin{split}\label{e:rGccbb1}
    &\phantom{{}={}}G_{c c'}^{(b b')}
    =G_{c c'}^{( b')}+(G^{(b b')}H)_{cb}G^{(b')}_{bc'}=G_{c c'}^{(b')}+
   \frac{G_{b c'}^{(b')}}{\sqrt{d-1}}\sum_{x\sim b}G_{c x}^{(b b')}\\
   &=G_{c c'}^{( b')}+   \frac{G_{b c'}^{(b')}}{\sqrt{d-1}}\sum_{x\sim b}\left(G_{c  x}^{(b )}-\frac{G_{c  b'}^{(b)}G_{b' x}^{(b )}}{G_{b'b'}^{(b )}}\right)
   =G_{c c'}^{( b')}+   \frac{G_{b c'}^{(b')}}{\sqrt{d-1}}\sum_{x\sim b}G_{c  x}^{(b )}
   -\sum_{x\sim b}\frac{G_{b c'}^{(b')}G_{c  b'}^{(b)}G_{b' x}^{(b )}}{\sqrt{d-1}G_{b'b'}^{(b )}}.
\end{split}\end{align}
If we average over the edges $(b , c )$, the first term on the righthand side of \eqref{e:rGccbb1} is small by \eqref{e:rGbs2}; for the second term, by the same reasoning, we can replace $\sum_{x\sim b}G_{c  x}^{(b )}$ by $\sum_{x\sim b}(G_{c  x}^{(b )}-Q)$; for the last term, we can bound it by the Ward identity:
\begin{align}\label{e:rwarduse1}
    \begin{split}\frac{A_{b'c'}}{Nd}\sum_{b \sim c } \sum_{x\sim b }\left|\frac{G_{b  c'}^{(b')}G_{c  b'}^{(b )}G_{b' x}^{(b )}  }{G_{b'b'}^{(b )}}\right|
    &\lesssim  \frac{A_{b'c'} N^{-\fb} }{N}\sum_{b \sim c }|G_{b  c'}^{(b')}G_{c  b'}^{(b )}|\\&\lesssim  \frac{A_{b'c'} N^{-\fb} }{N}\sum_{b \sim c }(|G_{b  c'}^{(b')}|^2+|G_{c  b'}^{(b )}|^2)\lesssim  N^{-\fb/2}A_{b'c'} \Phi,
\end{split}\end{align}
where in the first statement we used that  $|G_{b'x}^{(b)}|\lesssim N^{-\fb} $ for $b\sim x$ from \eqref{eq:local_law}; in the second statement we used the Cauchy–Schwarz inequality; and in the last statement we used \eqref{eq:local_law}.

Thus averaging over the edges $(b , c )$ for \eqref{e:rGccbb1} we conclude that
\begin{align}\begin{split}\label{e:rGbbcc0}
    \frac{A_{b'c'}}{Nd}\sum_{b \sim c } 
 G_{c c'}^{(b  b')}
    &= \frac{A_{b'c'}}{Nd}\sum_{b \sim c } \left(\frac{G_{b  c'}^{(b')}}{\sqrt{d-1}}\left(\sum_{x\sim b}G_{c  x }^{(b )}-Q) \right)+\OO(N^{-\fb/2} \Phi)\right).
\end{split}\end{align}
We can then average over edges $( b', c')$  in \eqref{e:rGbbcc0}. Similar to \eqref{e:rGccbb1}, we can first replace $G_{b  c'}^{( b')}$ as 
\begin{align}\begin{split}\label{e:Gccbb2}
    G_{b  c'}^{( b')}=G_{b  c'}+G_{ b b' }(H G^{( b')})_{b'c'}=G_{b  c'}+
   \frac{G_{b   b'}}{\sqrt{d-1}}\sum_{x\sim  b'}G^{( b')}_{c' x},
\end{split}\end{align}
and conclude
\begin{align*}\begin{split}
\frac{1}{(Nd)^2}\sum_{b \sim c \atop  b'\sim c'}G_{c  c'}^{(b   b')}
    &=\frac{1}{(Nd)^2}\sum_{b \sim c \atop  b'\sim c'}\frac{G_{b   b'}}{d-1}\left(\sum_{x\sim b }G_{c  x}^{(b )} -Q\right)\left( \sum_{x\sim  b'}G_{c' x}^{( b')}-Q\right) +\OO(N^{-\fb/2} \Phi)\\
&=\frac{1}{(Nd)^2}\sum_{b \sim c \atop  b'\sim c'}\frac{G_{b   b'}}{d-1}(G_{c  c }^{(b )}-Q) (G_{c' c'}^{( b')}-Q) \\
&+\frac{N^{-\fb/2} }{(Nd)^2}\sum_{b \sim c \atop  b'\sim c'} \OO\left(|G_{b   b'}|\left(\sum_{x\sim b , x\neq c }|G_{c  x}^{(b )}|+\sum_{x\sim  b', x\neq c'}|G_{c' x}^{( b')}|\right)+ \Phi\right),
\end{split}\end{align*}
where for the last equality we bound $|G_{c x}^{( b)}-\delta_{cx}Q|\lesssim N^{-\fb} $ for $x\sim b$, and $|G_{c' x}^{( b')}-\delta_{c'x}Q|\lesssim N^{-\fb} $ for $x\sim b'$ using \eqref{eq:local_law} and \eqref{eq:infbound}. This gives \eqref{e:core21}.
\end{proof}

\paragraph{Third term  \eqref{e:e3}.} We recall the first few terms of $\cU$ from  \eqref{e:Uterm}.
     We have the following decomposition for the third term in \eqref{e:e3}

\begin{align}\label{e:decomposeI1I2I3}
    &\frac{1}{Z_{\cF}}\sum_{\bfi} \bE\left[I(\cF,\cG)\bm1(\cG\in \Omega)\cU \right]
    =J_1+J_2+J_3,
\end{align}
where
\begin{align}
 &\label{e:defI1}J_1=\sum_{\al \in \sfA_i}\sum_{\bfi} \frac{\msc^{2(\ell+1)}(z )L^{(i)}_{l_\al l_\al}}{(d-1)^{\ell+2}Z_{\cF}} \bE\left[I(\cF,\cG)\bm1(\cG\in \Omega)(G_{c_\al c_\al}^{(b_\al)}-Q )^2\right],\\
   &\label{e:defI2}J_2= \sum_{ \al \in \sfA_i,\beta\in\qq{\mu}\atop\al\neq \beta}\sum_{\bfi^+} \frac{\msc^{2(\ell+1)}(z )(L^{(i)}_{l_\beta l_\beta}+L^{(i)}_{l_\al l_\beta})}{(d-1)^{\ell+2}Z_{\cF}}
\bE\left[I(\cF,\cG)\bm1(\cG\in \Omega)(G_{c_\al c_\beta}^{(b_\al b_\beta)})^2\right].
\end{align}
And $J_3$ is an $\OO(1)$-weighted sum of terms in the form 
\begin{align}\label{e:fhatRi03}
  \sum_{\bfi} \frac{(d-1)^{3(h-1)}}{Z_{\cF}} \bE\left[I(\cF,\cG)\bm1(\cG\in \Omega)R_h\right],\quad h\geq 2,
\end{align}
where $R_h$ is a product of $h$ terms in the form \eqref{e:rdefcE1}, which contains $G_{c_\al c_\al}^{(b_\al)}-Q $ or $G_{c_\al c_\beta}^{(b_\al b_\beta)}$. Moreover, either $h\geq 3$, or $h=2$ and  $R_h$ (recall from \eqref{e:Uterm}) is one of the following  terms 
\begin{align}\begin{split}\label{e:Rhform}
 &(G_{c_\al c_\al}^{(b_\al)}-Q ) (G_{c_\beta c_\beta}^{(b_\beta)}-Q ), \quad   (G_{c_\al c_\al}^{(b_\al)}-Q )G_{c_{\al'} c_{\beta'}}^{(b_{\al'} b_{\beta'})},\quad G_{c_{\al} c_{\beta}}^{(b_{\al} b_{\beta})}G_{c_{\al'} c_{\beta'}}^{(b_{\al'} b_{\beta'})},\\
&\{G_{c_\al c_\al}^{(b_\al)}-Q , G_{c_\al c_\beta}^{(b_\al b_\beta)}\}\times(Q -\msc(z )),
\end{split} \end{align}
where $\alpha\neq \beta\in\qq{\mu}, \al'\neq \beta'\in \qq{\mu}$ and $ \{\al, \beta\}\neq \{\al', \beta'\}$.

In the following we discuss the three terms $J_1, J_2, J_3$ one by one. For $J_1$, by the same argument as in \eqref{e:chQY}, we can change one copy of $(G_{c_\al c_\al}^{(b_\al)}-Q )$ in \eqref{e:defI1} to $(G_{c_\al c_\al}^{(b_\al)}-Y )$, and the error is bounded by $\OO(N^{-\fb/2} \bE[ \Psi])$. After such replacement, we get
\begin{align}\label{e:newterm3}
\sum_{\al \in \sfA_i}\sum_{\bfi}\frac{\msc^{2(\ell+1)}(z )L^{(i)}_{l_\al l_\al}}{(d-1)^{\ell+2}Z_{\cF}}  \bE\left[I(\cF,\cG)\bm1(\cG\in \Omega)(G_{c_\al c_\al}^{(b_\al)}-Y ) (G_{c_\al c_\al}^{(b_\al)}-Q )\right],
\end{align}
which is an $\OO(1)$-weighted sum of terms in the form \eqref{e:case1} with $\fq=0$.

For $J_2$ as in \eqref{e:defI2}, we claim the following estimate 
\begin{align}\label{e:Gccerror0}
|J_2|
\;\lesssim\;
\sum_{\alpha\neq\beta}\frac{1}{(d-1)^\ell\,Z_{\cF}}
\sum_{\bfi}\bE\!\left[
  I(\cF,\cG)\,\bm{1}(\cG\in\Omega)\, \big|G^{(b_\alpha b_\beta)}_{c_\alpha c_\beta}\big|^2
\right]
\;\lesssim\;
(d-1)^{2\ell}\,\bE[\Psi],
\end{align}
so this contribution can be absorbed into the error term $\cE$ in
\eqref{e:IFIF}.

If we temporarily ignore the indicator and the averaging over embeddings, the
claim \eqref{e:Gccerror0} boils down to compute the following quantity
\begin{align}\label{e:Gccerror}
\frac{1}{(Nd)^2}\sum_{b_\al, c_\al, b_\beta, c_\beta \in \qq{N}} A_{c_\al b_\al}A_{b_\beta c_\beta}|G^{(b_\al b_\beta)}_{c_\al c_\beta}|^2\lesssim \Phi.
\end{align}

 \begin{proof}[Proof of \eqref{e:Gccerror}]

Averages like \eqref{e:Gccerror} can be bounded by the classical Ward identity \eqref{e:Ward}, which gives
\begin{equation}\label{eq:rwardex}
\frac{1}{N^2}\sum_{i,j}\big|G_{ij}(z)\big|^2
=\frac{\Im m_N(z)}{N\Im[z]}\lesssim \Phi(z).
\end{equation}
However, the entry
$G^{(b_\alpha b_\beta)}_{c_\alpha c_\beta}$ (the Green's function with rows/columns
$b_\alpha,b_\beta$ removed) depends on the choice of these edges, but we can
express it in terms of the full Green’s function $G$ by  the Schur complement
formula and bound the average using the Ward identity. For simplicity of notation, we write $b_\al,c_\al, b_\beta, c_\beta$ as $b,c,b',c'$.

We start with the Schur complement formula \eqref{e:Schur1}
\begin{align}\label{e:scf}
 G_{cc'}^{(bb')}=G_{cc'}-(G (G|_{\{bb'\}})^{-1} G)_{cc'},
\quad
    (G|_{\{bb'\}})^{-1}
    &=\frac{1}{G_{bb}G_{b'b'} -G_{bb'}^2}
    \left[
    \begin{array}{cc}
    G_{b'b'} & -G_{bb'}\\
    -G_{bb'} & G_{bb}
    \end{array}
    \right].
\end{align}
Thanks to \eqref{eq:infbound}, the Green's function terms are bounded, and the denominators in \eqref{e:scf} are bounded away from $0$ when $b,b'$ are far from each other, we can get
\begin{align}\label{e:Gcc'bb'}
|G_{cc'}^{(bb')}|\lesssim |G_{bb'}|+|G_{bc'}|+|G_{b'c}|+|G_{cc'}|.
\end{align}
Then the Ward identity bound \eqref{eq:rwardex} leads to 
\begin{align}\begin{split}\label{e:ccbb}
   \frac{1}{(Nd)^2} \sum_{c\sim b\atop c'\sim b'} |G_{cc'}^{(bb')}|^2\lesssim 
  \frac{1}{(Nd)^2}\sum_{c\sim b\atop c'\sim b'}(|G_{bb'}|^2+|G_{bc'}|^2+|G_{b'c}|^2+|G_{cc'}|^2)\lesssim \Phi.
\end{split}\end{align}

\end{proof}

Finally, for $J_3$, if $h=2$, $R_h$ is given in \eqref{e:Rhform}. There are two cases
\begin{enumerate}
\item $R_h$ contains a factor $G_{c_\al c_\al}^{(b_\al)}-Q$, then we can first average over edges $(b_\al, c_\al)$ as in \eqref{e:core1}. In this case we can bound \eqref{e:fhatRi03} by the same way as for \eqref{e:ftt1}, and it is bounded by $\OO(N^{-\fb/2}\bE[\Psi])$
\item $R_h$ contains a factor $G_{c_\al c_\beta}^{(b_\al b_\beta)}$, then we can first average over edges $(b_\al, c_\al), (b_\beta, c_\beta)$ as in \eqref{e:core2}. In this case we can estimate \eqref{e:fhatRi03} by the same way as for \eqref{e:ftt2}, and get
\begin{align}\begin{split}\label{e:rhuanG}
        \frac{1}{(d-1)}\frac{(d-1)^{3(h+1)\ell}}{(d-1)^{\fq\ell/2}Z_{\cF}}\sum_{\bfi}  \bE\left[I(\cF,\cG)\bm1(\cG\in \Omega)(G^{(b_\beta)}_{c_\al c_\al}-Y) R'_{h+1}\right]+\OO(N^{-\fb/4}\bE[  \Psi]),
   \end{split} \end{align}
   where $R'_{h+1}$ is obtained from $R_h$ by replacing $G_{c_\al c_\beta}^{(b_\al b_\beta)}$ by $G_{b_{\al} b_{{\beta}}} (G_{c_{{\beta}} c_{{\beta}}}^{(b_{{\beta}})}-Q )$, and it contains $h+1$ factors in the form \eqref{e:rdefcE1}; and $\fq=12$. This leads to \eqref{e:case3}.
\end{enumerate}

If $h\geq 3$, we recall that $R_h$ contains either $G_{c_\al c_\al}^{(b_\al)}-Q $ or $G_{c_\al c_\beta}^{(b_\al b_\beta)}$. There are three cases
\begin{enumerate}
\item If $R_h$ contains a factor $G_{c_\al c_\al}^{(b_\al)}-Q$, we can change it to $(G_{c_\al c_\al}^{(b_\al)}-Y )$, and the error is bounded by $\OO(N^{-\fb/2} \bE[ \Psi])$.

\item If $R_h$ contains two factors in the form $G_{c_\al c_\beta}^{(b_\al b_\beta)}$, by Cauchy-Schwarz inequality and similarly to \eqref{e:Gccerror}, we can bound it as
\begin{align*}
  \sum_{\bfi} \frac{(d-1)^{3(h-1)}}{Z_{\cF}} \bE\left[I(\cF,\cG)\bm1(\cG\in \Omega)N^{-(h-2)\fb}|G_{c_\al c_\beta}^{(b_\al b_\beta)}|^2\right]=\OO(N^{-\fb/2} \bE[ \Psi]).
\end{align*}

\item If $R_h$ does not contain any factor in the form $G_{c_\al c_\al}^{(b_\al)}-Q$, contains exactly one factor in the form $G_{c_\al c_\beta}^{(b_\al b_\beta)}$ and all other factors are $Q-\msc(z)$, then by the same reasoning as in \eqref{e:rhuanG}, up to negligible error we can replace $G_{c_\al c_\beta}^{(b_\al b_\beta)}$ by $(G^{(b_\beta)}_{c_\al c_\al}-Y) \times(G_{b_{\al} b_{{\beta}}} (G_{c_{{\beta}} c_{{\beta}}}^{(b_{{\beta}})}-Q ))$.

\end{enumerate}
All three cases above lead to  \eqref{e:case3}.

\paragraph{Fourth term \eqref{e:e4}.} We recall $\cE$ from \eqref{e:defCE}, the last term \eqref{e:e4} is given by
\begin{align}\begin{split}\label{e:Eterm1}
&\phantom{{}={}}\frac{1}{Z_{\cF}}\sum_{\bfi}\bE\left[I(\cF,\cG)\bm1(\cG,\tcG\in \Omega)\cE\right]\\
&=\frac{1}{Z_{\cF}}\sum_{\bfi}\bE\left[I(\cF,\cG)\bm1(\cG,\tcG\in \Omega)\left(\frac{\msc^{2(\ell+1)}(z )} {(d-1)^{\ell+1}}\sum_{\al,\beta\in\sfA_i}(\wt G^{(\bT)}_{c_\alpha c_\beta}-G^{(b_\alpha b_\beta)}_{c_\alpha c_\beta})\right) \right]\\
&+\OO\left(\frac{1}{Z_{\cF}}\sum_{\bfi}\bE\left[I(\cF,\cG)\bm1(\cG,\tcG\in \Omega)\left(\frac{1}{N^{\fb/2}} \sum_{\al,\beta\in \qq{\mu}}|\wt G^{(\bT)}_{c_\alpha c_\beta}-G^{(b_\alpha b_\beta)}_{c_\alpha c_\beta}|+\frac{1}{N^2}\right) \right]\right).
\end{split}\end{align}

We can express $\tG^{(\bT)}_{c_\al c_\beta}$ in terms of the Green's function $G$ of the original graph using Schur complement formula. In \Cref{l:diffG1} we show
\begin{align}\label{e:replaceexample}
    \tG^{(\bT)}_{c_\al c_\beta}=G^{(b_\al b_\beta)}_{c_\al c_\beta}+\cE_{\al \beta},
\end{align}
where 
\begin{align}\begin{split}\label{e:rcEbound}
    |\cE_{\al \beta}|&\lesssim  \sum_{\gamma\in\qq{\mu}, x\in \cN_\gamma}(|G_{c_\al x}^{(\bT\bW)}|^2+|G_{c_\beta x}^{(\bT\bW)}|^2)+\sum_{\gamma\in\qq{\mu}\setminus\{\al,\beta\}}(|G^{(\bT b_\al b_\beta )}_{c_\al b_\gamma}|^2+ |G^{(\bT b_\al b_\beta)}_{c_\beta b_{\gamma} }|^2)\\
    &+\sum_{x\in \bT}(|G_{c_\al x}^{(b_\al b_\beta)}|^2+|G^{(b_\al b_\beta)}_{c_\beta x}|^2)+N^{-\fb }\Phi,
\end{split}\end{align}
and $N_\gamma=\{x\neq c_\gamma: x\sim b_\gamma\}\cup\{a_\gamma\}$.

After averaging over embeddings of $\cF$ (i.e., over $\bfi$) most sums in
\eqref{e:rcEbound} are controlled by \eqref{eq:rwardex}. For instance, averaging
over $(b_\alpha,c_\alpha)$ and $(b_\gamma,c_\gamma)$ bounds
$\bE\big[|G^{(\bT b_\alpha b_\beta)}_{c_\alpha b_\gamma}|^2\big]$ via Ward identity.

The delicate case is terms such as $|G^{(\bT\bW)}_{c_\alpha x}|^2$ with
$x\in \cN_\alpha$, where $x$ and $c_\alpha$ are graph–neighbors of a common
vertex and the embedding average does not decouple their indices. To handle
these, we introduce a new Ward-type estimate, which we call the
\emph{punctured-vertex Ward bound}:
if $i,j$ are two neighbors of a vertex $o$ (so $o\sim i$ and $o\sim j$), then
\begin{equation}\label{e:Gijo-bound}
\bE\!\left[\,\big|G^{(o)}_{ij}(z)\big|^2\,\right]
\;\leq\; N^{\fo}\,\frac{\bE[\Im m_N(z)]}{N\eta},
\end{equation}
proved in Lemma~\ref{lem:deletedalmostrandom}. The idea is to perform local
resampling at $o$, use the invariance
$\bE[|G^{(o)}_{ij}|^2]=\bE[|\wt G^{(o)}_{ij}|^2]$. We then expand $\wt G_{ij}^{(o)}(z)$ using the Schur complement formula. Crucially, we can bound $\bE[| \wt G_{ij}^{(o)}(z)|^2]$, by $\bE[| G_{ij}^{(o)}(z)|^2]$ times a small factor, and errors as in \eqref{eq:rwardex}, leading to the desired bound given by the right-hand side of \eqref{e:Gijo-bound}.

Combining \eqref{e:replaceexample}, \eqref{e:rcEbound} with Ward’s identity
\eqref{eq:rwardex} and the bound \eqref{e:Gijo-bound}, we obtain
\[
\frac{1}{Z_{\cF}}\sum_{\bfi}\bE\!\left[I(\cF,\cG)\,\bm{1}(\cG,\wt\cG\in\Omega)
\sum_{\alpha,\beta\in\sfA_i}\big(\wt G^{(\bT)}_{c_\alpha c_\beta}-G^{(b_\alpha b_\beta)}_{c_\alpha c_\beta}\big)\right]
=\OO\big((d-1)^{2\ell}\,\bE[\Psi]\big),
\]
and the second line in \eqref{e:Eterm1} is
$\OO\!\big(N^{-\fb/4}\,\bE[\Psi]\big)$. See \Cref{s:change_est} for the full
statement and proof.

%
%

\section{Switching edges,  the forest and admissible functions}\label{s:forest}

\begin{figure}
\centering
\includegraphics[scale=0.15, trim=12cm 0cm 0cm 0cm, clip]{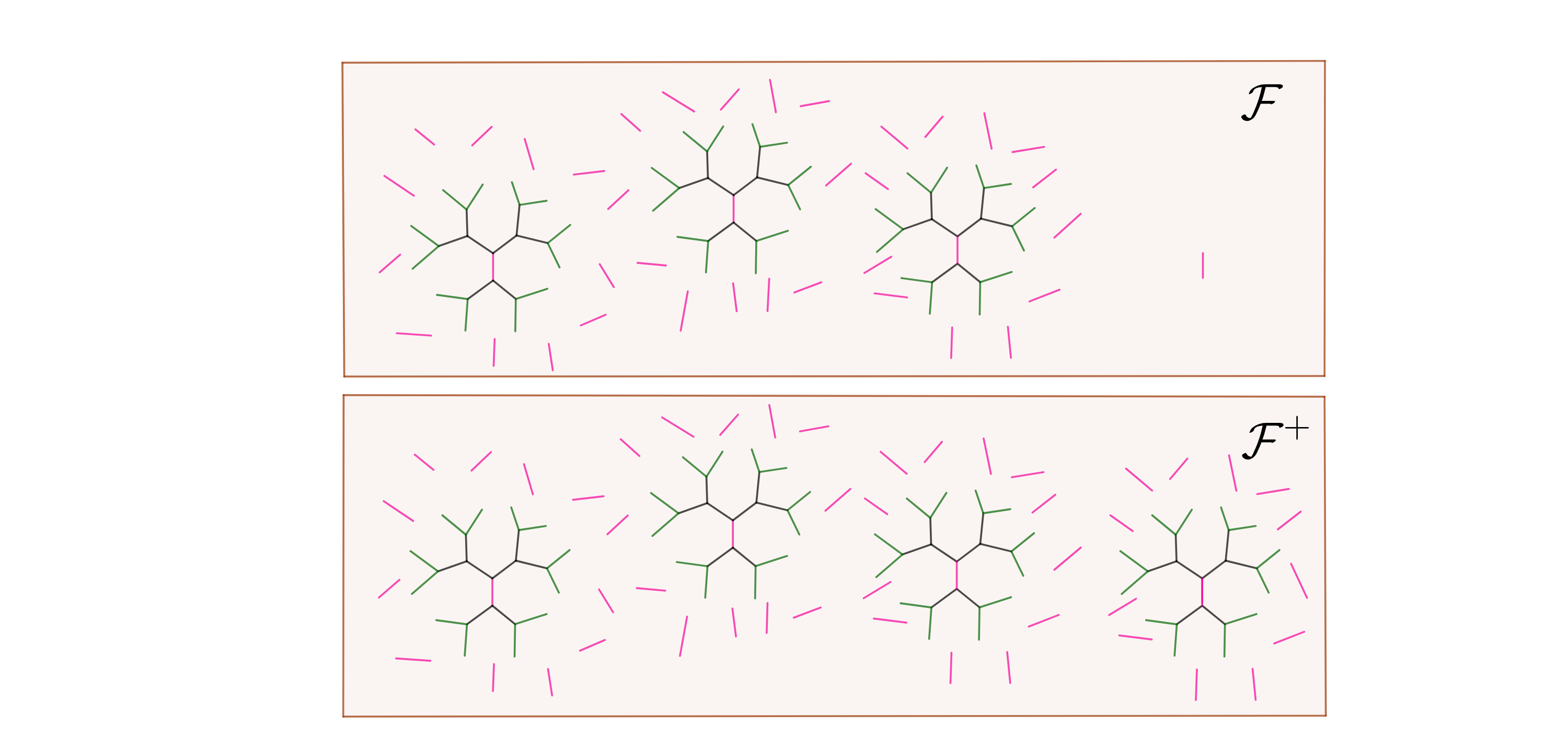}
\caption{Top Panel:
In the forest \(\mathcal{F}\), red edges represent core edges \(\mathcal{C}\). The used core edges belong to radius-\((\ell+1)\) balls, while each unused core edge \(\mathcal{C}^\circ\) forms its own connected component. Together, the red and green edges constitute the switching edges \(\mathcal{K}\).  Bottom Panel: We construct \(\mathcal{F}^+\) from \(\mathcal{F}\) by selecting an unused core edge (the rightmost red edge), expanding it into a radius-\((\ell+1)\) ball, and adding \(\mu\) new switching edges.
\label{fig:forest}}
\end{figure}

We recall from \Cref{p:iteration} that, up to negligible errors,
\[
\bE[\bm1(\cG\in \Omega)(Q-Y)]
\]
decomposes into an $\OO(1)$–weighted sum of terms of the form
\begin{align}\label{e:R_ifirst}
\bE\bigl[I(\cF,\cG)\,\bm1(\cG\in \Omega)\,(G_{c_\al c_\al}^{(b_\al)}-Y)\,R_\bfi\bigr],
\end{align}
where $\cF$ from \eqref{e:buildF} contains all switching edges, and
$R_\bfi$ is a product of factors of the form \eqref{e:rdefcE1}.
To estimate \eqref{e:R_ifirst}, we perform a local resampling around
the edge $(b_\al, c_\al)$ and repeat the procedure from
\Cref{s:expQ}, which naturally leads to an iterative scheme.

At each iteration, we locally resample the graph and express the
Green’s function of the switched graph in terms of the original one as in \Cref{s:expQ}.
Because almost all neighborhoods in $\cG$ are acyclic, with high
probability the edges involved in these resamplings have large tree
neighborhoods and are typically far apart. Hence, we may regard them
collectively as forming a forest. In this section, we formalize this
iterative procedure using a sequence of forests (see
\Cref{fig:forest}), which encode all edges involved in local
resamplings.

\medskip

Fix $\mu := d(d-1)^{\ell}$, the number of boundary edges of the
$d$–regular tree truncated at depth $\ell+1$.
We begin with a two–vertex \emph{template} (an unlabeled graph):
\[
\cF_0 := (V_0,E_0), \qquad
V_0=\{\mathtt i_0,\mathtt o_0\}, \quad
E_0=\{(\mathtt i_0,\mathtt o_0)\}.
\]
Here $\mathtt i_0$ and $\mathtt o_0$ are formal symbols
(distinguished template vertices) that become concrete vertices of
$\cG$ only once an embedding of $\cF$ into $\cG$ is specified.

\paragraph{Initialization.}
Construct $\cF_1$ from $\cF_0$ as follows:
\begin{itemize}
  \item extend the directed edge $(\mathtt i_0,\mathtt o_0)$ to a
        truncated $d$–regular tree $\cT_\ell(\mathtt o_0)$ rooted at
        $\mathtt o_0$;
  \item add $\mu$ boundary edges
        $\cM'_1=\{e'_\alpha\}_{\alpha\in\qq{\mu}}$ to
        $\cT_\ell(\mathtt o_0)$ to obtain
        $\cT_{\ell+1}(\mathtt o_0)$;
  \item introduce $\mu$ new directed edges
        $\cM_1=\{e_\alpha\}_{\alpha\in\qq{\mu}}$.
\end{itemize}
Thus,
\begin{equation}\label{eq:forestdef1}
  \cF_1=\cT_{\ell+1}(\mathtt o_0)\cup \cM_1, \qquad
  \cT_\ell(\mathtt o_0)\cup\cM'_1=\cT_{\ell+1}(\mathtt o_0).
\end{equation}
The subgraph $\cF$ in \eqref{e:buildF} can be viewed as an embedding
of $\cF_1$ into $\cG$.

\paragraph{General step.}
Given $\cF_s=(V_s,E_s)$ and the sets $\cM'_s,\cM_s$ created at the
previous steps, construct $\cF_{s+1}$ by
\begin{itemize}
  \item choosing a (previously created) core edge
        $e=(\mathtt i_s,\mathtt o_s)\in\cM_s$ and extending it to
        $\cT_\ell(\mathtt o_s)$;
  \item adding $\mu$ boundary edges
        $\cM'_{s+1}=\{e'_\alpha\}_{\alpha\in\qq{\mu}}$ to get
        $\cT_{\ell+1}(\mathtt o_s)$;
  \item adding $\mu$ new directed edges
        $\cM_{s+1}=\{e_\alpha\}_{\alpha\in\qq{\mu}}$.
\end{itemize}
Explicitly,
\begin{equation}\label{eq:forestdef2}
  \cF_{s+1}=\cF_s\cup \cT_{\ell+1}(\mathtt o_s)\cup \cM_{s+1},\quad   \cT_\ell(\mathtt o_s)\,\cup\,\cM'_{s+1}=\cT_{\ell+1}(\mathtt o_s).
\end{equation}
We refer to \Cref{fig:forest} for the construction of $\cF_{s+1}=\cF^+$ from $\cF_s=\cF$.
\paragraph{Embedding into $\cG$.}
Local resampling (Section~\ref{s:local_resampling}) yields an embedding
of $\cF_0\subset\cF_1\subset\cdots\subset\cF_{s+1}$ into $\cG$.
Under an embedding we write
$\cF_s=\cF_{\bfi_s}$ with vertex tuple
$\bfi_s\in \qq{N}^{|V_s|}$;
$\mathtt i_s\mapsto i_s$ and $\mathtt o_s\mapsto o_s$.
The neighborhood $\cT_\ell(\mathtt o_s)$ maps to the ball
$B_\ell(o_s;\cG)$; boundary edges $e'_\alpha,e_\alpha$ map to
$(\ell_\alpha,a_\alpha)$ and $(b_\alpha,c_\alpha)$ in $\cG$.

\paragraph{Core and switching edges.}
Let
\[
\cK_s:=\{(\mathtt i_0,\mathtt o_0)\}\cup
\cM_1\cup \cM'_1\cup\cdots\cup \cM_s\cup \cM'_s
\]
be the set of switching edges, and
\[
\cC_s:=\{(\mathtt i_0,\mathtt o_0)\}\cup \cM_1\cup\cdots\cup \cM_s
\]
the core edges (each component of $\cF_s$ contains exactly one core edge).
Write the used core edges as
$\{(\mathtt i_0,\mathtt o_0),\dots,(\mathtt i_{s-1},\mathtt o_{s-1})\}$ and
$\cC_s^\circ:=\cC_s\setminus
\{(\mathtt i_0,\mathtt o_0),\dots,(\mathtt i_{s-1},\mathtt o_{s-1})\}$
for the unused ones.

For most arguments we freeze a step and abbreviate
\[
\cF=\cF_s,\quad \cK=\cK_s,\quad \cC=\cC_s,\quad \cC^\circ=\cC_s^\circ,
\qquad
\cF^+=\cF_{s+1},\ \cK^+=\cK_{s+1},\ \cC^+=\cC_{s+1},\
(\cC^\circ)^+=\cC^\circ_{s+1}.
\]
In general,
\begin{equation}\label{e:cFtocF+}
  \cF=\!\!\bigcup_{e\in\cC^\circ}\!\{e\}
  \ \cup\!\!\!\!\bigcup_{(\mathtt i',\mathtt o')\in \cC\setminus \cC^\circ}\!\!\!\!
  \cT_{\ell+1}(\mathtt o'),
  \qquad
  \cK=\cC\ \cup\
  \!\!\!\!\bigcup_{(\mathtt i',\mathtt o')\in \cC\setminus \cC^\circ}\!\!\!\!
  \big(\cT_{\ell+1}(\mathtt o')\setminus \cT_{\ell}(\mathtt o')\big).
\end{equation}
If we expand an unused core edge $(\mathtt i,\mathtt o)\in\cC^\circ$ by
$\cT_{\ell+1}(\mathtt o)=\cT_\ell(\mathtt o)\cup\cM'$, where
$\cM'=\{e'_\alpha\}_{\alpha\in\qq{\mu}}$, and add
$\cM=\{e_\alpha\}_{\alpha\in\qq{\mu}}$, then
\begin{equation}\label{e:cF++}
  \cC^+=\cC\cup\cM,\qquad
  (\cC^\circ)^+=\cC^\circ\cup\cM\setminus\{(\mathtt i,\mathtt o)\},\qquad
  \cF^+=\cF\cup\cT_{\ell+1}(\mathtt o)\cup\cM,\qquad
  \cK^+=\cK\cup\cM\cup\cM'.
\end{equation}
We refer to \Cref{fig:forest} for $\cF$ and $\cF^+$.

\paragraph{Indicator for “good” embeddings.}

For most of this paper we view $\cF=(V, E)$ as an embedded (labeled) copy of the
template in $\cG$; this amounts to choosing a vertex assignment
$\bfi\in \qq{N}^{|V|}$. Different choices of $\bfi$
yield different embeddings. Under this convention, $\cK$, $\cC$, and
$\cC^\circ$ are also regarded as embedded subgraphs of $\cG$. With a slight abuse of notation, we write $\cF=(\bfi,E)$ to emphasize that we
have fixed an embedding of $\cF$ into $\cG$.  As discussed
above, one local resampling step produces a new embedded forest
$\cF^+=\cF^+({\bfi^+}, E^+)\subset \cG$ with vertex tuple $\bfi^+$.

The forest $\cF$ records the edges involved in all previous local resamplings.
To ensure these edges are well separated and lie in large tree neighborhoods, we
use the following indicator.

\begin{definition}\label{def:indicator}
Let $\cF=(\bfi,E)$ (as in \eqref{e:cFtocF+}) with core edges $\cC$, viewed as a
subgraph of a $d$-regular graph $\cG$ with adjacency matrix
$A$. Define $I(\cF,\cG)=1$ if and only if:
(i) $\cF$ is embedded in $\cG$; (ii) for every $(b,c)\in\cC$ and every
$x\in \cB_\ell(c;\cG)$, the ball $\cB_{\fR}(x;\cG)$ is a tree; and
(iii) distinct core edges are at pairwise distance at least $3\fR$.
Equivalently,
\begin{align}\label{e:defI}
I(\cF,\cG)
:= \prod_{\{x,y\}\in E} A_{xy}
   \;\prod_{(b,c)\in \cC}\;\prod_{x\in \cB_\ell(c;\cG)}
      \bm{1}\!\big(\cB_{\fR}(x;\cG)\ \text{is a tree}\big)
   \; \prod_{\substack{(b,c),(b',c')\in \cC\\ (b,c)\neq (b',c')}}      \bm{1}\!\big(\dist_\cG(c,c')\ge 3\fR\big).
\end{align}
\end{definition}

Fix a good embedding of $\cF=(\bfi, E)$ into $\cG\in \Omega$ with $I(\cF, \cG)=1$. Then we perform a local resampling around an unused core edge $(i,o)$. 
In the following lemma, we show that with high probability with respect to the randomness of $\bfS$, the randomly selected edges $(b_\al, c_\al)$ are far away from each other, and have large tree neighborhood. In particular $\wt \cG=T_\bfS \cG\in \oOmega$.
\begin{lemma}\label{lem:configuration}
Fix a $d$-regular graph $\cG\in \oOmega$, and a forest $\cF=(\bfi, E)$ (as in \eqref{e:cFtocF+}) viewed as a subgraph of $\cG$. 
Assume that $I(\cF,\cG)=1$ and $|\bfi|\leq N^{\fc/2}$. 
We consider the local resampling around an unused core edge $(i,o)\in \cC^\circ$, with resampling data $\{(l_\al, a_\al), (b_\al, c_\al)\}_{\al\in \qq{\mu}}$. 
We denote the set of resampling data $\sfF(\cG)\subset \sfS(\cG)$ (recall from \Cref{s:local_resampling}) such that the following holds 
\begin{enumerate}
    \item 
for any $\al\neq \beta\in \qq{\mu}$, $\dist_\cG(\{b_\al, c_\al\}\cup \bfi,\{b_\beta, c_\beta\} )\geq3\fR$; 
\item 
for any  $v\in \cB_{\ell}(\{b_\al, c_\al\}_{\al\in \qq{\mu}}, \cG)$, the radius $\fR$ neighborhood of $v$ is a tree.
\end{enumerate}
Then $\bP_\bfS(\sfF(\cG))\geq 1-N^{-1+2\fc}$ (where $\bP_{\bfS}(\cdot)$ is the probability with respect to the randomness of $\bfS$ as in \Cref{def:PS}). Also, for $\bfS\in \sfF(\cG)$ the following holds
\begin{enumerate}
    \item $\mu=d(d-1)^\ell$, ${\mathsf W}_{\bf S}=\qq{\mu}$ (recall from \eqref{Wdef}), and $\tcG=T_\bfS(\cG)\in \oOmega$; 
\item $I(\cF^+, \cG)=1$ and $I(\cF, \wt\cG)=1$.
\end{enumerate}
\end{lemma}

\begin{proof}[Proof of \Cref{lem:configuration}]
    We sequentially select $(b_\al, c_\al)$ uniformly random from $\cG^{(\bT)}$. For any fixed $\alpha$, we consider all edges that would break the requirements of the lemma. For the first requirement, we have
    \begin{align}\label{e:small1}
        \bP_{\bfS}(\dist_\cG(\bfi\cup_{1\leq \beta\leq \al-1}\{b_\beta, c_\beta\},\{b_\al, c_\al\} )\leq3\fR)\lesssim N^{-1}(|\bfi|+2\mu)d(d-1)^{3\fR}\leq N^{-1+3\fc/2}.
    \end{align}
    For the second requirement, we recall that  $\cG\in \oOmega$, in which all vertices except for $N^\fc$ many have radius $\fR$ tree neighborhood. Thus  
    \begin{align}\label{e:small2}
        \bP_{\bfS}(\cB_\fR(v,\cG) \text{ is not a tree for some } v\in \cB_\ell(\{b_\al, c_\al\}, \cG))\leq N^{-1} N^\fc d(d-1)^\ell\leq N^{-1+3\fc/2}.
    \end{align}
   The claim $\bP_\bfS(\sfF(\cG))\geq 1-N^{-1+2\fc}$ follows from union bounding over all $\alpha$ using \eqref{e:small1} and \eqref{e:small2}. 
   
    Under our assumption $I(\cF,\cG)=1$, the radius $\fR$ neighborhood of $o$ is a tree. Thus $\mu=d(d-1)^\ell$. Moreover, the neighborhoods $\cB_{3\fR/2}(o, \cG)$, and $\cB_{\fR}(\{b_\al, c_\al\}, \cG)$ for $\al\in \qq{\mu}$ are disjoint. It follows that $\dist_{\cG^{(\bT)}}(\{a_\al,b_\al,c_\al\}, \{a_\beta,b_\beta,c_\beta\})> {\fR/4}$ for all $\al\neq \beta\in \qq{\mu}$, and the subgraph $\cB_{\fR/4}(\{a_\al, b_\al, c_\al\}, \cG^{(\bT)})$ after adding the edge $\{a_\al, b_\al\}$ is a tree for all $\al \in \qq{\mu}$. We conclude that ${\mathsf W}_{\bf S}=\qq{\mu}$.

    Next we show that for any vertex $v\in \qq{N}$, the excess of $\cB_{\fR}(v, \tcG)$ is no bigger than that of $\cB_{\fR}(v, \cG)$. Then it follows that $\tcG\in \oOmega$. 
    If $\dist(v, \{a_\al, b_\al, c_\al\}_{\al\in \qq{\mu}})\geq \fR$, then $\cB_{\fR}(v, \tcG)=\cB_{\fR}(v, \cG)$, and the statement follows.
    Otherwise either $v\in \cB_{\fR}(\{a_\al\}_{\al\in \qq{\mu}}, \cG)\subset\cB_{3\fR/2}(o, \cG)$ or $v\in \cB_{\fR}(\{b_\al, c_\al\}, \cG)$ for some $\al\in \qq{\mu}$. We will discuss the first case. The second case can be proven in the same way, so we omit its proof. 
    If $v\in \cB_\ell(o, \cG)$, we denote $r:=\min_{\al \in \qq{\mu}}\dist_\cG(v, \{l_\al\})\leq \fR$. Then $\cB_{\fR}(v, \tcG)$ is a subgraph of 
    $\cB_{\fR}(v, \cG)\cup_{\al\in \qq{\mu}} \cB_{\fR-r-1}(c_\al,\cG)$ after removing $\{(b_\al,c_\al)\}_{\al\in \qq{\mu}}$ and adding $\{(l_\al, c_\al)\}_{\al\in \qq{\mu}}$. By our construction of $\sfF(\cG)$, $\cB_{\fR-r-1}(c_\al,\cG)$ are disjoint trees. We conclude that $\cB_{\fR}(v, \tcG)$ is a tree.
    If $v\not\in \cB_\ell(o, \cG)$,
    we denote $r:=\min_{\al \in \qq{\mu}}\dist_\cG(v, a_\al)\leq \fR$, then $\cB_{\fR}(v, \tcG)$ is a subgraph of 
    $\cB_{\fR}(v, \cG)\cup_{\al\in \qq{\mu}} 
    \cB_{\fR-r-1}(c_\al,\cG)$ 
    after removing $\{(b_\al,c_\al)\}_{\al\in \qq{\mu}}$ and adding $\{(a_\al, b_\al)\}_{\al\in \qq{\mu}}$. 
   Again by our construction of $\sfF(\cG)$, $\cB_{\fR-r-1}(c_\al,\cG)$ are disjoint trees, we conclude the excess of $\cB_{\fR}(v, \tcG)$ is at most that of $\cB_{\fR}(v, \cG)$.

The claim $I(\cF^+, \cG)=1$ follows from the construction of $\sfF(\cG)$. It also follows from  the above discussion that $\cB_\fR(v,\tcG)$ is a tree for any $v\in \cB_\ell(o, \cG)$. One can then check that $I(\cF,\tcG)=1$. This finishes the proof of the second statement in  \Cref{lem:configuration}.

\end{proof}

\paragraph{Number of ``good" embeddings.}
From \eqref{e:cFtocF+}, each connected component of $\cF$ is either an unused core edge or a radius-$(\ell+1)$ ball corresponding to a used core edge. The following proposition states that the total number of embeddings where $I(\cF,\cG)=1$ is approximately equal to that of choosing each connected component independently.

\begin{proposition}\label{p:sumA}
Given a template $\cF=(V, E)$ with core edges $\cC$ and unused core edges $\cC^\circ$ as in \eqref{e:cFtocF+}, as well as a $d$-regular graph $\cG\in \oOmega$, we have 
\[
\sum_{\bfi} I(\cF,\cG)=Z_{\cF}\left(1+\OO\left(\frac{1}{N^{1-2\fc}}\right)\right),
\]
where
\begin{align}\label{e:sumA2}\begin{split}
Z_{\cF}:=(Nd)^{|\cC|}\left([(d-1)!]^{1+d+d(d-1)+\cdots+d(d-1)^{\ell-1}}\right)^{|\cC\setminus\cC^\circ|}.
\end{split}\end{align}
Here $|\cC|$ is the number of core edges; and $|\cC\setminus\cC^\circ|$ is the number of used core edges. We remark that $Z_\cF$ depends only on the template $\cF$ but not $\cG$.
\end{proposition}

\begin{remark}
In the rest of this article, we have many expressions in the following form 
\begin{align}
\frac{1}{Z_\cF} \sum_\bfi \bE[I(\cF, \cG)\times (\cdots )].
\end{align}
Thanks to \Cref{p:sumA}, the above expression can be viewed as an average over all possible embedding of the template $\cF$ into $\cG$. 
\end{remark}

\begin{proof}
We notice that $|\cC|$ is also the number of connected components of $\cF$, and $|\cC\setminus\cC^\circ|$ is the number of connected components in $\cF$ which are balls of radius $\ell+1$.

We can prove \eqref{e:sumA2} by 
induction on the number of connected components. If $\cF$ consists of a single edge $\cF=\{b,c\}$ which is an unused core edge, then 
\begin{align}\label{e:single_edge}
    \sum_{\bfi}I(\cF,\cG)=\sum_{b,c} A_{bc}\prod_{v\in \cB_\ell(c,\cG)}\bm1(\cB_{\fR}(v, \cG) \text{ is a tree}) )=Nd\left(1+\OO\left(\frac{1}{N^{1-3\fc/2}}\right)\right),
\end{align}
where we used the definition of $\overline\Omega$ from \Cref{def:omegabar}. 
If $\cF$ consists of a radius $(\ell+1)$-ball, corresponding to one used core edge, then we can also first sum over its core edge. The number of choices of this is the same as \eqref{e:single_edge}.  Then we sum over the remaining vertices. Each interior vertex of the radius-$(\ell+1)$ ball contributes a factor $(d-1)!$, since there are $(d-1)!$ ways to embed its children vertices. We get
\begin{align}\label{e:ball}
    \sum_{\bfi}I(\cF,\cG)=Nd[(d-1)!]^{1+d+d(d-1)+\cdots+d(d-1)^{\ell-1}}\left(1+\OO\left(\frac{1}{N^{1-3\fc/2}}\right)\right).
\end{align}

If the statement holds for $\cF$ with $\theta$ connected components, next we show it for $\cF$ with $\theta+1$ connected components. We can first sum over the indices corresponding to a connected component, fixing the other indices. 
 If it is a single edge, we get a factor similar to \eqref{e:single_edge}; if it is a radius-$(\ell+1)$ ball, we get a factor similar to \eqref{e:ball}. Next we can sum over the remaining $\theta$ connected components of $\cF$, which gives \eqref{e:sumA2}.

\end{proof}

\subsection{Admissible functions} \label{s:admissible}
Later, we repeatedly localize the Green's function on the embedded forest
$\cF\subset \cG$ and separate the \emph{tree-like} main term from the
\emph{fluctuation}. Concretely, for vertices $w,w'\in\bfi$ we approximate the
global Green's function entry $G_{ww'}(z)$ by a deterministic local kernel
$L_{ww'}(z,\cF,\cG)$, which coincides with the Green's function on copies of the
$d$-regular tree whenever the neighborhood of $\cF$ is cycle-free. The error
$G^\circ:=G-L$ will be treated perturbatively. To organize the resulting
expressions, we introduce a class of \emph{admissible functions}, namely products
of resolvent factors associated with core and switching edges.

\begin{definition}[Local Green's Function]\label{d:localGF}
  Given a forest $\cF=(\bfi, E)$ embedded in $\cG$, we introduce the local Green's function $L(z,\cF,\cG)$: for  $w,w'\in \bfi$, \begin{align}\begin{split}\label{e:local_Green}
    &L_{ww'}(z,\cF,\cG)=P_{ww'}(\cF,z,\msc(z)).
\end{split}\end{align}
Given the event $I(\cF,\cG)=1$, $L_{ww'}$ is simply the Green's function of copies of $d$-regular trees (recall from \Cref{greentree}): $L_{ww'}=0$ if $w,w'$ are disconnected in $\cF$, otherwise, 
\begin{align*}
L_{ww'}=\md(z)\left(-\frac{\msc(z)}{\sqrt{d-1}}\right)^{\dist_{\cF}(w,w')}.
\end{align*} 

We also denote the centered version of the Green's function as
\begin{align}\label{e:G-L}
G_{ww'}^\circ(z)=G_{ww'}(z)-L_{ww'}(z,\cF,\cG).
\end{align}
When the context is clear, we will simply write $G^\circ(z), L(z,\cF,\cG)$ as $G^\circ, L$ for simplicity.

\end{definition}

\begin{remark}

 At each step, we expand the forest $\cF=(\bfi, E)$ to a new forest $\cF^+=(\bfi^+, E^+)$ by including local resampling data ($\cF\subset \cF^+$ are embedded subgraph of $\cG$). Given the event $I(\cF^+,\cG)=1$, the local Green's functions are compatible (both are given by the Green's function of copies of $d$-regular trees)
     \begin{align*}
        L_{ww'}(z,\cF,\cG)=L_{ww'}(z,\cF^+,\cG),\quad w,w'\in \bfi.
            \end{align*}

Later, we need the local Green's function with one vertex removed: Let $(i,o)\in \cC\setminus\cC^\circ$, and recall $P^{(i)}$ from \eqref{e:defPi},
\begin{align}\label{e:defLi}
    L^{(i)}_{ww'}:= L^{(i)}_{ww'}(z,\cF, \cG):=P^{(i)}_{ww'}(\cF,z,\msc(z)).
\end{align}
\end{remark}

\begin{definition}[Admissible Function]
    \label{def:pgen}
Consider a forest $\cF=(\bfi, E)$ as defined in \eqref{e:cFtocF+}, with switching edges $\cK$, core edges $\cC$, unused core edges $\cC^\circ$. For any nonnegative integers $r\geq0$, we denote the set of admissible functions $\Adm(r, \cF,\cG)$ where a function $R_{\bfi}\in \Adm(r,\cF,\cG)$ contains $r$ factors of the form 
\begin{align}\begin{split}\label{e:defcE1}
    &\{(G_{c c}^{(b)}-Q)\}_{(b,c)\in \cC^\circ}, \quad  \{ G_{c c'}^{(bb')}, G_{b c'}^{(b')}, G_{bb'}\}_{(b,c)\neq (b',c')\in \cC^\circ},\quad \{ G^{\circ}_{ss'}\}_{s,s'\in \cK},\quad (Q-\msc(z)).
\end{split}\end{align}
\end{definition}  

\begin{remark}
    At each step, we expand the forest $\cF$ to a new forest $\cF^+$ by including local resampling data. This change also affects the admissible set of functions, which now expands as follows:
    \begin{align*}
        \Adm(r,\cF,\cG)\subset \Adm(r,\cF^+,\cG).
    \end{align*}
\end{remark}

We now give the general ways of bounding the terms involved in the admissible functions (recall from \Cref{def:pgen}). 

\begin{proposition}\label{p:small_Ri}
     We take $z\in {\bf D}$ (recall from \eqref{e:D}). Then the following holds:
     \begin{enumerate}
     \item For any factor $B$ in \eqref{e:defcE1}, we have
 \begin{align}\label{e:Bsmall}
       \bm1(\cG\in \Omega)I(\cF, \cG) |B|\lesssim N^{-\fb}.
    \end{align} 

     \item  Let $R_\bfi\in \Adm(r,\bmr,\cF,\cG)$ as in \Cref{def:pgen}.
      If $R_\bfi$ contains two terms of the form $\{G_{cc'}^{(bb')},G_{bc'}^{(b')}, G_{bb'}\}$ with $(c,b)\neq (c',b')\in \cC^\circ$, then
    \begin{align}\label{e:refined_bound}
        \frac{(d-1)^{3\ell r}}{Z_\cF}\sum_{\bfi}\bm1(\cG\in \Omega)I(\cF, \cG)R_\bfi =\OO\left(\frac{(d-1)^{3\ell r}N^\fo}{N^{(r-2)\fb} } \Phi\right).
    \end{align}
     If we further assume $r\geq 3$, then the expectation of \eqref{e:refined_bound} is bounded by  $\OO(N^{-\fb/2}\bE[\Psi_p])$;
    \end{enumerate}
    \end{proposition}
\begin{proof}[Proof of \Cref{p:small_Ri}]
The claim \eqref{e:Bsmall} follows from the definition \eqref{eq:infbound} of the set \(\Omega\) and can be proved in the same way as \Cref{l:basicG}. Thus, we omit its proof.
For \eqref{e:refined_bound}, we show the case that $R_\bfi$ contains $(G_{cc'}^{(bb')})^2$. The other cases can be proven in the same way so we omit. We can first bound the other factors of $R_\bfi$ using \eqref{e:Bsmall}, as $N^{-(r-2)\fb}$. Then \eqref{e:refined_bound} reduces to
\begin{align}
        \frac{(d-1)^{3\ell r}}{Z_\cF}\sum_{\bfi}\bm1(\cG\in \Omega)I(\cF, \cG)|R_\bfi |\lesssim \frac{(d-1)^{3\ell r}}{N^{(r-2)\fb} Z_\cF}\sum_{\bfi}\bm1(\cG\in \Omega)I(\cF, \cG)|G_{cc'}^{(bb')}|^2\lesssim \frac{(d-1)^{3\ell r}\Phi}{N^{(r-2)\fb} Z_\cF},
\end{align}
which follows from the same argument as \eqref{e:Gccerror}.
\end{proof}
\section{Switching using 
the Schur complement formula}\label{sec:schurlemma}
In this section we will use the Schur complement formula to study the Green's function after local resampling. 
We recall the local resampling and related notation from \Cref{s:local_resampling}. We also introduce the following S-Product term.
\begin{definition}[S-Product term]\label{d:S-product}
    Fix $r\geq 0$, we define $R_r$ to be a \emph{S-product term} of order $r$ (where ``\emph{S}" indicates that these terms arise from expansions using the Schur complement formula) if it is a product of $r$ factors in the following forms: 
\begin{align*}
   (G_{c_\al c_\al}^{(b_\al)}-Q), \quad  G_{c_\al c_{\beta}}^{(b_\al b_\beta)},\quad  (Q-\msc(z)),\quad \al\neq \beta\in \qq{\mu}.
\end{align*} 
\end{definition}

In the following proposition, we derive an expansion for factors that are Green's function entries with at most one index in $\{i,o\}$. They can be proven in exactly the same way as \Cref{l:coefficient} using the Schur complement formula, so we omit their proofs.
\begin{proposition}\label{lem:offdiagswitch}
We assume that $\cG, \widetilde \cG\in \Omega$ and $I(\cF^+,\cG)=1$ (recall from \eqref{e:defI}), and define the index set $\sfA_i := \{ \alpha \in \qq{\mu} : \dist_{\cT}(i, l_\al) = \ell+1 \}$ (see \Cref{fig:Ai}). Then for any unused core edges $(b,c)\neq (b',c')\in \cC^\circ\setminus\{ (i,o)\}$, the following holds:
\begin{enumerate}
    \item  $\widetilde G^{(ib)}_{oc}$ and $\widetilde G^{(b)}_{ic}$ can be rewritten as a weighted sum 
\begin{align}\label{e:tGoc_exp}
  \frac{1} {(d-1)^{\ell/2}}\sum_{\al\in \qq{\mu}}\fc_1(\bm1(\al\in \sfA_i))G_{c_\al c}^{(b_\al b)}
   +\cU  +\cE,
\end{align}
where $|\fc_1(\cdot)|\lesssim 1$; $\cU$ is an $\OO(1)$-weighted sum of terms of the form $(d-1)^{3r\ell}R_r G_{c_\al c}^{(b_\al b)}$, for $R_r$ an S-product term (see \Cref{d:S-product}) with $r\geq 1$, and the error $\cE$ is bounded by
\begin{align}\label{e:tGoc_cE}
    |\cE|\lesssim \sum_{\al\in \qq{\mu}} |\wt G_{c_\al c}^{(\bT b)}-G_{c_\al c}^{(b_\al b)}|+\sum_{\al, \beta\in \qq{\mu}}|\wt G_{c_\al c_\beta}^{(\bT)}-G^{(b_\al b_\beta)}_{c_\al c_\beta}|+N^{-2}.
\end{align}
  \item  $\widetilde G_{ib}$ and $\wt G_{ob}^{(i)}$ can be rewritten as a weighted sum 
\begin{align*}
  \frac{1} {(d-1)^{\ell/2}}\sum_{\al\in \qq{\mu}}\fc_1(\bm1(\al\in \sfA_i))G^{(b_\al)}_{c_\al b }
   +\cU  +\cE,
\end{align*}
where $|\fc_1(\cdot)|\lesssim 1$; $\cU$ is an $\OO(1)$-weighted sum of terms of the form $(d-1)^{3r\ell}R_r G_{c_\al b}^{( b_\al)}$, where $R_r$ is an S-product term (see \Cref{d:S-product}) with $r\geq 1$;
and the error $\cE$ is bounded by
\begin{align*}
    |\cE|\lesssim \sum_{\al\in \qq{\mu}} |\wt G_{c_\al b}^{(\bT )}-G_{c_\al b}^{(b_\al )}|+\sum_{\al, \beta\in \qq{\mu}}|\wt G_{c_\al c_\beta}^{(\bT)}-G^{(b_\al b_\beta)}_{c_\al c_\beta}|+N^{-2}.
\end{align*}
\item  $\wt G_{cc}^{(b)}, \wt G_{cc'}^{(bb')},\wt G_{bc'}^{(b')}$ and $ \wt G_{bb'}$ can be rewritten as
\begin{align*}
&\wt G_{cc}^{(b)}=G_{cc}^{(b)}+\cE,\quad 
    |\cE|\lesssim |\wt G^{(\bT b)}_{cc}-G^{( b)}_{cc}|
    +(d-1)^\ell\sum_{\al\in \qq{\mu}} |\wt G_{c_\al c}^{(\bT b)}|^2,
\\
&\wt G_{cc'}^{(bb')}=G_{cc'}^{(bb')}+\cE,\quad
    |\cE|\lesssim |\wt G^{(\bT bb')}_{cc'}-G^{( bb')}_{cc'}|
    +(d-1)^\ell\sum_{\al\in \qq{\mu}} (|\wt G_{c_\al c}^{(\bT bb')}|^2+|\wt G_{c'c_\al}^{(\bT bb')}|^2),
    \\
&\wt G_{bc'}^{(b')}=G_{bc'}^{(b')}+\cE,\quad
    |\cE|\lesssim |\wt G^{(\bT b')}_{bc'}-G^{( b')}_{bc'}|
    +(d-1)^\ell\sum_{\al\in \qq{\mu}} (|\wt G_{bc_\al}^{(\bT b')}|^2+|\wt G_{c'c_\al}^{(\bT b')}|^2),
\\
&\wt G_{bb'}=G_{bb'}+\cE,\quad
    |\cE|\lesssim |\wt G^{(\bT)}_{bb'}-G_{bb'}|
    +(d-1)^\ell\sum_{\al\in \qq{\mu}} (|\wt G_{bc_\al}^{(\bT)}|^2+|\wt G_{b'c_\al}^{(\bT)}|^2).
\end{align*}
\end{enumerate}
\end{proposition}

\section{Switching using the Woodbury formula}\label{sec:fanalysis}
In this section, we introduce a novel expansion based on the Woodbury formula \eqref{e:woodbury}.
In the rest of this section, we assume that $I(\cF^+, \cG)=1$ (recall from \eqref{e:defI}). Then the switching edges $(b_\al, c_\al)_{\al\in \qq{\mu}}$ are far away from each other, and have large tree neighborhood. Thanks to \Cref{lem:configuration},  $I(\cF^+, \cG)=1$ holds with high probability provided $I(\cF, \cG)=1$.

We compare the normalized adjacency matrix of the switched graph to that of the original graph, $\wt H-H$.
We recall from \eqref{e:H-H}
\begin{align*}
   \wt H-H=-\sum_{\al\in\qq{\mu} }\xi_\al,\quad  \xi_\al:=\frac{1}{\sqrt{d-1}}\left(\Delta_{l_\al a_\al}
    +\Delta_{b_\al c_\al}
    -\Delta_{l_\al c_\al}-\Delta_{a_\al b_\al}\right).
\end{align*}
We denote the rank of this difference as $r=\OO((d-1)^\ell)$, and rewrite 
\begin{align*}
    \wt H-H=UV^\top,
\end{align*}
where $U, V$ are $N\times r$ matrices, and their nonzero rows correspond to the vertices $\{l_\al, a_\al, b_\al, c_\al\}_{\al\in \qq{\mu}}$.
Then, the Woodbury formula \eqref{e:woodbury} gives us 
\begin{align}\label{e:tGG}
\tG-G=(H-z+UV^\top)^{-1}-(H-z)^{-1}=-GU(\mathbb I+V^\top G U)^{-1}V^\top G.
\end{align}

We recall $\cF^+$ as in \eqref{eq:forestdef2}, and denote by $\wt \cF^+$ the switched version of it
\begin{align*}
  \cF^+:=\cF\cup \cB_{\ell}(o,\cG)\cup\bigcup_{\al=1}^\mu \{(l_\alpha,a_\alpha), (b_\alpha, c_\alpha)\},\quad   \wt \cF^+:=\cF\cup \cB_{\ell}(o,\cG)\cup\bigcup_{\al=1}^\mu \{(l_\alpha,c_\alpha), (a_\alpha, b_\alpha)\}.
\end{align*} 
We view $\cF^+, \wt\cF^+$ as subgraphs of $\cG, \wt\cG$ respectively.  We will analyze \eqref{e:tGG} using local Green's functions
 \begin{align}\label{e:defPtP}
   L:=P(\cF^+,z ,\msc(z )), \quad
   \widetilde L:=P(\wt\cF^+,z ,\msc(z )),
\end{align} 
as was defined in \Cref{def:pdef}. We remark that condition on $I(\cF^+, \cG)=1$, both $L$ and $\wt L$ are simply the Green's function of copies of $d$-regular trees.

Notice that when restricted to the vertex set of $\cF^+$ (which contains the vertices $\{l_\al, a_\al, b_\al, c_\al\}_{\al\in \qq{\mu}}$), 
\begin{align}\label{e:inverseL}
    \widetilde L^{-1}-L^{-1}=\widetilde H-H=-\sum_{\al\in\qq{\mu} }\xi_\al=UV^\top.
\end{align} 
We can use the Woodbury formula on $ L, \widetilde L$ as well, giving
\begin{align}\label{e:tP-P}
    \widetilde L-L=-LU(\mathbb I+V^\top L U)^{-1}V^\top L.
\end{align}

A crucial observation is that the quantity $-U(\mathbb I+V^\top L U)^{-1}V^\top$ in \eqref{e:tGG} and \eqref{e:tP-P} take very simple form.
\begin{lemma}\label{l:defF2}
We introduce the following matrix $F$, which is nonzero on the vertex set $\{l_\al, a_\al, b_\al, c_\al\}_{\al\in \qq{\mu}}$,
\begin{align}\label{e:defF}
F:=\sum_{\al \in \qq{\mu}} \xi_{\al}+\sum_{\al, \beta\in \qq{\mu}} \xi_{\al}\tL \xi_{\beta}.
\end{align}
Then 
\begin{align}\label{e:defF2g}
    F=-U(\mathbb I+V^\top L U)^{-1}V^\top.
\end{align}
\end{lemma}

\begin{proof}[Proof of \Cref{l:defF2}]
The nonzero rows of $U, V$ are parametrized by $\{l_\alpha, a_\alpha, b_\alpha, c_\alpha\}_{\al\in \qq{\mu}}$.
By rearranging the above expression \eqref{e:tP-P} (we view all the matrices as restricted on the vertex set of $\cF^+$), we get
\begin{align}\label{e:Fz}
L^{-1}\tL L^{-1}-L^{-1}=-U(\mathbb I+V^\top L U)^{-1}V^\top.
\end{align}
We can reorganize \eqref{e:Fz} as
\begin{align}\begin{split}\label{e:defF2}
    &\phantom{{}={}}-U(\mathbb I+V^\top L U)^{-1}V^\top =L^{-1}\tL L^{-1}-L^{-1}=L^{-1}\tL \tL^{-1}+L^{-1}\tL (L^{-1}-\tL^{-1})-L^{-1}\\
    &=L^{-1}\tL (L^{-1}-\tL^{-1})
    =(L^{-1}-\tL^{-1})\tL (L^{-1}-\tL^{-1})+\tL^{-1}\tL (L^{-1}-\tL^{-1})\\
     &=(L^{-1}-\tL^{-1})+(L^{-1}-\tL^{-1})\tL (L^{-1}-\tL^{-1})=\sum_{\al \in \qq{\mu}} \xi_{\al}+\sum_{\al, \beta\in \qq{\mu}} \xi_{\al}\tL \xi_{\beta}=F,
\end{split}\end{align}
where in the last statement we used  \eqref{e:inverseL}.
\end{proof}

Our next lemma attempts to expand $\tG-G$ in terms of $\tL-L$.

\begin{lemma}\label{lem:woodbury} 
We assume that $\cG, \widetilde \cG\in \Omega$ and $I(\cF^+,\cG)=1$ (recall from \eqref{e:defI}), and recall $G^\circ=(G-L)$. Then we have:  
\begin{align}\label{e:tG-Gdiff}
&\tG-G=\sum_{k\geq 0} GF(G^\circ F)^{k}G,
\end{align}
\end{lemma}

\begin{proof}

Thanks to \eqref{eq:infbound},   uniformly for $x,y\in \{l_\al, a_\al, b_\al, c_\al\}_{\al \in \qq{\mu}}$.
We can then expand \eqref{e:tGG} using the resolvent identity \eqref{e:resolv} and  \eqref{e:defF2g} to conclude that
\begin{align*}
\tG-G
&=-GU(\mathbb I+V^\top G U)^{-1}V^\top G
=-GU(\mathbb I+V^\top L U+V^\top G^\circ U)^{-1}V^\top G\\
&
=-GU\left((\mathbb I+V^\top L U)^{-1}\sum_{k\geq 0}(-1)^k(V^\top G^\circ U (\mathbb I+V^\top L U)^{-1})^k \right)V^\top G\\
&
=\sum_{k\geq 0}(-1)^{k+1} GU(\mathbb I+V^\top L U)^{-1}(V^\top G^\circ U (\mathbb I+V^\top L U)^{-1})^k V^\top G\\
&=\sum_{k\geq 0} GF(G^\circ F)^{k}G
.
\end{align*}
This gives \eqref{e:tG-Gdiff}. 
\end{proof}

\begin{proposition}
\label{c:Qmchange}
We assume that $\cG, \widetilde \cG\in \Omega$ and $I(\cF^+,\cG)=1$ (recall from \eqref{e:defI}). Then for $w,w'\in \qq{N}$, we have
\begin{align}\begin{split}\label{e:Gsw_exp}
    |(\tG-G)_{ww'}|\leq (d-1)^\ell\sum_{x\in \{l_\al, a_\al, b_\al, c_\al\}_{\al \in \qq{\mu}}} (|G_{wx}|^2+|G_{w'x}|^2).
\end{split}\end{align}
As a consequence, we have the following bounds
\begin{align}\label{e:Qtmtbound}
|\wt Q -Q |, |\wt m -m |\lesssim (d-1)^{3\ell}N^\fo \Phi.
\end{align}
\end{proposition}

We start with the following estimates, which will be used later to prove \Cref{c:Qmchange}. 
\begin{claim}\label{c:Fest}
The matrix $F$ from \eqref{e:defF} has nonzero entries only on the vertices $\{l_\al, a_\al, b_\al, c_\al\}_{\al\in \qq{\mu}}$, and 
\begin{align}\begin{split}\label{e:Fbound}
    \sum_{s,s'\in \{l_\al, a_\al, b_\al, c_\al\}_{\al\in \qq{\mu}}}|F_{ss'}|
    \lesssim \ell (d-1)^\ell.
\end{split}\end{align}
\end{claim}
    
\begin{proof}[Proof of \Cref{c:Fest}]
    It is easy to see from the expression \eqref{e:defF} that $F$
 has nonzero entries only on the vertices $\{l_\al, a_\al, b_\al, c_\al\}_{\al\in \qq{\mu}}$.  We remark that condition on $I(\cF^+, \cG)=1$, both $L$ and $\wt L$ are simply the Green's function of copies of $d$-regular trees. Hence,  the estimate \eqref{e:Gtreemkm} gives
\begin{align}\label{e:Piolha}
  |\tL_{{\sfJ}_\al \sfJ'_{\al'}}|\lesssim (d-1)^{-\dist_{\cT}(l_\al,l_{\al'})/2},\text{ for } \sfJ, \sfJ'\in \{l,a,b,c\},\quad \al, \al'\in \qq{\mu}.
\end{align}
 By plugging \eqref{e:Piolha} into \eqref{e:defF}, we get
\begin{align*}\begin{split}
    &|F_{{\sfJ}_\al \sfJ'_{\al'}}|\lesssim (d-1)^{-\dist_{\cT}(l_\al, l_{\al'})/2},
   \text{ for }  \sfJ, \sfJ'\in \{l,a,b,c\},\quad \al, \al'\in \qq{\mu},\\
    &\sum_{s,s'\in \{l_\al, a_\al, b_\al, c_\al\}_{\al\in \qq{\mu}}}|F_{ss'}|
    \lesssim \sum_{\al, \al'\in \qq{\mu}}(d-1)^{-\dist_{\cT}(l_\al, l_{\al'})/2}\lesssim \ell (d-1)^\ell,
\end{split}\end{align*}
where in the last inequality we used that $\{\al'\in \qq{\mu}: \dist_{\cT}(l_\al, l_{\al'})=2r\}=\OO((d-1)^{r})$ for $0\leq r\leq \ell$.
\end{proof}

\begin{proof}[Proof of \Cref{c:Qmchange}]
 For any $x,y\in \{l_\al, a_\al, b_\al, c_\al\}_{\al \in \qq{\mu}}$, we have 
\begin{align}\label{e:sumF}
|(F(G^\circ F)^{k})_{xy}|
&\leq N^{-k\fb}\sum_{s_1, s_2}|F_{s_1s_2}|\sum_{s_3, s_4}|F_{s_3s_4}|\cdots 
  \sum_{s_{2k-1}, s_{2k}}|F_{s_{2k-1}s_{2k}}|\leq N^{-k\fb}\ell^k (d-1)^{k\ell} 
\end{align}
where in the first statement we used
$| G^{\circ}_{xy}|\leq N^{-\fb} $ uniformly for $x,y\in \{l_\al, a_\al, b_\al, c_\al\}_{\al \in \qq{\mu}}$ from \eqref{eq:infbound}; in the second statement we used \eqref{e:Fbound}.

The claim \eqref{e:Gsw_exp} follows from \eqref{e:tG-Gdiff},
     \begin{align}\begin{split}\label{e:tG-Gdiff4}
    | (\wt G-G)_{ww'}|=\left|\sum_{k\geq 0} (GF(G^\circ F)^{k}G)_{ww'}\right|
     &\leq\sum_{x,y\in \{l_\al, a_\al, b_\al, c_\al\}_{\al \in \qq{\mu}}} \sum_{k\geq 0}(\ell N^{-\fb} (d-1)^{\ell} )^k|G_{wx}||G_{yw'}|\\
     &\leq \sum_{x,y\in \{l_\al, a_\al, b_\al, c_\al\}_{\al \in \qq{\mu}}}|G_{wx}||G_{yw'}|\\
     &\leq (d-1)^\ell\sum_{x\in \{l_\al, a_\al, b_\al, c_\al\}_{\al \in \qq{\mu}}} (|G_{wx}|^2+|G_{w'x}|^2).
    \end{split}\end{align}
where in the first statement we used \eqref{e:tG-Gdiff}; in the second statement we used \eqref{e:sumF}; in the third statement we sum the geometry series; in the last statement we used Cauchy-Schwartz inequality.

Next we prove \eqref{e:Qtmtbound} for $|\wt Q -Q |$, the statement for $|\wt m -m |$ follows from \eqref{e:tmmdiff}.
The difference $\wt Q -Q $ can be rewritten as
\begin{align}\begin{split}\label{e:tQ-Qdiff}
  &\phantom{{}={}}\frac{1}{Nd}\sum_{ \{u,v\}\notin \{\{l_\al, a_\al\},\{b_\al, c_\al\}\}_{\al\in\qq{\mu}} }A_{uv}(\widetilde G_{vv}^{(u)}-G_{vv}^{(u)})+\frac{1}{Nd}\sum_{\{u,v\}\in \{\{l_\al, c_\al\},\{a_\al, b_\al\}\}_{\al\in\qq{\mu}}}(\wt G^{(u)}_{vv}-G_{vv}^{(u)})\\
    &+\frac{1}{Nd}\sum_{\{u,v\}\in \{\{l_\al, c_\al\},\{a_\al, b_\al\}\}_{\al\in\qq{\mu}}}G_{vv}^{(u)}
    -\frac{1}{Nd}\sum_{\{u,v\}\in \{\{l_\al, a_\al\},\{b_\al, c_\al\}\}_{\al\in\qq{\mu}}}G^{(u)}_{vv}. 
\end{split}\end{align}
For $G_{vv}^{(u)}, \wt G_{vv}^{(u)}$ in \eqref{e:tQ-Qdiff}, we can rewrite them using the Schur complement formula \eqref{e:Schurixj},
\begin{align}\label{e:schurexp}
    G_{vv}^{(u)}=G_{vv}-\frac{G_{uv}^2}{G_{uu}},\quad
    \wt G_{vv}^{(u)}=\wt G_{vv}-\frac{\wt G_{uv}^2}{\wt G_{uu}}.
\end{align}
For the difference $\widetilde G_{vv}^{(u)}- G_{vv}^{(u)}$, using \eqref{e:schurexp} and \eqref{eq:infbound}, we can bound it as
\begin{align}\begin{split}\label{e:Guuvdiff}
 | \widetilde G_{vv}^{(u)}- G_{vv}^{(u)}|\leq 
 |\wt G_{vv}-G_{vv}|
 + |\wt G_{uv}-G_{uv}|
 + |\wt G_{uu}-G_{uu}|.
\end{split}\end{align}
By plugging \eqref{e:Guuvdiff} into \eqref{e:tQ-Qdiff}, we conclude that
\begin{align*}
|\wt Q-Q|
&\lesssim \sum_{u\sim v}( |\wt G_{vv}-G_{vv}|
 + |\wt G_{uv}-G_{uv}|
 + |\wt G_{uu}-G_{uu}|)+\frac{(d-1)^\ell}{N}\\
& \lesssim (d-1)^\ell \sum_{u\sim v}\sum_{x\in \{l_\al, a_\al, b_\al, c_\al\}_{\al \in \qq{\mu}}}|(G_{ux}|^2+|G_{vx}|^2)+\frac{(d-1)^\ell}{N}\lesssim (d-1)^{3\ell}N^\fo \Phi,
\end{align*}
where in the second statement we used \eqref{e:Gsw_exp}; in the last inequality, we used the Ward identity bound.
\end{proof}

\section{Proof  for the self-consistent equation} \label{s:proofoutline}

In this section we prove \eqref{e:QY} in \Cref{t:recursion}. 
As discussed in \Cref{s:forest}, at each iteration, we estimate \eqref{e:R_ifirst} by performing a local resampling around $(i,o)$. We will show that the expectation breaks down into an $\OO(1)$-weighted sum of terms in the same form.
We begin with a weighted version of \eqref{e:R_ifirst}, as presented on the left-hand side of \eqref{e:maint} in the following proposition. Here, the additional factor $(d-1)^{3r\ell}$ depends on the admissible function $R_\bfi$. The reader can interpret this as follows: each term in $R_r$ (a product of $r$
factors of the form \eqref{e:defcE1}) is accompanied by a factor $(d-1)^{3\ell}$.  Thanks to \eqref{e:Bsmall}, even with these factors, the size of the terms remains small, i.e. bounded by $(d-1)^{3r\ell}N^{-\fb r}\ll1$.
These factors are introduced to ensure that all the expansions in this section are \(\OO(1)\)-weighted sums of terms, as defined in \Cref{def:O1sum}. Specifically, combinatorial factors are absorbed into \((d-1)^{3r \ell}\).

The proposition below expresses the expectation of Green's functions of the graph $\cG$ in terms of the quantities of the new graph $\widetilde \cG$ after local resampling.

\begin{proposition}\label{p:add_indicator_function}
  Consider a forest $\cF=(\bfi, E)$ as in \eqref{e:cFtocF+} and a function $(G_{oo}^{(i)}-Y)R_\bfi$ with $R_\bfi\in \Adm(r,\cG,\cF)$ and $r\geq 1$. We perform a local resampling around $(i, o) \in \cF$ using the resampling data ${\bf S}=\{(l_\al, a_\al), (b_\al, c_\al)\}_{\al\in\qq{\mu}}$, denoting the new graph as $\widetilde \cG = T_\bfS(\cG)$, with its corresponding Green's function $\widetilde G$. Then 
    \begin{align}\begin{split}\label{e:maint}
    &\phantom{{}={}}\frac{(d-1)^{3r\ell}}{Z_\cF}\sum_{\bfi}\bE\left[I(\cF,\cG)\bm1(\cG\in \Omega) (G_{oo}^{(i)}-Y)R_\bfi\right]\\
    &=\frac{(d-1)^{3r\ell}}{Z_{\cF^+}}\sum_{\bfi^+}\bE\left[I(\cF^+,\cG)\bm1(\cG,\tcG\in \Omega)(\widetilde G_{oo}^{(i)}-Y) \wt R_\bfi\right]+\OO(N^{-\fb/2}\bE[\Psi]).
\end{split}\end{align}
Here, $\widetilde R_\bfi$ is obtained by computing $R_\bfi$ for the graph $\widetilde \cG$.
\end{proposition}

If we temporarily ignore the indicator and the averaging over embeddings, the
above proposition reduces to the symmetry
\begin{align}\label{e:changet}
 \bE\!\left[(G^{(i)}_{oo}-Y) R_\bfi\right]
= \bE\!\left[(\wt G^{(i)}_{oo}-\wt Y) \wt R_\bfi\right]
= \bE\!\left[(\wt G^{(i)}_{oo}- Y) \wt R_\bfi\right] -\bE\!\left[(\wt Y-Y) \wt R_\bfi\right].
\end{align}
The indicator $I(\cF,\cG)$ just restricts attention to “good” placements where we have good estimates for the Green's function. In this regime we have $|\wt R_\bfi|\lesssim N^{-\fb}$ and  
\begin{align*}
|\wt Y-Y|=|Y_\ell(\wt Q,z)-Y_\ell(Q,z)|\lesssim |\ell (\wt Q-Q)|\lesssim \ell (d-1)^{3\ell}N^\fo\Phi
\end{align*}
where we used \eqref{e:Yl_derivative} and \eqref{e:Qtmtbound}. So the last term in \eqref{e:changet} is negligible. We omit the proof of \Cref{p:add_indicator_function}.

The right-hand side of \eqref{e:maint} involves the Green's function of the switched graph $\tcG$. The following two propositions help evaluate them, and express them as $\OO(1)$-weighted sums of terms involving only the Green's function of the original graph $\cG$, with negligible error. More importantly, these terms match the structure of the left-hand side of \eqref{e:maint}.

\begin{proposition}\label{p:general}
Given a forest $\cF=(\bfi, E)$ and a function $(G_{oo}^{(i)}-Y)R_\bfi$ with $R_\bfi\in \Adm(r,\cF,\cG)$ (recall from \Cref{def:pgen}). We construct $\cF^+=(\bfi^+, E^+)$ (as given by \eqref{e:cF++}) by performing a local resampling around  $(i,o)\in \cF$ with resampling data ${\bf S}=\{(l_\al, a_\al), (b_\al, c_\al)\}_{\al\in\qq{\mu}}$, and denote $\wt \cG=T_\bfS(\cG)$. 
\begin{enumerate}
\item Let  $r\geq 2$ and take $R_{\bfi}\in \Adm(r,\cF,\cG)$. Then, up to an error of size $\OO(N^{-\fb/4}\bE[ \Psi])$,  
    \begin{align}\label{e:higher_case3}
       \frac{(d-1)^{3r\ell}}{(d-1)^{\fq \ell/2}Z_{\cF^+}}\sum_{\bfi^+}\bE\left[I(\cF^+,\cG)\bm1(\cG,\tcG\in \Omega)(\widetilde G_{oo}^{(i)}-Y) \wt R_\bfi\right]
    \end{align}
    can be rewritten as an $\OO(1)$-weighted sum of terms in the following form
     \begin{align}\label{e:case3_copy}
       \frac{(d-1)^{3r^+ \ell}}{(d-1)^{\fq^+\ell/2}Z_{\cF^+}}\sum_{\bfi^+}  \bE[\bm1(\cG\in \Omega)I(\cF^+, \cG)(G_{c_\al c_\al}^{(b_\al)}-Y)R_{\bfi^+}],
    \end{align}
    where $R_{\bfi^+}\in \Adm(r^+,\cF^+,\cG)$, where either $\fq^+\geq \fq+1$, $r^+\geq r$; or  $\fq^+\geq \fq$, $r^+\geq r+1$.

\item Let $R_\bfi=( G_{oo}^{(i)}- Q )$. Then, up to an error of size $\OO((d-1)^{-\fq\ell/2}N^\fo\bE[\Psi])$,  
\begin{align}\label{e:higher_case1}
       \frac{1}{(d-1)^{\fq\ell/2}Z_{\cF^+}}\sum_{\bfi}  \bE[I(\cF^+,\cG)\bm1(\cG,\tcG\in \Omega)(\wt G_{oo}^{(i)}-Y )\wt R_\bfi]
    \end{align}
 can be rewritten as an $\OO(1)$-weighted sum of terms in the form of \eqref{e:case3_copy} with $R_{\bfi^+}\in \Adm(r^+,\cF^+,\cG)$, where $r^+\geq 2$ and $\fq^+\geq \fq$, or in the following form
    \begin{align}\label{e:higher_case11}
       \frac{1}{(d-1)^{\fq^+ \ell/2}Z_{\cF}}\sum_{\bfi^+}  \bE[\bm1(\cG\in \Omega)I(\cF^+, \cG)(G_{c_\al c_\al}^{(b_\al)}-Y)(G_{c_\al c_\al}^{(b_\al)}-Q)],
    \end{align}
    where $\fq^+ \geq \fq+1$.
     \end{enumerate}

\end{proposition}

The claim \eqref{e:QY} follows from iterating \Cref{p:add_indicator_function} and \Cref{p:general}.
\begin{proof}[Proof of \eqref{e:QY}]
We recall the sequence of forests $\cF_0\subset\cF_1\subset \cF_2\subset \cdots$ from \Cref{s:forest}. They encode all edges involved in local resamplings. By \Cref{c:resample} and \Cref{p:iteration},  up to an error $(d-1)^{2\ell}\bE[\Psi]$, $\bE[\bm1(\cG\in \Omega)(Q-Y)]$ is an $\OO(1)$-weighted sum of terms of the following two forms
  \begin{align}\label{e:st1}
       \frac{1}{(d-1)^{\fq\ell/2}Z_{\cF_1}}\sum_{\bfi}  \bE[\bm1(\cG\in \Omega)I(\cF_1, \cG)(G_{c_\al c_\al}^{(b_\al)}-Y)(G_{c_\al c_\al}^{(b_\al)}-Q)],
    \end{align}
where $\fq\geq 0$; or
  \begin{align}\label{e:st2}
       \frac{(d-1)^{3r\ell}}{(d-1)^{\fq\ell/2}Z_{\cF_1}}\sum_{\bfi}  \bE[\bm1(\cG\in \Omega)I(\cF_1, \cG)(G_{c_\al c_\al}^{(b_\al)}-Y)R_{\bfi_1}],
    \end{align}
where $\fq\geq 0$ and $r\geq 2$, and the function $R_{\bfi_1}\in \Adm(r,\cF_1, \cG)$.

The above expression \eqref{e:st1} aligns with the form of \Cref{p:add_indicator_function}, allowing us to apply \Cref{p:add_indicator_function} and \Cref{p:general} for the iteration process.  The result in \Cref{p:general} essentially states that, after further expansion, \eqref{e:st1} either maintain the same form with an additional factor of $(d-1)^{-\ell/2}$, or they transform into \eqref{e:st2}.
Similarly, after expansion, \eqref{e:st2} remains in the same form, either with an additional  $(d-1)^{-\ell/2}$ factor, or an extra term in the form of \eqref{e:defcE1}, which is bounded by $N^{-\fb}\ll (d-1)^{-\ell/2}$.
Therefore, after finitely many steps, namely $\OO(4\log_{d-1}(N)/\ell)$, all terms are bounded by $\OO(N^{-2})=\OO(\bE[\Psi]/N)$. Meanwhile the errors from \Cref{p:add_indicator_function} and \Cref{p:general} are all bounded by $\OO((d-1)^{2\ell}\bE[\Psi])$. This gives \eqref{e:QY}.

\end{proof}

\subsection{Proof of \Cref{p:general}}\label{s:general}

We denote $R_\bfi\in\Adm(r,\cF,\cG)$ in \Cref{p:general} (as in \Cref{def:pgen}), which contains $r$ factors in the form of \eqref{e:defcE1}. 

In the following we prove \Cref{p:general} assuming $\fq=0$. The case with $\fq\geq 1$ follows from simply multiplying $(d-1)^{-\fq \ell/2}$. 
In order to use our various propositions from \Cref{s:expQ}, \Cref{sec:schurlemma} and \Cref{sec:fanalysis} that allow us to reduce $(\wt G_{oo}^{(i)}-Y ) \wt R_\bfi $ to terms of the unswitched graph $\cG$, we need to classify the factors of $R_\bfi$ based on their dependence on $i,o$. Let 
\begin{align}\label{e:introB_j}
     B_0=(G_{oo}^{(i)}-Y ),\quad R_{\bfi}=\prod_{j=1}^{r}B_j,
\end{align}
where $B_j$ is of one of the following forms
\begin{align}\begin{split}\label{e:Blist}
&\{G_{cc}^{(b)}-Q \}_{(b,c)\in \cC^\circ},\quad
\{G_{oc}^{(ib)}, G_{ic}^{(b)}, G_{ib}, G_{ob}^{(i)}\}_{(i,o)\neq (b,c)\in \cC^\circ},\\
&\{ G_{cc'}^{(bb')}, G_{bc'}^{(b')}, G_{bb'}\}_{(i,o)\neq (b,c), (b',c')\in \cC^\circ},\quad (Q -\msc(z )).
\end{split}\end{align}

We can write \eqref{e:higher_case3} (with $\fq=0$) as
\begin{align}\begin{split}\label{e:sreplace0}
    &\phantom{{}={}}\frac{(d-1)^{3r\ell}}{Z_{\cF^+}}\sum_{\bfi^+}\bE\left[I(\cF^+,\cG)\bm1(\cG,\tcG\in \Omega)(\wt G_{oo}^{(i)}-Y ) \wt R_\bfi\right]\\
    &=\frac{(d-1)^{3r\ell}}{Z_{\cF^+}}\sum_{\bfi^+}\bE\left[I(\cF^+,\cG)\bm1(\cG,\tcG\in \Omega) (\wt G_{oo}^{(i)}-Y )\prod_{j=1}^r \wt B_j\right].
\end{split}\end{align}
In the following we discuss the terms $B_j$  as in \eqref{e:introB_j} after the local
resampling.  In \Cref{i1}-\Cref{i7} below, for any $h\geq 0$, we denote $R_h$ a product of $h$ factors of the form 
\begin{align}\label{e:list_factor}
    (G_{cc}^{(b)}-Q ),  \quad \{G_{cc'}^{(bb')}, G_{bc'}^{(b')}, G_{bb'}\},\quad  (Q -\msc(z )),
\end{align}
where $(b,c)\neq (b',c')\in \cK^+$.
They are terms from \Cref{l:coefficient} and \Cref{lem:offdiagswitch}. We refer to \Cref{f:decomposition} for a picture illustration. 
\begin{enumerate}
\item \label{i1}
By \Cref{l:coefficient}, $\wt G_{oo}^{(i)}-Y =\widehat B_0+\cE_0$, where $\wh B_0$ is given by
\begin{align}\label{e:fcase1W0}
\wh B_0= \frac{\msc^{2(\ell+1)}(z )} {(d-1)^{\ell+1}}\sum_{\al\in\sfA_i}(G^{(b_\alpha)}_{c_\alpha c_\alpha}-Q )+\frac{\msc^{2(\ell+1)}(z )} {(d-1)^{\ell+1}}\sum_{\al\neq \beta\in\sfA_i}G^{(b_\alpha b_\beta)}_{c_\alpha c_\beta}
   + \cU_0. 
\end{align}
Here, $\cU_0$ is an $\OO(1)$-weighted sum of terms of the form $(d-1)^{3(h-1)\ell}R_h$ with $h\geq 2$, and each term contains at least one factor of the form  $( G_{c_\al c_{\al}}^{(b_\al)}-Q )$ or $G_{c_\al c_{\beta}}^{(b_\al b_\beta)}$; $\cE_0$ is $\cE$ in \Cref{l:coefficient}.

    \item \label{i2}
For $B_j=(G_{oo}^{(i)}-Q )$  we can first rewrite $\wt B_j$ as
\begin{align*}
    \wt B_j=(\wt G_{oo}^{(i)}-\wt Q )=(\wt G_{oo}^{(i)}-Y )+(Y -\wt Q ),
\end{align*}
then expand according to \Cref{l:coefficient}. This gives that $\wt B_j=\wh B_j+\cE_j$, where
\begin{align}\label{e:zhankai1}
\wh B_j&=\frac{\msc^{2(\ell+1)}(z )} {(d-1)^{\ell+1}}\sum_{\al\in\sfA_i}(G^{(b_\alpha)}_{c_\alpha c_\alpha}-Q )+\frac{\msc^{2(\ell+1)}(z )} {(d-1)^{\ell+1}}\sum_{\al\neq \beta\in\sfA_i}G^{(b_\alpha b_\beta)}_{c_\alpha c_\beta}+\cU_j.
\end{align}
Here $\cU_j$ is an $\OO(1)$-weighted sum of terms of the form $(d-1)^{3(h-1)\ell}R_h$ with $h\geq 2$; 
$\cE_j$ is $ \cE+( (Y -Q ) +
    (Q -\wt Q ))$, where $\cE$ is from \Cref{l:coefficient}.

\item \label{i3}
For $B_j\in \{G_{oc}^{(ib)}, G_{ic}^{(b)}\}_{(i,o)\neq (b,c)\in \cC^\circ}$, by \Cref{lem:offdiagswitch}, we have $\wt B_j=\wh B_j+\cE_{j}$, where
\begin{align}\label{e:zhankai3}
  \wh B_j=(d-1)^{-\ell/2}\sum_{\al\in \qq{\mu}}\fc_1(\bm1(\al\in \sfA_i)) G_{c_\al c}^{(b_\al b)}+\cU_j.
\end{align}
Here $\cU_j$ is an $\OO(1)$-weighted sum of terms of the form $(d-1)^{3h \ell}R_hG^{(b_\al b)}_{c_\al c}$ with $h\geq 1$;  and 
$\cE_j$ is $ \cE$ in the first statement of \Cref{lem:offdiagswitch}.

For $B_j\in \{ G_{ib},  G_{ob}^{(i)}\}_{(i,o)\neq (b,c)\in \cC^\circ}$, we have a similar expansion $\wt B_j=\wh B_j+\cE_{j}$, where
\begin{align}\label{e:zhankai4}
  \wh B_j=(d-1)^{-\ell/2}\sum_{\al\in \qq{\mu}}\fc_1(\bm1(\al \in \sfA_i)) G_{c_\al b}^{(b_\al )}+\cU_j.
\end{align}
Here $\cU_j$ is an $\OO(1)$-weighted sum of terms of the form $(d-1)^{3h\ell}R_hG^{(b_\al)}_{c_\al b}$ with $h\geq 2$; $ \cE_j$ is $\cE$ in the second statement of \Cref{lem:offdiagswitch}.

\item \label{i4}
For $B_j\in \{ G_{cc}^{(b)}, G_{cc'}^{(bb')}, G_{bc'}^{(b')}, G_{bb'}\}_{(b,c)\neq  (b',c')\in \cC^\circ}$, by \Cref{lem:offdiagswitch}, we have  $\widetilde B_j=\widehat B_j+\cE_j$, where 
 $   \widehat B_j=\cU_j=B_j$,
and $\cE_j$ is $\cE$ in the third statements of \Cref{lem:offdiagswitch}.


\item \label{i7}
For $B_j=Q -\msc(z )$, we have by the  statement \eqref{e:Qtmtbound} in \Cref{c:Qmchange},
$\wt B_j=\widehat B_j+\cE_j$, where 
$
\widehat B_j=\cU_j=B_j=Q-\msc(z)
$,
and $|\cE_j|\lesssim (d-1)^{6\ell}N^{\fo}\Phi$.

\end{enumerate}

\begin{figure}[t]
    \centering
    \begin{tikzpicture}[
      >=stealth,
      box/.style={draw=brown!70!black, very thick, rounded corners=1pt,
                  fill=brown!5, inner sep=4pt},
      smallbox/.style={box, inner sep=3pt},
      brace/.style={decorate, decoration={brace, amplitude=4pt}},
      arr/.style={->, thick, draw=gray!70},
      math/.style={font=\large},
      every node/.style={font=\normalsize}
    ]


    \node[box, math] (topG0) at (-5,3.5)
      {$\bigl(\widetilde{G}^{(i)}_{oo}-Y\bigr)$};
      \node[math] at (-0.75,3.5) {$\times$};

    \node[box, math] (topG1) at (-2,3.5)
      {$\wt B_j=\bigl(\widetilde{G}^{(i)}_{oo}-\wt Q\bigr)$};

     \node[math] at (2.,3.5) {$\times$};
      
    \node[box, math] (topG2) at (1.6,3.5)
      {$\wt B_j\in \{G_{oc}^{(ib)}, G_{ic}^{(b)}\}$};
      
    \node[box, math] (topG3) at (5.2,3.5)
      {$\wt B_j\in \{G_{ib}, G_{ob}^{(i)}\}$};

 \node[box, math] (topG4) at (8,3.5)
      {other $\wt B_j$};

\node[math] at (-3.75,3.5) {$\times$};
\node[math] at (-0.25,3.5) {$\times$};
\node[math] at (3.45,3.5) {$\times$};
\node[math] at (6.925,3.5) {$\times$};

    \node[box, math] (leftBig) at (-4.7,0.4)
    {$
      \begin{aligned}
      &\frac{\sum_{\alpha\in \sfA_i}\bigl(G^{(b_\alpha)}_{c_\alpha c_\alpha}-Q\bigr)}
              {(d-1)^{\ell+1}}\\[4pt]
      &\frac{\sum_{\alpha\neq\beta\in \sfA_i} G^{(b_\alpha b_\beta)}_{c_\alpha c_\beta}}
              {(d-1)^{\ell+1}}\\[4pt]
      &\mathcal U_j : (d-1)^{3(h-1)\ell} R_h
      \end{aligned}
    $};

    \node[box, math] (midBig) at (0.1,0.4)
    {$
      \begin{aligned}
      &\frac{\sum_{\alpha\in \qq{\mu}}\fc(\bm1(\al\in \sfA_i))G^{(b_\alpha b)}_{c_\alpha  c}}{(d-1)^{\ell/2}}\\
      &\mathcal U_j : (d-1)^{3h\ell} R_h
      \end{aligned}
    $};

     \node[box, math] (rightBig) at (5.4,0.4)
    {$
      \begin{aligned}
      &\frac{\sum_{\alpha\in \qq{\mu}}\fc(\bm1(\al\in \sfA_i))G_{c_\alpha b}^{(b_\al)}}{(d-1)^{\ell/2}}\\
      &\mathcal U_j : (d-1)^{3h\ell} R_h
      \end{aligned}
    $};

    \node[box, math] (lastBig) at (8.5,0.4)
    {$
      B_j
    $};

    \draw[arr] (topG0.south) -- (leftBig.north);
    \draw[arr] (topG1.south) -- (leftBig.north);

    \draw[arr] (topG2.south) -- (midBig.north);
    \draw[arr] (topG3.south) -- (rightBig.north);
    \draw[arr] (topG4.south) -- (lastBig.north);

    \end{tikzpicture}
    \caption{Schematic decomposition of the expression into three cases and products.}
    \label{f:decomposition}
\end{figure}
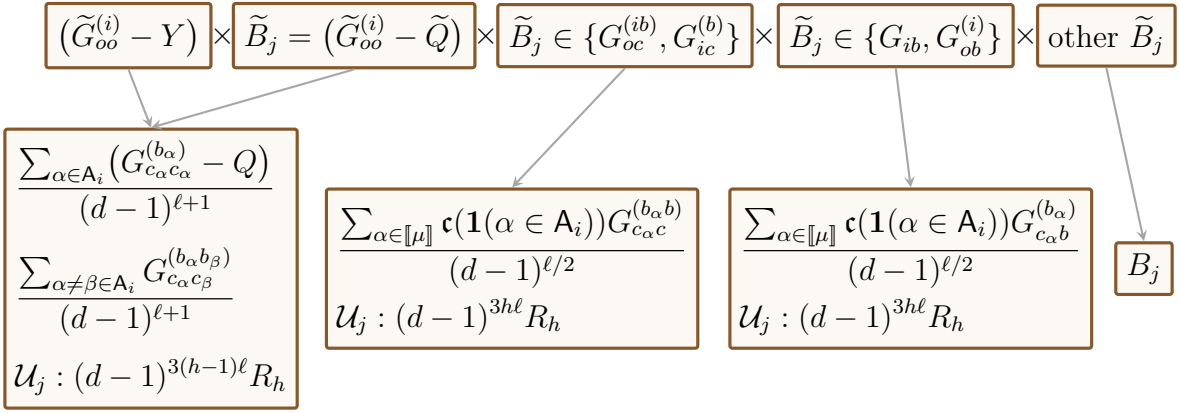

We can replace $(\wt G_{oo}^{(i)}-Y )\prod_{j=1}^r \wt B_j$ in \eqref{e:sreplace0} with $\prod_{j=0}^r \widehat B_j$, with the overall error from this substitution being negligible. 
\begin{lemma}\label{l:yibu}
Adopt the notation and assumptions in \Cref{p:general}, we can rewrite \eqref{e:sreplace0} as
\begin{align}
\label{e:hatRi}
  \eqref{e:sreplace0}=  \frac{(d-1)^{3r\ell}}{Z_{\cF^+}}\sum_{\bfi^+}\bE\left[I(\cF^+ ,\cG)\bm1(\cG\in \Omega) \prod_{j=0}^r \widehat B_j\right]+\OO(N^{-\fb/2}\bE[  \Psi]).
\end{align}
\end{lemma}

If we temporarily ignore the indicator and the averaging over embeddings, the
total substitution error satisfies
\[
(d-1)^{3r\ell} N^{-r\fb}\sum_{j=0}^r\bE[|\cE_j|]
\]
which follows from the bound $|B_j|\leq N^{-\fb}$ in \eqref{e:Bsmall}. 
Using the same argument as in \eqref{e:Eterm1}, we obtain
\(\bE[|\cE_j|]\lesssim (d-1)^{2\ell}\bE[\Psi]\) in all five cases above,
except for \Cref{i2}, where we also have the additional errors (ignoring the
indicator functions)
\begin{align*}
\bE[|Y-Q|+|\wt Q-Q|]
  \lesssim N^{\fb/8}\bE[\Psi]+\bE[(d-1)^{3\ell}N^\fo \Phi],
\end{align*}
where we used the definitions \eqref{e:defPhi} and estimate \eqref{e:Qtmtbound}.
Finally, \eqref{e:hatRi} follows from the fact that \(r\geq 1\).

We can further decompose $(d-1)^{3r\ell}\prod_{j=0}^{r} \widehat B_j$ in \eqref{e:hatRi}. The discussions in \Cref{i1}--\Cref{i7} provide the decomposition of $(d-1)^{3r\ell}\prod_{j=0}^{r} \widehat B_j$ as an $\OO(1)$-weighted sum of terms of the following form:
\begin{align}\begin{split}\label{e:oneterm0}
&\frac{(d-1)^{3\widehat h\ell}}{(d-1)^{(k_1 +(k_2+k_3)/2)\ell}} \sum_{\bm\al, \bm\beta, \bm\gamma}  \prod_{m=1}^{k_1}  A_{\al_  m}  \prod_{  m=1}^{k_2}  B_{\beta_m} 
\prod_{m=1}^{k_3/2}  C_{\gamma_{2m-1} \gamma_{2m}}  R_{\widehat h+1-k_1-k_2-k_3/2}.
\end{split}\end{align}
Here the summation for $\bm\al, \bm\beta, \bm\gamma$ runs through each $\al_m, \beta_m, \gamma_m$ in one of the sets $\qq{\mu}$, $\sfA_i$ or $\qq{\mu}\setminus \sfA_i$ (we recall $\sfA_i$ from \Cref{fig:Ai}).
The factors in \eqref{e:oneterm0} are defined as follows:
\begin{align}\begin{split}\label{e:defABC}
  &A_\al=   ( G_{c_\al c_\al}^{(b_\al)}-Q ), \quad B_\beta\in \left\{ G_{c_\beta c}^{(b_\beta b)},  G_{c_\beta b}^{(b_\beta)}\right\}_{(b,c)\in \cC^\circ},\quad C_{\gamma \gamma'}=  G_{c_\gamma c_{\gamma'}}^{(b_\gamma b_{\gamma'})}.
\end{split}\end{align}
 Here $A_\al$ originates from $G_{c_\al c_\al}^{(b_\al)}-Q $ in \eqref{e:fcase1W0}, \eqref{e:zhankai1}. $B_\beta$ comes from $G_{c_\al c}^{(b_\al b)}, G_{c_\al b}^{(b_\al)}$ in \eqref{e:zhankai3}, \eqref{e:zhankai4}.  $C_{\gamma \gamma'}$ are from $G_{c_\al  c_\al}^{(b_\al b_\beta)}$ in \eqref{e:fcase1W0}, \eqref{e:zhankai1}. In total there are $k_1+k_2+k_3/2$ such terms. $R_{\widehat h+1-k_1-k_2-k_3/2}$ collects all $R_h$ factors from $\cU_j$ as in \Cref{i1}--\Cref{i7}, with a total count of ${\widehat h+1-k_1-k_2-k_3/2}$. Thus, the summand in \eqref{e:oneterm0} consists of $\widehat h+1$  factors in total. We notice that after expansion, each $\wt B_j$ gives at least one term, so $\widehat h+1 \geq r+1$. Finally the coefficient $(d-1)^{3\widehat h\ell}$ arises from the following crucial observation: in the replacements outlined in \Cref{i1}--\Cref{i7}, each term $(d-1)^{3\ell}B_j$ is replaced by one of the terms $(d-1)^{3\ell}A_\al, (d-1)^{3\ell}A_\beta, (d-1)^{3\ell}C_{\gamma\gamma'}$ as in \eqref{e:defABC}, or a factor $R_h$ (as in \eqref{e:list_factor}) for $h\geq 1$ with coefficient at most $(d-1)^{3h\ell}$. This leads to the following statement

\begin{claim}
\eqref{e:hatRi} can be written as an $\OO(1)$-weighted sum of terms of the following form:
For $\widehat h\geq r$, 
\begin{align}\begin{split}\label{e:oneterm}
&\frac{1}{(d-1)^{(k_1 +(k_2+k_3)/2)\ell}} \sum_{\bm\al, \bm\beta, \bm\gamma}   \sum_{\bfi^+} \frac{(d-1)^{3\widehat h\ell}}{Z_{\cF^+}}\times \bE\left[I(\cF^+,\cG)\bm1(\cG\in \Omega) \widehat R_{\bfi^+}\right],\\
&\widehat R_{\bfi^+}
=R_{\widehat h+1}
=\prod_{m=1}^{k_1}  A_{\al_  m}  \prod_{  m=1}^{k_2}  B_{\beta_m} 
\prod_{m=1}^{k_3/2}  C_{\gamma_{2m-1} \gamma_{2m}}  E(\bm \chi).
\end{split}\end{align}
Here $R_{\widehat h+1}$ is the summand in \eqref{e:oneterm0}. 
In this way $\widehat R_{\bfi^+}\in \Adm(\widehat h+1,\cF^+,\cG)$.
The summation for $\bm\al, \bm\beta, \bm\gamma$ is over each $\al_m, \beta_m, \gamma_m$ in one of the sets $\qq{\mu}$, $\sfA_i$ or $\qq{\mu}\setminus \sfA_i$;
 and the factors in \eqref{e:oneterm} are given by \eqref{e:defABC}, and $E(\bm\chi)=R_{\widehat h+1-k_1-k_2-k_3/2}$ is a product of the remaining terms which depends on $\{b_\chi, c_\chi\}_{\chi \in \bm\chi}$ for some ${\bm\chi}\subset\qq{\mu}$.
\end{claim}

The summation over \(\bm\alpha, \bm\beta, \bm\gamma\) in \eqref{e:oneterm} produces
\((d-1)^{(k_1+k_2+k_3)\ell}\) terms of the form \eqref{e:case3_copy}. Note that
this number is much larger than the normalization factor
\((d-1)^{(k_1+(k_2+k_3)/2)\ell}\) in the denominator. However, as we will show
in \Cref{l:erbu}, most of these terms can be seen almost immediately to be
negligible. The cancellation mechanism is similar to those in
\eqref{e:core1} and \eqref{e:core21}. Before stating \Cref{l:erbu}, we first
introduce some notation.

We view $\bm\al, \bm\beta, \bm\gamma, \bm\chi$ as words, which are sequences of indices in $\qq{\mu}$. In particular $(\bm\al, \bm\beta, \bm\gamma)$ is a word with length $k_1+k_2+k_3$.
Given $\bm\chi$, we partition words $ \bm\omega\in \qq{\mu}^{k_1+k_2+k_3}$ into equivalence classes.  Two words $\bm\omega\sim \bm\omega'$ are equivalent if there is a bijection on $\qq{\mu}$ which preserves $\bm\chi$ and maps $\bm\omega$ to $\bm\omega'$. We remark that the expectation in \eqref{e:oneterm} depends only on the equivalence class of $(\bm\al, \bm\beta, \bm\gamma)$. 

For any $f_0, f_1\geq 0$, let $\mathsf W(f_0, f_1)$ denote a set of representatives for equivalence classes of $\qq{\mu}^{k_1+k_2+k_3}$. Here, for a word $\bm\omega\in \mathsf W(f_0, f_1)$, $f_0$ is the number of distinct indices (ignoring multiplicity)  that do not appear in 
$\bm\chi$, and $f_1$ is the number of these indices appearing exactly once in $\bm\omega$. The length of $\bm\omega$ is $k_1+k_2+k_3$, and $f_0-f_1$ of these distinct indices appear at least twice in $\bm\omega$. This implies
\begin{align}\label{e:f0f1}
    k_1+k_2+k_3\geq  f_1+2(f_0-f_1)\Rightarrow k_1+k_2+k_3\geq 2f_0-f_1.
\end{align}

\begin{example}
The following are two possible terms of \(\widehat R_{\bfi^+}\):
\begin{align*}
&A_1 A_2 B_2 C_{47}C_{37} G^{b_8 b}_{c_8 c}\bigl(G^{(b_4)}_{c_4 c_4}-Q\bigr),
\quad
\bm\omega=(\bm\alpha,\bm\beta,\bm\gamma)=(1,2),(2),(4,7,3,7),\quad \bm\chi=(8,4), \\
&A_1 A_5 B_5 C_{49}C_{39} G^{b_8 b}_{c_8 c}\bigl(G^{(b_4)}_{c_4 c_4}-Q\bigr),
\quad
\bm\omega'=(\bm\alpha',\bm\beta',\bm\gamma')=(1,5),(5),(4,9,3,9),\quad \bm\chi=(8,4).
\end{align*}
In this example, \(\bm\omega\sim \bm\omega'\) (by mapping \(2\) to \(5\) and \(7\) to \(9\), and keeping all other indices in \(\qq{\mu}\) fixed). By permutation invariance of the vertices, the two terms above have the same expectation. Moreover, in \(\bm\omega\) the indices \(1,2,3,7\) do not appear in \(\bm\chi=(8,4)\), and among them \(1\) and \(3\) appear only once, so \(\bm\omega\in\mathsf W(4,2)\).

\end{example}

The expectation in \eqref{e:oneterm} depends only on the equivalence class of $(\bm\al, \bm\beta, \bm\gamma)$. Moreover, for fixed $f_0$ and $f_1$, the summation of $(\bm\al, \bm\beta,\bm\gamma)\sim \bm\omega$, contains $\OO((d-1)^{f_0\ell})$ terms. Thus for the summation over $(\bm\al, \bm\beta, \bm\gamma)$ in \eqref{e:oneterm}, we can first sum over the equivalence classes. We recall the summand from \eqref{e:oneterm}
\begin{align}\label{e:partition2}
   \sum_{\bfi^+} \frac{(d-1)^{3\widehat h \ell}}{Z_{\cF^+}}\bE\left[I(\cF^+,\cG)\bm1(\cG\in \Omega) \widehat R_{\bfi^+}(\bm\al, \bm\beta, \bm\gamma, \bm\chi)\right].
\end{align}
We then have
\begin{align}\begin{split} \label{e:partition}
    \eqref{e:oneterm}&=\frac{1}{(d-1)^{(k_1 +(k_2+k_3)/2)\ell}} \sum_{\bm\al, \bm\beta, \bm\gamma}   \eqref{e:partition2}\\
    &=\frac{1}{(d-1)^{(k_1 +(k_2+k_3)/2)\ell}} \sum_{f_0,f_1}\sum_{\bm\omega\in \mathsf W(f_0,f_1)}\sum_{(\bm\al, \bm\beta, \bm\gamma)\sim \bm\omega}   \eqref{e:partition2},\\
    &=\sum_{f_0,f_1}\sum_{(\bm\al, \bm\beta, \bm\gamma)\in \mathsf W(f_0,f_1)}\frac{\OO((d-1)^{f_0\ell})}{(d-1)^{(k_1 +(k_2+k_3)/2)\ell}} \times   \eqref{e:partition2}.
\end{split}\end{align}

The following lemma states that for given $\bm\al, \bm\beta, \bm\gamma$, the summands in the last term of \eqref{e:partition} is either negligible, or it can be reduced to a term as in \eqref{e:oneterm2} below, where each index in $\bm\al, \bm\beta, \bm\gamma$ appears at least twice in $R_{\bfi^+}'$.

\begin{lemma}\label{l:erbu}
Fix $0\leq f_1\leq f_0$ satisfying \eqref{e:f0f1}, and a word $(\bm\al, \bm\beta, \bm\gamma)\in \mathsf W(f_0, f_1)$.
Let $\mathsf I_{\rm single}\subset \qq{\mu}$ denote the set of indices that appear only once among 
$(\bm\al, \bm\beta, \bm\gamma)$, and do not appear in $\bm\chi$. Then $|\mathsf I_{\rm single}|=f_1$, and
\begin{align}\begin{split}\label{e:oneterm1}
&\frac{(d-1)^{f_0\ell}}{(d-1)^{(k_1 +(k_2+k_3)/2)\ell}}   \sum_{\bfi^+} \frac{(d-1)^{3\widehat h\ell}}{Z_{\cF^+}}\times \bE\left[I(\cF^+,\cG)\bm1(\cG\in \Omega) \widehat R_{\bfi^+}(\bm\al, \bm\beta, \bm\gamma,\bm\chi)\right]
\end{split}\end{align}
satisfies
\begin{enumerate}
    \item If there exists $\al_m\in \mathsf I_{\rm single}$ then $\eqref{e:oneterm1}=\OO(N^{-\fb/4}\bE[ \Psi])$.
\item Otherwise, 
\begin{align}\begin{split}\label{e:oneterm2}
\eqref{e:oneterm1}=\frac{(d-1)^{3(\widehat h+f_1)\ell}}{(d-1)^{\fq^+\ell/2}Z_{\cF^+}}
\sum_{\bfi^+}  \bE\left[I(\cF^+,\cG)\bm1(\cG\in \Omega)  (d-1)^{-f_1/2}R'_{\bfi^+}\right]+\OO(N^{-\fb/4}\bE[ \Psi]),
\end{split}\end{align}
where $\fq^+=k_1+(k_1 +k_2+k_3 -2 f_0+6f_1)\geq 0$, and $R'_{\bfi^+}\in \Adm(\widehat h+f_1+1,\cF^+,\cG)$ is obtained from $\widehat R_{\bfi^+}$ by making the following substitutions: 
\begin{align}\begin{split}\label{e:final_replace}
&G_{c_{\beta_m} c}^{(b_{\beta_m} b)}\Rightarrow G_{b_{\beta_m} c}^{(b)}(G_{c_{\beta_m} c_{\beta_m}}^{(b_{\beta_m})}-Q ),
\quad 
G_{c_{\beta_m} b}^{(b_{\beta_m})}\Rightarrow G_{b_{\beta_m} b}(G_{c_{\beta_m} c_{\beta_m}}^{(b_{\beta_m})}-Q ), \quad \beta_m\in \mathsf I_{\rm single},
\\
 &   G_{c_{\gamma_{2m-1}} c_{\gamma_{2m}}}^{(b_{\gamma_{2m-1}} b_{\gamma_{2m}})}
    \Rightarrow 
 G_{b_{\gamma_{2m-1}} c_{\gamma_{2m}}}^{ (b_{\gamma_{2m}})}(G_{c_{\gamma_{2m-1}} c_{\gamma_{2m-1}}}^{(b_{\gamma_{2m-1}})}-Q ) , 
\quad \gamma_{2m-1}\in \mathsf I_{\rm single},\quad \gamma_{2m}\not\in \mathsf I_{\rm single},\\
 &   G_{c_{\gamma_{2m-1}} c_{\gamma_{2m}}}^{(b_{\gamma_{2m-1}} b_{\gamma_{2m}})}
    \Rightarrow 
 G_{c_{\gamma_{2m-1}} b_{\gamma_{2m}}}^{ (b_{\gamma_{2m-1}})}(G_{c_{\gamma_{2m}} c_{\gamma_{2m}}}^{(b_{\gamma_{2m}})}-Q ) , 
\quad \gamma_{2m-1}\notin \mathsf I_{\rm single},\quad \gamma_{2m}\in \mathsf I_{\rm single},\\
&G_{c_{\gamma_{2m-1}} c_{\gamma_{2m}}}^{(b_{\gamma_{2m-1}} b_{\gamma_{2m}})}
\Rightarrow 
G_{b_{\gamma_{2m-1}} b_{{\gamma_{2m}}}}(G_{c_{\gamma_{2m-1}} c_{\gamma_{2m-1}}}^{(b_{\gamma_{2m-1}})}-Q ) (G_{c_{{\gamma_{2m}}} c_{{\gamma_{2m}}}}^{(b_{{\gamma_{2m}}})}-Q ), 
\quad \gamma_{2m-1}, \gamma_{2m}\in \mathsf I_{\rm single}.
\end{split}\end{align}
\end{enumerate}

\end{lemma}
The first statement in \Cref{l:erbu} extends~\eqref{e:ftt1}, and the fourth claim in~\eqref{e:final_replace} extends~\eqref{e:ftt2}. 
Since the underlying ideas are analogous, we omit the proof. 
The first three claims in~\eqref{e:final_replace} essentially follow from the following estimates:
\begin{align}
\begin{split}\label{e:sum_one_index}
&\phantom{{}={}}\frac{1}{Nd}\sum_{b_\al, c_\al \in \qq{N}} 
A_{c_\al b_\al} A_{bc} G^{(b_\al b)}_{c_\al c} \\
&= 
\frac{1}{Nd}\sum_{b_\al, c_\al \in \qq{N}} 
A_{c_\al b_\al} A_{bc} 
\frac{G^{(b)}_{b_{\al} c}}{\sqrt{d-1}}
\bigl(G_{c_{\al} c_{\al}}^{(b)} - Q \bigr)
+ \text{``negligible error,''},\\
&\phantom{{}={}}\frac{1}{Nd}\sum_{b_\al, c_\al \in \qq{N}} 
A_{c_\al b_\al} A_{bc} G^{(b_\al)}_{c_\al b} \\
&= 
\frac{1}{Nd}\sum_{b_\al, c_\al \in \qq{N}} 
A_{c_\al b_\al} A_{bc} 
\frac{G_{b_{\al} b}}{\sqrt{d-1}}
\bigl(G_{c_{\al} c_{\al}}^{(b_{\al})} - Q \bigr)
+ \text{``negligible error,''}
\end{split}
\end{align}
which can be proven in the same manner as~\eqref{e:core21}. 
Again, we omit the proof.
We can repeat the above substitutions for all \(f_1\) indices in
\(\mathsf I_{\rm single}\), which appear only once among
\((\bm\alpha, \bm\beta, \bm\gamma)\) and do not appear in \(\bm\chi\).
Each substitution contributes an extra factor, and after these substitutions
\begin{align}
\widehat R_{\bfi^+}\Rightarrow (d-1)^{-f_1/2}R'_{\bfi^+},\quad R'_{\bfi^+}\in
\Adm(\widehat h+f_1+1,\cF^+,\cG),
\end{align}
and
\[
\fq^+=2k_1+k_2+k_3-2f_0+6f_1
= k_1/2+(k_1+k_2+k_3-2f_0+6f_1).
\]

We remark that in  \eqref{e:oneterm2}, if $\fq^+\geq 1$, we gain an additional factor of  $(d-1)^{-\ell/2}$. We do not obtain this extra factor only if $k_1=0, f_1=0$ and each index in $(\bm\al, \bm\beta, \bm\gamma)$ appears exactly twice without appearing in $\bm\chi$ (so $k_2+k_3=2f_0$).

\begin{proof}[Proof of the first statement in \Cref{p:general}] 
Up to a negligible error, the expression \eqref{e:higher_case3} can be rewritten as an $\OO(1)$-weighted sum of terms in the form of  \eqref{e:oneterm1}. We also refer back to the more explicit expression given in \eqref{e:oneterm0} and \eqref{e:oneterm}. 
If the assumptions in the first statement in \Cref{l:erbu} hold, there is nothing to prove. So in the rest of the proof we can focus on the second case \eqref{e:oneterm2}. 

There are several cases in which we can apply \eqref{e:oneterm2}, based on the decomposition of  $\widehat B_0$ (as in \eqref{e:fcase1W0}) using terms $\msc^{2(\ell+1)}(z )(d-1)^{-(\ell+1)}\sum_{\al\in \sfA_i} (G_{c_\al c_\al}^{(b_\al)}-Q )$, $\msc^{2(\ell+1)}(z )(d-1)^{-(\ell+1)}\sum_{\al\neq\beta\in \sfA_i} (G_{c_\al c_\beta}^{(b_\al b_\beta)}-Q )$ or a term in $\cU_0$. We treat each of these separately.

\begin{enumerate}

\item  Assume $\widehat R_{\bfi^+}$ in \eqref{e:oneterm} contains the factor $\msc^{2(\ell+1)}(z )(d-1)^{-(\ell+1)}\sum_{\al\in\sfA_i} (G_{c_\al c_\al}^{(b_\al)}-Q )$ from the decomposition of $\widehat  B_0$, then in \eqref{e:oneterm2}, $k_1\geq 1$ and $R'_{\bfi^+}$ contains a factor $(G_{c_\al c_\al}^{(b_\al)}-Q )$. Let $R'_{\bfi^+}= (G_{c_\al c_\al}^{(b_\al)}-Q )R_{\bfi^+}$, then $R_{\bfi^+}\in\Adm(\widehat h+f_1,\cF^+,\cG)$ (recall from \eqref{e:oneterm2}).  We claim that replacing $Q $ with $Y $ yields a negligible error. To see this, we write 
    \begin{align}\begin{split}\label{e:finaleq}
       &\phantom{{}={}}\frac{(d-1)^{3(\widehat h+f_1)\ell}}{(d-1)^{\fq^+\ell/2} Z_{\cF^+}}
       \sum_{\bfi^+}\bE\left[I(\cF^+,\cG)\bm1(\cG\in \Omega)  (G_{c_\al c_\al}^{(b_\al)}-Q )R_{\bfi^+}\right]\\
       &=\frac{(d-1)^{3(\widehat h+f_1)\ell}}{(d-1)^{\fq^+\ell/2} Z_{\cF^+}}\sum_{\bfi^+}\bE\left[I(\cF^+,\cG)\bm1(\cG\in \Omega)  (G_{c_\al c_\al}^{(b_\al)}-Y )R_{\bfi^+}\right]\\
       &+
       \frac{(d-1)^{3(\widehat h+f_1)\ell}}{(d-1)^{\fq^+\ell/2} Z_{\cF^+}}\sum_{\bfi^+}\bE\left[I(\cF^+,\cG)\bm1(\cG\in \Omega)  (Q -Y )R_{\bfi^+}\right].
    \end{split}\end{align}
    Note that $k_1\geq 1$, so $\fq^+\geq 1$, and the second term on the right-hand side of \eqref{e:finaleq} is bounded as
\begin{align*}\begin{split}
        &\phantom{{}={}}\frac{(d-1)^{3(\widehat h+f_1)\ell}}{(d-1)^{\fq^+\ell/2} Z_{\cF^+}}\sum_{\bfi^+}\bE\left[I(\cF^+,\cG)\bm1(\cG\in \Omega)  |Q -Y ||R_{\bfi^+}|\right]\\
        &\lesssim 
         \frac{(d-1)^{3(\widehat h+f_1)\ell}N^{-(\widehat h+f_1)\fb}}{(d-1)^{\fq^+\ell/2} }\bE\left[\bm1(\cG\in \Omega)  |Q -Y |\right]\lesssim N^{-\fb/2}\bE[  \Psi],
    \end{split}\end{align*}
    where in the first inequality we used \eqref{e:Bsmall}; in the second inequality we used \eqref{e:defPhi} and $\wh h\geq r\geq 2$.
    The first term on the right-hand side of \eqref{e:finaleq} is in the form of \eqref{e:case3_copy}, by setting $r^+= \widehat h+f_1\geq \widehat h\geq r$.

\item \label{ii:offab} If \eqref{e:oneterm} contains the factor $\msc^{2(\ell+1)}(z )(d-1)^{-(\ell+1)}\sum_{\al\neq \beta\in\sfA_i} G_{c_\al c_{\beta}}^{(b_\al b_{\beta})}$ from the decomposition of $\widehat  B_0$, then $k_3\geq 2$ in \eqref{e:oneterm2}.  There are several cases for the factor $G_{c_\al c_{\beta}}^{(b_\al b_{\beta})}$ contained in $\widehat R_{\bfi^+}$.

    
    If in  $\widehat R_{\bfi^+}$ (from \eqref{e:oneterm}) $\al\in \mathsf I_{\rm single}$, then $f_1 \geq 1$ in \eqref{e:oneterm2}. Also, \eqref{e:final_replace} implies that $R'_{\bfi^+}$ contains $G_{c_\al c_\al}^{(b_\al)}-Q $. By the same argument as in \eqref{e:finaleq}, this leads to \eqref{e:case3_copy} by setting $r^+= \widehat h+f_1\geq r+1$. The same conclusion holds if $\beta\in \mathsf I_{\rm single}$. 
    
   In the remaining cases $\al, \beta\not\in \mathsf I_{\rm single}$. There are again two cases: either $\widehat R_{\bfi^+}$ in \eqref{e:oneterm} contains at least two terms in the form $\{G_{cc'}^{(bb')}, G_{c b'}^{(b)}, G_{bb'}\}_{(b,c)\neq (b',c')\in \cK^+}$, or $\widehat R_{\bfi^+}$ in \eqref{e:oneterm} contains both factors $(G_{c_\al c_\al}^{(b_\al)}-Q ), (G_{c_{\beta} c_{\beta}}^{(b_{\beta})}-Q )$.
    
   In the first case, if $\widehat R_{\bfi^+}$ in \eqref{e:oneterm} contains at least two terms in the form $\{G_{cc'}^{(bb')}, G_{c b'}^{(b)}, G_{bb'}\}_{(b,c)\neq (b',c')\in \cK^+}$, so does \eqref{e:oneterm2}. By our assumption $r\geq 2$, we have $\widehat h+f_1+1\geq \widehat h+1 \geq r+1\geq 3$. We remark that this is the only point in the argument where the assumption $r\geq 2$ is required; all other parts of the proof remain valid for $r\geq 1$. Then \eqref{e:refined_bound} (with $r$ taking value $\widehat h+f_1+1\geq 3$)  implies that \eqref{e:oneterm2} is bounded by $\OO(N^{-\fb/2}\bE[ \Psi])$. 

 In the second case, $\widehat R_{\bfi^+}$ in \eqref{e:oneterm} contains factors $(G_{c_\al c_\al}^{(b_\al)}-Q ), (G_{c_{\beta} c_{\beta}}^{(b_{\beta})}-Q )$. Then they are either contained in $E(\bm\chi)$ or $k_1\geq 1$. In both cases $\fq^+\geq 1$ in \eqref{e:oneterm2}. The same argument as in \eqref{e:finaleq} leads to  \eqref{e:case3_copy}, by setting $r^+= \widehat h+f_1\geq \widehat h\geq r$.

\item In the remaining case, $\widehat R_{\bfi^+}$ in \eqref{e:oneterm} contains a factor $(d-1)^{3(h-1)\ell}R_{h}$ from $\cU_0$ in the decomposition  \eqref{e:fcase1W0} of $\widehat  B_0$. Here $h\geq 2$, $R_{h}$ is an $S$-product term (recall from \Cref{d:S-product}), and it contains at least one factor of the form  $( G_{c_\al c_{\al}}^{(b_\al b_\al)}-Q )$ or $G_{c_\al c_{\beta}}^{(b_\al b_\beta)}$. Moreover, in this case the factor $R_{h}$ is included in $E(\bm\chi)$ in \eqref{e:oneterm}, and $\widehat h\geq r+(h-1)\geq r+1$.

If $R_{h}$ contains at least one term of the form $(G_{c_\al c_\al}^{(b_\al)}-Q )$, by the same argument as in \eqref{e:finaleq}, \eqref{e:oneterm2} leads to \eqref{e:case3_copy} with $r^+=\widehat h+f_1\geq \widehat h\geq r+1$. 

In the other cases, $R_{h}$ contains at least one term of the form $G_{c_\al c_\beta}^{(b_\al b_\beta)}$. By the same argument as in the second statement of \Cref{l:erbu}, if $\{b_\al, c_\al\}$ do not appear in other terms of $\widehat R_{\bfi^+}$ (but $\{b_\beta, c_\beta\}$ do appear), we can replace $G_{c_\al c_\beta}^{(b_\al b_\beta)}$ by $
      G_{b_\al c_{\beta}}^{(b_{\beta})}(G_{c_\al c_\al}^{(b_\al)}-Q )/\sqrt{d-1} 
$ ); if $\{b_\beta, c_\beta\}$ do not appear in other terms of $\widehat R_{\bfi^+}$ (but $\{b_\al, c_\al\}$ do appear), we can replace $G_{c_\al c_\beta}^{(b_\al b_\beta)}$ by $
      G_{b_\beta c_{\al}}^{(b_{\al})}(G_{c_\beta c_\beta}^{(b_\beta)}-Q )/\sqrt{d-1} 
$ ); and if $\{b_\al, c_\al, b_\beta, c_\beta\}$ do not appear in other terms of $\widehat R_{\bfi^+}$, we can replace it by $
      G_{b_\al b_{\beta}}(G_{c_\al c_\al}^{(b_\al)}-Q ) (G_{c_{\beta} c_{\beta}}^{(b_{\beta})}-Q ) /(d-1)$. Moreover, the errors from such replacements are bounded by $\OO(N^{-\fb/4}\bE[ \Psi])$. 
Then by the same argument as in \Cref{ii:offab}, either \eqref{e:oneterm2} is bounded by $\OO(N^{-\fb/4}\bE[ \Psi])$, or \eqref{e:oneterm2} leads to  \eqref{e:case3_copy} with $r^+\geq \widehat h+f_1\geq \widehat h\geq r+1$.
\end{enumerate}
\end{proof}

\begin{proof}[Proof of the second statement in \Cref{p:general}] 
Up to a negligible error, the expression \eqref{e:higher_case1} can also be rewritten as an $\OO(1)$-weighted sum of terms in the form of  \eqref{e:oneterm1}. We also refer back to the more explicit expression given in \eqref{e:oneterm0} and \eqref{e:oneterm}. 
If the assumptions in the first statement in \Cref{l:erbu} hold, there is nothing to prove. So in the rest of the proof we can focus on the second case \eqref{e:oneterm2}.

We recall that in \eqref{e:higher_case1}, $r=1$ and $B_1=(G_{oo}^{(i)}-Q)$. If $\widehat h+f_1\geq 2$ (from \eqref{e:oneterm2}), we can proceed in exactly the same manner as in the proof of \eqref{e:higher_case3}.
Otherwise, $\widehat h=1$ and $f_1=0$. We assume this scenario in the following discussion.
\begin{enumerate}
    \item 

Assume $\widehat R_{\bfi^+}$ in \eqref{e:oneterm} contains $\msc^{2\ell+2}(z )(d-1)^{-(\ell+1)}\sum_{\al\in\sfA_i} (G_{c_\al c_{\al}}^{(b_\al)}-Q )$ from the decomposition \eqref{e:fcase1W0} of $\widehat  B_0$. 
Since $\widehat h=1, f_1=0$, $\widehat R_{\bfi^+}$ also contains $\msc^{2(\ell+1)}(z )(d-1)^{-(\ell+1)}\sum_{\al\in\sfA_i} (G_{c_\al c_{\al}}^{(b_\al)}-Q )$ from the decomposition of $\widehat  B_1$ (we recall the precise coefficients from \eqref{e:easy_G-Y}). Then \eqref{e:oneterm2} is of the form 
       \begin{align}\label{e:main1}
       \frac{\msc^{4(\ell+1)}(z )}{(d-1)^{2(\ell+1)}}\sum_{\al\in \sfA_i}\sum_{\bfi^+} \frac{1}{Z_{\cF^+}} \bE[\bm1(\cG\in \Omega)I(\cF^+, \cG)(G_{c_\al c_\al}^{(b_\al)}-Q )^2 ].
    \end{align}
After replacing a copy of $(G_{c_\al c_\al}^{(b_\al)}-Q )$ by $(G_{c_\al c_\al}^{(b_\al)}-Y )$, \eqref{e:main1} is an $\OO(1)$-weighted sum of terms in the form \eqref{e:higher_case11} with $\fq^+=2$, and the error is bounded by $\OO(N^{-\fb/2}\bE[\Psi  ])$. 

\item If  \eqref{e:oneterm}  contains $\msc^{2(\ell+1)}(z )(d-1)^{-(\ell+1)}\sum_{\al\neq \beta\in\sfA_i} G_{c_\al c_{\beta}}^{(b_\al b_{\beta})}$ from the decomposition \eqref{e:fcase1W0} of $\widehat B_0$. Since $\widehat h=1, f_1=0$, $\widehat R_{\bfi^+}$ also contains $\msc^{2(\ell+1)}(z )(d-1)^{-(\ell+1)}\sum_{\al\neq \beta\in\sfA_i} G_{c_\al c_{\al}}^{(b_\al b_\beta)}$ from the decomposition of $\widehat  B_1$ (we recall the precise coefficients from \eqref{e:easy_G-Y}). Then \eqref{e:oneterm2} is of the form 
 \begin{align}\label{e:main2}
       \frac{2\msc^{4(\ell+1)}(z )}{(d-1)^{2(\ell+1)}}\sum_{\al\neq\beta\in \sfA_i}\sum_{\bfi^+} \frac{1}{Z_{\cF^+}} \bE[\bm1(\cG\in \Omega)I(\cF^+, \cG)(G_{c_\al c_\beta}^{(b_\al b_\beta)})^2].
    \end{align}
By \eqref{e:refined_bound} with \(r=2\), the quantity in \eqref{e:main2} is bounded by
\(\OO\bigl(N^{\fo}\,\bE[\Psi]\bigr)\), where we also used that the factor
\((d-1)^{-2(\ell+1)}\) is canceled by the summation over \(\alpha\neq\beta\),
which contains \(\OO\bigl((d-1)^{2(\ell+1)}\bigr)\) terms.

\item
In the remaining case, $\widehat R_{\bfi^+}$ in \eqref{e:oneterm} contains a factor $(d-1)^{3(h-1)\ell}R_{h}$ with $h\geq 2$ from $\cU_0$ in the decomposition \eqref{e:fcase1W0} of $\widehat  B_0$. Then $\widehat h\geq r+1\geq 2$. 
\end{enumerate}

\end{proof}

\section{Proof of the first loop equation}
\label{s:first_loop}
In this section, we prove the first loop equation \eqref{e:Qrefined_bound}. We need to identify the leading order error terms from  \Cref{p:iteration} and \Cref{p:general}. These refined estimates are presented in the following three propositions.

\begin{proposition}\label{p:track_error1}
Adopt the notation and assumptions in \Cref{p:iteration}, and define the index set $\sfA_i := \{ \alpha \in \qq{\mu} : \dist_{\cT}(i, l_\al) = \ell+1 \}$ (see \Cref{fig:Ai}). We recall the local Green's functions $L$ and $ L^{(i)}$ (with vertex $i$ removed) from \eqref{e:local_Green} and \eqref{e:defLi}.
$I_1$ in \eqref{e:IFIF} is explicitly given by
\begin{align}\label{e:case1 erm}
    I_1=\sum_{\al\in \sfA_i}\sum_{\bfi^+}\frac{\msc^{2(\ell+1)}(z )L_{l_\al l_\al}^{(i)}}{(d-1)^{\ell+2}Z_{\cF^+}}\bE\left[I(\cF^+,\cG)\bm1(\cG\in \Omega) (G_{c_\al c_\al}^{(b_\al)}-Y)(G_{c_\al c_\al}^{(b_\al)}-Q)\right].
\end{align}

We have the following refined expression for the error term $\cE$ in \eqref{e:IFIF}
\begin{align}\label{e:daerrorE}
    \cE=\eqref{e:first erm0}+\eqref{e:second erm0}+\OO(N^{-\fb/4}  \Psi),
\end{align}
where 
\begin{align}
     \eqref{e:first erm0}   &\label{e:first erm0}=\sum_{\al,\beta\in \sfA_i}\sum_{\bfi^+}\frac{\msc^{2(\ell+1)}(z )}{(d-1)^{\ell+1}Z_{\cF^+}}\bE\left[I(\cF^+,\cG)\bm1(\cG,\wt \cG\in \Omega)(\wt G_{c_\al c_\beta}^{(\bT)}-G_{c_\al c_\beta}^{(b_\al b_\beta)}) \right],\\
       \eqref{e:second erm0} &\label{e:second erm0}=\sum_{\al\in \sfA_i,\beta\in \qq{\mu}\atop 
       \al\neq  \beta}\sum_{\bfi^+}\frac{\msc^{2(\ell+1)}(z )(L_{l_\beta l_\beta}^{(i)}+L_{l_\al l_\beta}^{(i)})}{(d-1)^{\ell+2}Z_{\cF^+}}\bE\left[I(\cF^+,\cG)\bm1(\cG\in \Omega)(G_{c_\al c_\beta}^{(b_\al b_\beta)})^2\right],
\end{align}
\end{proposition}

\begin{proposition}\label{p:track_error2}
Adopt the notation and assumptions in \Cref{p:general}. The error from expanding $I_1$ (from \eqref{e:case1 erm}) as in \eqref{e:higher_case1}
is given by 
\begin{align}\begin{split}\label{e:track_error2}
      &\left(1-\left(\frac{\msc(z )}{\sqrt{d-1}}\right)^{2\ell+2}\right)   \frac{2\md(z )\msc^{6(\ell+1)}(z )}{(d-1)^{2\ell+3}} \sum_{\al\neq\beta\in \sfA_i}\sum_{\bfi^+} \frac{1}{Z_{\cF^+}} \bE[\bm1(\cG\in \Omega)I(\cF^+, \cG)(G_{c_\al c_\beta}^{(b_\al b_\beta)})^2 ]
      +\OO(N^{-\fb/4} \bE[\Psi]).
\end{split}\end{align}
\end{proposition}

For $z$ close to the spectral edge $\pm 2$, the following proposition gives refined estimates for the error terms in \eqref{e:first erm0}, \eqref{e:second erm0} and \eqref{e:track_error2}.
\begin{proposition}\label{l:first erm}
Adopt the notation and assumptions in \Cref{p:iteration}, and recall $\cA$ from \eqref{e:KMdistribution}. For $z\in \bf D$ (recall from \eqref{e:D}) and $|z-2|\leq N^{-\fg}$, we have the following estimates for the terms involved in the error \eqref{e:daerrorE}: 
\begin{align}\label{e:first erm}
   \eqref{e:first erm0}&=\left(\frac{d(d-1)^{\ell}}{d-2} -\frac{d}{d-2}\right)\frac{1}{\cA^2}\bE\left[\bm1(\cG\in \Omega)\frac{\del_z m_N(z)}{N}\right]+\OO\left(\frac{\bE[\Psi]}{(d-1)^\ell} \right),\\
\label{e:refine_Gccerror}
    \eqref{e:second erm0}&=\left(\frac{d+2}{d-2}-\frac{d(d-1)^{\ell}}{d-2}-(\ell+1)\right)\frac{1}{\cA^2}\bE\left[\bm1(\cG\in \Omega)\frac{\del_z m_N(z)}{N}\right]+\OO\left(\frac{\bE[\Psi]}{(d-1)^\ell} \right),
\end{align}
Moreover, the error \eqref{e:track_error2} satisfies
\begin{align}\label{e:GIGG3}
       \eqref{e:track_error2}= -\frac{2}{d-2}\frac{1}{\cA^2}\bE\left[\bm1(\cG\in \Omega)\frac{\del_z m (z)}{N}\right]+\OO\left(\frac{\bE[ \Psi]}{(d-1)^\ell} \right).
\end{align}
If $|z+2|\leq N^{-\fg}$, analogous statements hold after multiplying the right-hand sides by $-1$. 
\end{proposition}


\begin{proof}[Proof of \Cref{t:recursion}] 
We will prove \eqref{e:Qrefined_bound} only for $|z-2|\leq N^{-\fg}$, the other case $|z+2|\leq N^{-\fg}$, can be established in exactly the same way. 
To prove \eqref{e:Qrefined_bound}, we must track the errors from the iteration process more carefully.  These refined error estimates are presented in \Cref{p:track_error1}, \Cref{p:track_error2} and \Cref{p:track_error2}.
By adding \eqref{e:first erm} and \eqref{e:refine_Gccerror},  the error $\cE$ from \Cref{p:iteration} is given by
\begin{align}\begin{split}\label{e:final_error1}
&-\left(\ell+1-\frac{2}{d-2}\right)\frac{1}{\cA^2}\bE\left[\bm1(\cG\in \Omega)\frac{\del_z m_N(z)}{N}\right]+\OO\left(\frac{\bE[\Psi]}{(d-1)^\ell} \right).
\end{split}\end{align}

From \Cref{p:general}, the error from expanding \eqref{e:higher_case3} is small, i.e. bounded by $\OO(N^{-\fb/4}\bE[\Psi]$.
The errors from expanding \eqref{e:higher_case1} with $\fq\geq 1$ are bounded by $\OO((d-1)^{-\ell/2}N^\fo\bE[\Psi])$. For $\fq=0$, the error from expanding \eqref{e:higher_case1} is given in \eqref{e:track_error2} and \eqref{e:GIGG3}
\begin{align}\label{e:final_error2}
 -\frac{2}{d-2}\frac{1}{\cA^2}\bE\left[\bm1(\cG\in \Omega)\frac{\del_z m_N(z)}{N}\right]+\OO\left(\frac{\bE[\Psi]}{(d-1)^\ell} \right).
\end{align}
The correction terms in the microscopic loop equation \eqref{e:Qrefined_bound} is obtained by summing the refined errors from \eqref{e:final_error1} and \eqref{e:final_error2}.

\end{proof}

\begin{proof}[Proof of \Cref{p:track_error1}]
The first statement \eqref{e:case1 erm} is from \eqref{e:newterm3}.
The decomposition of the error $\cE$ as in \eqref{e:daerrorE} is from the first term in \eqref{e:Eterm1} and $J_2$ in \eqref{e:defI2}.
\end{proof}

\begin{proof}[Proof of \Cref{p:track_error2}]

We recall $I_1$ from \eqref{e:case1 erm}. Conditioned on $I(\cF^+,\cG)=1$, the expectation in \eqref{e:case1 erm} does not depend on $\al$.  Moreover, $|\sfA_i|=(d-1)^{\ell+1}$, and by \eqref{e:Gtreemsc}, $L_{l_\al l_\al}^{(i)}=\md(z )(1-(-\msc(z )/\sqrt{d-1})^{2\ell+2})$.
We denote $(i,o)=(b_\al,c_\al)$ and $(\cF, \bfi)=(\cF^+, \bfi^+)$, and rewrite $I_1$ from \eqref{e:case1 erm} as
\begin{align}\label{e:I1_rewrite}
\left(1-\left(\frac{\msc(z )}{\sqrt{d-1}}\right)^{2\ell+2}\right)\frac{\md(z )\msc^{2(\ell+1)}(z ) }{(d-1)Z_{\cF}}\sum_{\bfi}\bE\left[I(\cF,\cG)\bm1(\cG\in \Omega) (G_{oo}^{(i)}-Y )(G_{oo}^{(i)}-Q )\right],
\end{align}
which is in the form of \eqref{e:higher_case1}, up to the constant. 

 From the proof of \eqref{e:higher_case1} in \Cref{p:general}, the errors from expanding \eqref{e:higher_case1} are either bounded by $\OO(N^{-\fb/4} \bE[ \Psi])$, or given by \eqref{e:main2}. 
 Thus, the error from expanding \eqref{e:I1_rewrite} 
is given by 
\begin{align*}
   &\left(1-\left(\frac{\msc(z )}{\sqrt{d-1}}\right)^{2\ell+2}\right)   \frac{2\md(z )\msc^{6(\ell+1)}(z )}{(d-1)^{2\ell+3}}\sum_{\al\neq\beta\in \sfA_i}\sum_{\bfi^+} \frac{1}{Z_{\cF^+}} \bE[\bm1(\cG\in \Omega)I(\cF^+, \cG)(G_{c_\al c_\beta}^{(b_\al b_\beta)})^2]
      +\OO(N^{-\fb/4} \bE[ \Psi]).
\end{align*}
This finishes the proof of \eqref{e:track_error2}.

\end{proof}

\begin{proof}[Proof of \Cref{l:first erm}]
To illustrate the basic ideas, we will only prove \eqref{e:refine_Gccerror}.

First, note that the expectation is independent of the choice of $(\alpha,\beta)$.
Therefore, we may first sum over $\alpha\in\sfA_i$ with $\alpha\neq \beta$ to compute
the corresponding coefficient.

We will use that for $|z\pm 2|\le N^{-\fg}$,
\begin{align}\begin{split}\label{e:mscmd}
\msc(z) &= -1+\OO\!\left(\sqrt{|z-2|}\right)= -1+\OO(N^{-\fg/2}),\\
\md(z)  &= -\frac{d-1}{d-2}+\OO\!\left(\sqrt{|z-2|}\right)
        = -\frac{d-1}{d-2}+\OO(N^{-\fg/2}).
\end{split}\end{align}
Moreover, on the event $I(\cF^+,\cG)=1$, the vertex $o$ has a tree neighborhood of
radius $\fR$. In particular,
\[
\mu=d(d-1)^{\ell},\qquad |\sfA_i|=(d-1)^{\ell+1},
\]
and $L^{(i)}_{l_\beta l_\beta}$ as well as $L^{(i)}_{l_\alpha l_\beta}$ are given
explicitly by the Green's function on the $d$-regular tree (see \Cref{greentree}).
A direct computation then yields the following identity, whose proof we omit:
\begin{align}\label{e:thesum}
 \sum_{\substack{\alpha\in \sfA_i,\ \beta\in \qq{\mu}\\ \alpha\neq \beta}}
 \frac{\msc^{2(\ell+1)}(z)\bigl(L_{l_\beta l_\beta}^{(i)}+L_{l_\alpha l_\beta}^{(i)}\bigr)}{(d-1)^{\ell+2}}
 =
 \left(\frac{d+2}{d-2}-\frac{d(d-1)^{\ell}}{d-2}-(\ell+1)
 +\OO\bigl((d-1)^{-\ell}\bigr)\right).
\end{align}

By plugging \eqref{e:thesum} into  \eqref{e:second erm0}, we get
    \begin{align}\begin{split}\label{e:plugin}
   \eqref{e:second erm0} &=\left(\frac{d+2}{d-2}-\frac{d(d-1)^{\ell}}{d-2}-(\ell+1)+\OO((d-1)^{-\ell})\right) \sum_{\bfi^+}\frac{1}{Z_{\cF^+}}\bE\left[I(\cF^+,\cG)\bm1(\cG\in \Omega) (G_{c_\al c_\beta}^{(b_\al b_\beta)})^2\right]
   \end{split}
   \end{align}
 Next we show 
 \begin{align}\label{e:averageGG}
    &\sum_{\bfi^+}\frac{1}{Z_{\cF^+}}\bE\left[I(\cF^+,\cG)\bm1(\cG\in \Omega) (G_{c_\al c_\beta}^{(b_\al b_\beta)})^2\right]=\frac{1}{\cA^2}\bE\left[\bm1(\cG\in \Omega)\frac{\del_z m (z)}{N}\right]+\OO\left(N^{-\fb/2}\bE[ \Psi]\right),
\end{align}
and  \eqref{e:refine_Gccerror} follows from combining \eqref{e:plugin} and \eqref{e:averageGG}

If we temporarily ignore the indicator and the averaging over embeddings, writing $(b_\al, c_\al, b_\beta, c_\beta)$ as $(b,c,b',c')$, the
above statement \eqref{e:averageGG} reduces to computing 
\begin{align}\label{e:remove_indices}
\frac{1}{(Nd)^2}\sum_{b\sim c\atop b'\sim c'}(G_{c c'}^{(bb')})^2=\frac{1}{\cA^2}\frac{\del_z m (z)}{N}+\OO(N^{-\fb/2} \Phi).
\end{align}

We start with the Schur complement formula \eqref{e:Schur1}
\begin{align*}
 G_{cc'}^{(bb')}=G_{cc'}-(G (G|_{\{bb'\}})^{-1} G)_{cc'},
\quad
    (G|_{\{bb'\}})^{-1}
    &=\frac{1}{G_{bb}G_{b'b'} -G_{bb'}^2}
    \left[
    \begin{array}{cc}
    G_{b'b'} & -G_{bb'}\\
    -G_{bb'} & G_{bb}
    \end{array}
    \right].
\end{align*}

Conditioned on $I(\cF^+,\cG)=1$, by \eqref{eq:infbound} and \eqref{e:mscmd}, the  terms
\begin{align}\begin{split}\label{e:centered erm}
   &G_{bc}-\frac{\sqrt{d-1}}{d-2},\quad  G_{b'c'}-\frac{\sqrt{d-1}}{d-2}, \quad 
   G_{b b'},\quad  G_{b c'},\quad  G_{c b'},\quad G_{c c'}\\
   &G_{bb}+\frac{d-1}{d-2},\quad  G_{c c}+\frac{d-1}{d-2}, \quad  G_{b'b'}+\frac{d-1}{d-2}, \quad 
 G_{c'c'}+\frac{d-1}{d-2},
\end{split}\end{align}
are all bounded by $\OO(N^{-\fb})$.
The Schur complement formula \eqref{e:Schurixj} imply
\begin{align}\begin{split}\label{e:remove_one}
  G_{c c'}^{(b b')}
    &=G_{cc'}-\frac{G_{cb}G_{bc'}}{G_{bb}}-\frac{G_{cb'}G_{b'c'}}{G_{b'b'}} +\frac{G_{cb}G_{bb'}G_{b'c'}}{G_{bb}G_{b'b'}}+ \OO\left(|G_{bb'}|^2+|G_{bc'}|^2+|G_{cb'}|^2+|G_{cc'}|^2\right) \\ 
&
    = G_{c c'}-\frac{ G_{b c'}}{\sqrt{d-1}}-\frac{ G_{c b'}}{\sqrt{d-1}}+\frac{ G_{b b'}}{d-1} +\cE,
\end{split}\end{align}
where $|\cE|\lesssim N^{-\fb} (|G_{b b'}|+|G_{b c'}|+|G_{c b'}|+|G_{c c'}|)$.

By plugging \eqref{e:remove_one} into the first statement in \eqref{e:remove_indices}, we get
\begin{align}\begin{split}\label{e:square1}
   \frac{1}{(Nd)^2}\sum_{b\sim c\atop b'\sim c'}(G_{c c'}^{(bb')})^2
   &=\frac{1}{(Nd)^2}\sum_{b\sim c\atop b'\sim c'}\left(G_{c c'}-\frac{ G_{b c'}}{\sqrt{d-1}}-\frac{ G_{c b'}}{\sqrt{d-1}}+\frac{ G_{b b'}}{d-1} \right)^2 +\OO(N^{-\fb/2}\Phi)\\
   &=\frac{1}{\cA^2}\frac{\del_z m (z)}{N}+\OO(N^{-\fb/2} \Phi).
\end{split}\end{align}
where in the first statement we used the Ward identity \eqref{eq:rwardex} to bound the error; and the second statement follows from repeated using the following relations, we omit further details.
\begin{align}\begin{split}\label{e:sum_neighbor}
    &\frac{1}{\sqrt{d-1}}\sum_{x\sim u}G_{xv}
    =(HG)_{uv}=zG_{uv}+\delta_{uv}=2G_{uv}+\delta_{uv}+\OO(N^{-\fg}|G_{uv}|),\\
    &\frac{1}{N^2}\sum_{x,y}G^2_{xy}=\frac{\Tr[G^2]}{N^2}=\frac{\del_z m (z)}{N}.
\end{split}\end{align}
  \end{proof}

\section{Error from local resampling}
\label{s:change_est}
In this section we prove the estimates for the error terms arising from the local resampling used to bound \eqref{e:e4}. A key ingredient is the following \emph{punctured-vertex} Ward bound.

\begin{proposition}
    \label{lem:deletedalmostrandom}
 We take $z\in {\bf D}$ (recall from \eqref{e:D}), and recall the indicator functions 
 \begin{align*}
   I(\{o,i\},\cG)= A_{oi}\bm1(\cB_{\fR}(o, \cG) \text{ is a tree}),
\end{align*}
  from \eqref{e:defineI}. 
    Then the following holds for $i,j\sim o$ and $i\neq j$, 
    \begin{align}\label{e:sameasdisconnect}
        \frac{1}{N}\sum_{o\in \qq{N}}\bE[I(\{i,o\},\cG) |G^{(o)}_{ij}|^2  ]\lesssim \bE[N^\fo \Phi  ].
    \end{align}
\end{proposition}

We recall the resampling data ${\bf S}=\{(l_\al, a_\al), (b_\al, c_\al)\}_{\al\in\qq{\mu}}$ around $o$ from \Cref{s:local_resampling},  and let
\begin{equation}\label{e:cF_copy}
 \cF:=\{i,o\},\quad \cF^+ :=B_{\ell}(o, \cG)\cup \{(l_\al, a_\al), (b_\al, c_\al)\}_{\al\in\qq{\mu}}=B_{\ell+1}(o, \cG)\cup \{ (b_\al, c_\al)\}_{\al\in\qq{\mu}}=(\bfi^+ , E^+ ),
\end{equation}
which contains all the switching edges, see \Cref{fig:FG}. We also recall the following indictor function from \eqref{e:defI}
\begin{align}\begin{split}\label{e:defI_copy}
&I(\cF ,\cG)
:=A_{io}\prod_{x\in \cB_\ell(o;\cG)}
      \bm{1}\!\big(\cB_{\fR}(x;\cG)\ \text{is a tree}\big)\\
& I(\cF^+ ,\cG)
:= \prod_{\{x,y\}\in E^+ } A_{xy}
   \;\prod_{c\in \{o,c_\al,\cdots, c_\mu\}\atop x\in \cB_\ell(c;\cG)}
      \bm{1}\!\big(\cB_{\fR}(x;\cG)\ \text{is a tree}\big)
   \; \prod_{\substack{c\neq c'\in \{o,c_1,\cdots, c_\mu\}}}      \bm{1}\!\big(\dist_\cG(c,c')\ge 3\fR\big).
\end{split}\end{align}

As a consequence of \Cref{lem:deletedalmostrandom}, the following proposition states that during the local resampling, the errors $\cE$ from \Cref{l:coefficient} (after averaging) are negligible. 
\Cref{lem:task2} follows from \Cref{lem:deletedalmostrandom} and Schur complement formula \eqref{e:Schur1}. 
The proofs of \Cref{lem:deletedalmostrandom} and \Cref{lem:task2} will be given in \Cref{s:removeonevertex}.
\begin{proposition}

\label{lem:task2}
 We take $z\in {\bf D}$, and denote the resampling data ${\bf S}=\{(l_\al, a_\al), (b_\al, c_\al)\}_{\al\in\qq{\mu}}$,  the following holds
\begin{align}\begin{split}\label{eq:task2}
    &\frac{1}{Z_{\cF^+}}\sum_{\bfi^+}\bE[I(\cF^+,\cG)\bm1(\cG\in \Omega)|\widetilde G^{(\bT)}_{c_\alpha c_\beta}-G_{c_\alpha c_\beta}^{(b_\alpha b_\beta)}| ]\lesssim(d-1)^{\ell}\bE[N^\fo\Phi  ].
   \end{split}
    \end{align} 
 
\end{proposition}

\subsection{Schur complement formula revisit}
\label{s:Schur_revisit}
Adopt the notations as in the proof of \Cref{l:coefficient}, and condition on that $I(\cF^+,\cG)=1$, so the switching edges have tree neighborhood and are far away from each other. Then the normalized adjacency matrix $\widetilde H^{(\bT)}$ of $\tcG^{(\bT)}$ is in the block form
\begin{align*}
    \widetilde H^{(\bT)}=
    \left[
    \begin{array}{cc}
        \wt H^{(\bT)}_{\bW} & \wt B^\top\\
        \wt B & \widetilde H^{(\bT)}_{\bW^\complement}
    \end{array}
    \right].
\end{align*}
We also denote the Green's function of $\cG^{(\bT)}$ and $\wt\cG^{(\bT)}$ as $ G^{(\bT)}$ and $\wt G^{(\bT)}$ respectively.

In this section, we investigate the error from replacing $\wt G^{(\bT)}_{c_\al c_\beta}$ with $ G_{c_\al c_\beta}^{(b_\al b_\beta)}$. 
We notice that $ G_{c_\al c_\beta}^{(b_\al b_\beta)}$ can be obtained from $\wt G^{(\bT)}_{c_\al c_\beta}$ through the following steps. First, we remove $\bW=\{b_1, b_2, \cdots, b_\mu\}$, which gives $G_{c_\al c_\beta}^{(\bT \bW)}$; we then add $\bW\setminus \{b_\al, b_\beta\}$ back, which gives $G_{c_\al c_\beta}^{(\bT b_\al b_\beta)}$; finally we add $\bT$ back, which gives $G_{c_\al c_\beta}^{(b_\al b_\beta)}$. The errors from these replacements are explicit, thanks to the Schur complement formulas \eqref{e:Schur1}:
\begin{align}
\label{e:removeW}&\wt G^{(\bT)}_{c_\al c_\beta}-G^{(\bT\bW)}_{c_\al c_\beta}=(G^{(\bT\bW)}\wt B  \wt G^{(\bT)}|_{\bW} {\wt B}^\top G^{(\bT\bW)})_{c_\al c_\beta},\\
 \label{e:addW}&G^{(\bT\bW)}_{c_\al c_\beta}-G^{(\bT b_\al b_\beta)}_{c_\al c_\beta}=-(G^{(\bT\bW)}(G^{(\bT b_\al b_\beta)}|_{\bW\backslash\{b_\alpha,b_\beta\}})^{-1}   G^{(\bT b_\al b_\beta)})_{c_\al c_\beta},\\
 \label{e:schur_removeT}
&G_{c_\al c_\beta}^{(b_\al b_\beta)}-G^{(\bT b_\al b_\beta)}_{c_\al c_\beta}=(G^{(b_\al b_\beta)}(G^{(b_\al b_\beta)}|_\bT)^{-1}G^{(b_\al b_\beta)})_{c_\al c_\beta}.
\end{align}

The following lemma provide leading order terms for  the replacement errors associated with the above equations \eqref{e:removeW}, \eqref{e:addW} and \eqref{e:schur_removeT}. 
\begin{lemma}\label{l:diffG1}
Fix $z\in \bf D$ and recall $\cF, \cF^+$ from \eqref{e:cF_copy}. We denote $\cT=\cB_\ell(o,\cG)$ with vertex set $\bT$. We assume that $\cG, \tcG\in \Omega$ and $I(\cF^+,\cG)=1$. For any indices $\al,\beta\in\qq{\mu}$, the following holds: 
\begin{enumerate}
\item   
    The difference $\wt G^{(\bT)}_{c_\al c_\beta}-G^{(\bT\bW)}_{c_\al c_\beta}$, is given by
    \begin{align}\label{e:diffG1}
    \wt G^{(\bT)}_{c_\al c_\beta}-G^{(\bT\bW)}_{c_\al c_\beta}=\frac{\md(z )}{d-1}\sum_{\gamma\in\qq{\mu}}\left( \sum_{x\in \cN_\gamma}G_{c_\al x}^{(\bT\bW)}\right)\left(\sum_{x\in \cN_\gamma}G_{c_\beta x}^{(\bT\bW)}\right)+\cE,
    \end{align}
   where $\cN_\gamma=\{x\neq c_\gamma: x\sim b_\gamma \text{ in }\cG\}\cup \{a_\gamma\}$, which enumerates the adjacent vertices of $b_\gamma$ in $\tcG$, and
    \begin{align*}
     |\cE|\lesssim N^{-\fb/2}\sum_{\gamma\in\qq{\mu}}\sum_{x\in \cN_\gamma}(|G_{c_\al x}^{(\bT\bW)}|^2+|G_{c_\beta x }^{(\bT\bW)}|^2)+N^{-\fb }\Phi.
    \end{align*}
    As a consequence of \eqref{e:diffG1}, we have
    \begin{align}\label{e:total_diffG1}
    |\wt G^{(\bT)}_{c_\al c_\beta}-G^{(\bT\bW)}_{c_\al c_\beta}|\lesssim \sum_{\gamma\in\qq{\mu}} \sum_{x\in \cN_\gamma}(|G_{c_\al x}^{(\bT\bW)}|^2+|G_{c_\beta x}^{(\bT\bW)}|^2)+N^{-\fb }\Phi.
    \end{align}
\item
    The difference $G^{(\bT\bW)}_{c_\al c_\beta}-G^{(\bT b_\al b_\beta)}_{c_\al c_\beta}$ is given by
    \begin{align}\label{e:diffG2}
    G^{(\bT\bW)}_{c_\al c_\beta}-G^{(\bT b_\al b_\beta)}_{c_\al c_\beta}=-\frac{1}{\md(z )}\sum_{\gamma\in\qq{\mu}\setminus\{\al,\beta\}}G^{(\bT b_\al b_\beta )}_{c_\al b_\gamma}  G^{(\bT b_\al b_\beta)}_{b_{\gamma} c_\beta}+\cE,
    \end{align}
where
\begin{align*}
|\cE|
\lesssim N^{-\fb/2} \sum_{\gamma\in\qq{\mu}\setminus\{\al,\beta\}}|G^{(\bT b_\al b_\beta )}_{c_\al b_\gamma}|^2+|G^{(\bT b_\al b_\beta)}_{ c_\beta b_{\gamma}}|^2.
\end{align*}
    As a consequence of \eqref{e:diffG2}, we have
\begin{align}\label{e:total_diffG2}
    |G^{(\bT\bW)}_{c_\al c_\beta}-G^{(\bT b_\al b_\beta)}_{c_\al c_\beta}|\lesssim \sum_{\gamma\in\qq{\mu}\setminus\{\al,\beta\}}(|G^{(\bT b_\al b_\beta )}_{c_\al b_\gamma}|^2+ |G^{(\bT b_\al b_\beta)}_{c_\beta b_{\gamma} }|^2).
\end{align}
\item
The difference $G^{(\bT b_\al b_\beta)}_{c_\al c_\beta}-G_{c_\al c_\beta}^{(b_\al b_\beta)}$ is given by 
\begin{align}\label{e:diffG3}
G^{(\bT b_\al b_\beta)}_{c_\al c_\beta}-G_{c_\al c_\beta}^{(b_\al b_\beta)}=-\sum_{\dist(x,o)=\ell\atop  x\sim x'\in \bT}\left(\frac{1}{\sqrt{d-1}}G_{c_\al x}^{(b_\al b_\beta)}G^{(b_\al b_\beta)}_{x' c_\beta}
+\frac{1}{\msc(z )}G_{c_\al x}^{(b_\al b_\beta)}G^{(b_\al b_\beta)}_{x c_\beta}\right)+\cE,
\end{align}
where
    \begin{align*}
       |\cE|
       \lesssim N^{-\fb/2}\sum_{x\in \bT}(|G^{(b_\al b_\beta)}_{c_\al x}|^2+|G^{(b_\al b_\beta)}_{ c_\beta x}|^2 ).
    \end{align*}
    As a consequence of \eqref{e:diffG3}, we have
\begin{align}\label{e:total_diffG3} 
|G^{(\bT b_\al b_\beta)}_{c_\al c_\beta}-G_{c_\al c_\beta}^{(b_\al b_\beta)}|\lesssim 
\sum_{x\in \bT}(|G_{c_\al x}^{(b_\al b_\beta)}|^2+|G^{(b_\al b_\beta)}_{c_\beta x}|^2).
\end{align}
\item For any $x\sim b_\beta$ in $\cG$, the following holds 
\begin{align}\begin{split}\label{e:Greplace}
    &|G^{(\bT \bW)}_{c_\al x}-G_{c_\al x}^{(b_\al b_\beta)}|\lesssim \sum_{\gamma\in\qq{\mu}\setminus\{\al,\beta\}}(|G^{(\bT b_\al b_\beta )}_{c_\al b_\gamma}  |^2+|G^{(\bT b_\al b_\beta)}_{b_{\gamma} c_\beta}|^2)
    +\sum_{y\in \bT}(|G_{c_\al y}^{(b_\al b_\beta)}|^2+|G^{(b_\al b_\beta)}_{ x y}|^2).
\end{split}\end{align}
\end{enumerate}
\end{lemma}
\begin{proof}[Proof of \Cref{l:diffG1}]
We can decompose the right-hand side of \eqref{e:removeW} as given by:
\begin{align*}
&I:=\frac{1}{d-1}\sum_{\gamma,\gamma'\in \qq{\mu}}\sum_{x\in \cN_\gamma\atop y\in \cN_{\gamma'}}G_{c_\al x}^{(\bT\bW)} \wt G^{(\bT)}_{b_\gamma b_{\gamma'}}   G^{(\bT\bW)}_{yc_\beta },
\end{align*}
where $\cN_\gamma=\{x\neq c_\gamma: x\sim b_\gamma \text{ in }\cG\}\cup \{a_\gamma\}$.

Since $\tcG\in \Omega$ and $I(\cF^+,\cG)=1$, \eqref{eq:local_law1} gives that for $\gamma,\gamma'\in\qq{\mu}$, $|\wt G^{(\bT)}_{b_\gamma b_{\gamma}}-\delta_{\gamma\gamma'}\md(z)|\lesssim N^{-\fb} $. The leading order term of $I$ is then for pairs $\gamma=\gamma'$, giving
\begin{align}\label{e:II_approx}
I= \frac{\md(z )}{d-1}\sum_{\gamma\in\qq{\mu}} \left(\sum_{x\in \cN_\gamma}G_{c_\al x}^{(\bT\bW)}\right)\left(\sum_{x\in \cN_\gamma}G_{c_\beta x}^{(\bT\bW)}\right)+\OO\left(N^{-\fb} \sum_{\gamma, \gamma'\in\qq{\mu}}\sum_{x\in \cN_\gamma\atop y\in \cN_{\gamma'}}|G_{c_\al x}^{(\bT\bW)}||G_{yc_\beta }^{(\bT\bW)}|\right).
\end{align}
The claim \eqref{e:diffG1} follows from \eqref{e:II_approx}.

To prove \eqref{e:diffG2}, we start with \eqref{e:addW}. Consider
\begin{align}\begin{split}\label{e:defcE1beta}
&\phantom{{}={}}(G^{(\bT b_\al b_\beta)}(G^{(\bT b_\al b_\beta)}|_{\bW\backslash\{b_\alpha,b_\beta\}})^{-1}   G^{(\bT b_\al b_\beta)})_{c_\al c_\beta}\\
&=\sum_{\gamma, \gamma'\in\qq{\mu}\setminus\{\al,\beta\}}
G^{(\bT b_\al b_\beta)}_{c_\al b_\gamma}(G^{(\bT b_\al b_\beta)}|_{\bW\backslash\{b_\alpha,b_\beta\}})^{-1}_{b_\gamma b_{\gamma'}}   G^{(\bT b_\al b_\beta)}_{b_{\gamma'} c_\beta}.
\end{split}\end{align}
The leading order term in \eqref{e:defcE1beta} is given by those with $\gamma=\gamma'$,
\begin{align}\begin{split}\label{e:IVerror}
&\eqref{e:defcE1beta}= \frac{1}{\md(z )}\sum_{\gamma\in\qq{\mu}\setminus\{\al,\beta\}}G^{(\bT b_\al b_\beta )}_{c_\al b_\gamma}  G^{(\bT b_\al b_\beta)}_{b_{\gamma} c_\beta}+\cE,\\
&\cE:=\sum_{\gamma,\gamma'\in\qq{\mu}\setminus\{\al,\beta\}}
G^{(\bT b_\al b_\beta )}_{c_\al b_\gamma}((G^{(\bT b_\al b_\beta)}|_{\bW\backslash\{b_\alpha,b_\beta\}})^{-1}_{b_\gamma b_{\gamma'}} -\delta_{\gamma\gamma'} /\md(z )) G^{(\bT b_\al b_\beta)}_{b_{\gamma'} c_\beta}.
\end{split}\end{align}
Thanks to \eqref{eq:local_law}, $|G^{(\bT b_\al b_\beta)}_{b_\gamma b_{\gamma'}}-\md(z )\delta_{\gamma\gamma'}|\lesssim N^{-\fb}$. Thus 
$|(G^{(\bT b_\al b_\beta)}|_{\bW\backslash\{b_\alpha,b_\beta\}})^{-1}_{b_\gamma b_{\gamma'}} -\delta_{\gamma\gamma'} /\md(z )|\lesssim N^{-3\fb/4} $, and  the error $\cE$ in \eqref{e:IVerror} is bounded as
\begin{align}\label{e:IVerror2}
   |\cE|\lesssim N^{-3\fb/4} \sum_{\gamma,\gamma'\in\qq{\mu}\setminus\{\al,\beta\}}|G^{(\bT b_\al b_\beta )}_{c_\al b_\gamma}||G^{(\bT b_\al b_\beta)}_{b_{\gamma'} c_\beta}|\lesssim N^{-\fb/2} \sum_{\gamma\in\qq{\mu}\setminus\{\al,\beta\}}|G^{(\bT b_\al b_\beta )}_{c_\al b_\gamma}|^2+|G^{(\bT b_\al b_\beta)}_{ c_\beta b_{\gamma}}|^2.
\end{align}
The claim \eqref{e:diffG2} follows from combining \eqref{e:IVerror} and \eqref{e:IVerror2}.

To prove \eqref{e:diffG3}, we can rewrite the right-hand side of \eqref{e:schur_removeT} explicitly as
\begin{align}\label{e:schur_removeT2}
 -(G^{(b_\al b_\beta)}(G^{(b_\al b_\beta)}|_\bT)^{-1}G^{(b_\al b_\beta)})_{c_\al c_\beta}
    =-\sum_{x,y\in \bT} G^{(b_\al b_\beta)}_{c_\al x}(G^{(b_\al b_\beta)}|_\bT)^{-1}_{xy}G^{(b_\al b_\beta)}_{y c_\beta}.
\end{align}
Since $I(\cF^+;\cG)=1$, by \eqref{eq:local_law}, for $x,y\in \bT$, $|G^{(b_\al b_\beta)}_{xy}-(H_\bT-z -\msc(z )\mathbb{I}^{\del})^{-1}_{xy}|\lesssim N^{-\fb}$, where $\mathbb{I}^\del_{xy}=\bm1(\dist_\cT(o,x)=\ell)\delta_{xy}$. Thus we have
\begin{align*}
|(G^{(b_\al b_\beta)}|_\bT)_{xy}^{-1}-(H^{(b_\al b_\beta)}-z -\msc(z )\mathbb{I}^{\del})_{xy}|\lesssim N^{-3\fb/4}, \text{ for }x,y\in \bT,
\end{align*}
and 
\begin{align}\label{e:diffG4}
    \eqref{e:schur_removeT2}=  -\sum_{x,y\in \bT}G^{(b_\al b_\beta)}_{c_\al x}(H-z -\msc(z )\mathbb{I}^\del)_{xy}G^{(b_\al b_\beta)}_{y c_\beta}+\OO\left(N^{-3\fb/4} \sum_{x,y\in \bT}|G^{(b_\al b_\beta)}_{c_\al x}||G^{(b_\al b_\beta)}_{y c_\beta}|\right).
\end{align}

For the summation over $x,y\in \bT$ in \eqref{e:diffG4},  if $x\in \bT$ but $\dist_\cT(x, o)<\ell$, then $\mathbb{I}^\del_{xy}=0$, and  we have
\begin{align}\begin{split}\label{e:HGexp}
&\phantom{{}={}}\sum_{y\in \bT}(H-z -\msc(z )\mathbb{I}^\del)_{xy}G^{(b_\al b_\beta)}_{y c_\beta}
=\sum_{y\in \bT}(H-z )_{xy}G^{(b_\al b_\beta)}_{y c_\beta}=0,
\end{split}\end{align} 
where for the last equality, we used that by the definition of the Green's function.
Thus by plugging \eqref{e:HGexp} into \eqref{e:diffG4}, it follows that 
\begin{align}\begin{split}\label{e:interior erm}
    \sum_{\dist(x,o)<\ell\atop y\in \bT}G_{c_\al x}^{(b_\al b_\beta)}(H-z -\msc(z )\mathbb{I}^\del)_{xy}G^{(b_\al b_\beta)}_{y c_\beta}=0.
\end{split}\end{align}

If $x\in \bT$ and $\dist_\cT(x, o)=\ell$, we denote the parent node of $x$ as $x'$. We then have, by the self-consistent equation of $\msc$,
\begin{align}\begin{split}\label{e:boundary erm}
&\phantom{{}={}}\sum_{y\in \bT}G_{c_\al x}^{(b_\al b_\beta)}(H-z -\msc(z )\mathbb{I}^\del)_{xy}G^{(b_\al b_\beta)}_{y c_\beta}\\
&=\frac{1}{\sqrt{d-1}}G_{c_\al x}^{(b_\al b_\beta)}G^{(b_\al b_\beta)}_{x' c_\beta}
-(z +\msc(z ))G_{c_\al x}^{(b_\al b_\beta)}G^{(b_\al b_\beta)}_{x c_\beta}\\
&=\frac{1}{\sqrt{d-1}}G_{c_\al x}^{(b_\al b_\beta)}G^{(b_\al b_\beta)}_{x' c_\beta}
+\frac{1}{\msc(z )}G_{c_\al x}^{(b_\al b_\beta)}G^{(b_\al b_\beta)}_{x c_\beta}.
\end{split}\end{align} 
The claim \eqref{e:diffG3} follows from plugging \eqref{e:interior erm} and \eqref{e:boundary erm} into \eqref{e:diffG4}.

The claims \eqref{e:Greplace} follow from the same arguments as in \eqref{e:diffG2} and \eqref{e:diffG3}, so we omit the proof. 
\end{proof}
\subsection{Proof of \Cref{lem:deletedalmostrandom} and \Cref{lem:task2}}
\label{s:removeonevertex}
\begin{proof}[Proof of \Cref{lem:deletedalmostrandom}]
We can replace the indicator function $I_o$ by $I(\cF,\cG)$ by the same argument for \Cref{c:resample}
\begin{align}\begin{split}\label{e:switching_Pi}
    &\phantom{{}={}} \frac{1}{N}\sum_{o\in \qq{N}}\bE[|G^{(o)}_{ij}|^2  ]=\frac{1}{Z_\cF}\sum_{\bfi}\bE\left[\bm1(\cG\in \Omega) I(\cF,\cG) |G^{(o)}_{ij}|^2  \right]+\OO(N^{-1+\fc})\\
    &\lesssim\frac{1}{Z_{\cF^+}}\sum_{\bfi^+}\bE\left[\bm1(\cG, \wt \cG\in \Omega) I(\cF^+,\cG) |\wt G^{(o)}_{ij}|^2   \right]+\OO(N^{-\fb}\bE[\Phi  ]).
\end{split}\end{align}

We are now left to write the Green's function of the switched graph in terms of the original graph. Let $\cT=B_\ell(o,\cG)$ and $L:=P(\cT,z ,\msc(z ))$. Here $L$ is consistent with the local Green's function  (as defined in \eqref{e:local_Green}) on $\cT$, and we use the same symbols to represent them.  We notice that since $i,j$ are distinct neighbors of $o$, so $i,j$ are in different connected components of $\cT^{(o)}$. Thus $L^{(o)}_{ij}=0$, and $\wt G^{(o)}_{ij}=\wt G^{(o)}_{ij}-L^{(o)}_{ij}$. We will use the same argument as in the proof of \Cref{l:coefficient}.   In the rest, we condition on that $\cG, \tcG\in \Omega$ and $I(\cF^+,\cG)=1$. Then by the same argument as for \eqref{eq:resolventexp}, we have
\begin{align}\label{e:GooY2}
    \wt G^{(o)}_{ij}&=\left(L^{(o)}\sum_{k=1}^\fp \left((\widetilde  B^\top(\tG^{(\bT)}-\msc(z ))\widetilde  B)  L^{(o)}\right)^k\right)_{ij}+\OO(N^{-2}).
\end{align}

For any $1\leq k\leq \fp$, the $k$-th term in \eqref{e:GooY2} is an $\OO(1)$-weighted sum of terms of the following form
\begin{align}\label{e:PUP2}
   \sum_{x_1, x_2,\cdots, x_{2k}\in \bT}L^{(o)}_{i x_1} V_{x_1 x_2} L^{(o)}_{x_2 x_3} V_{x_3 x_4} 
   L^{(o)}_{x_4 x_5}\cdots V_{x_{2k-1} x_{2k}}L^{(o)}_{x_{2k} j}.
\end{align}
Here  $V=(\wt B^\top (\wt G^{(\bT)}-\msc(z ))\wt B)$, and \eqref{eq:local_law1} gives that  $|V_{xy}|\lesssim N^{-\fb} $ for any $x, y\in \bT$. 

We recall that $i,j$ are in different connected components of $\cT^{(o)}$. For the sequence of indices $x_0:=i, x_1, x_2,\cdots, x_{2k}, x_{2k+1}:=j$, there exists some pair of consecutively listed vertices that are in different connected components of $\cT^{(o)}$. If for some $0\leq m\leq k$, $x_{2m}, x_{2m+1}$ are in different connected components of $\cT^{(o)}$, then $L^{(o)}_{x_{2m}x_{2m+1}}=0$ and \eqref{e:PUP2} vanishes. Thus we only need to consider the case that for some $1\leq m\leq k$, $x_{2m-1}, x_{2m}$ are in different connected components of $\cT^{(o)}$. 
In this case, 
\begin{align}\label{e:Vxm}
V_{x_{2m-1}, x_{2m}}=(\wt B^\top (\wt G^{(\bT)}-\msc(z ))\wt B)_{x_{2m-1}, x_{2m}}
 =\frac{1}{d-1}\sum_{\al, \beta: l_\al=x_{2m-1}, l_\beta=x_{2m}}\wt G^{(\bT)}_{c_\al c_\beta}.
 \end{align}

If $k=1$, then $m=1$ and we can compute \eqref{e:PUP2} using \eqref{e:Vxm} as
 \begin{align}\begin{split}\label{e:k=1Pxx}
     &\phantom{{}={}}\frac{1}{d-1}\sum_{\al\neq \beta\in \qq{\mu}}L^{(o)}_{i l_\al }\wt G^{(\bT)}_{c_\al c_\beta} L^{(o)}_{l_\beta j}=\frac{\fc}{(d-1)^{\ell}}\sum_{\dist_\cT(i, l_\al)=\ell-1\atop \dist_\cT(j, l_\beta)=\ell-1}\wt G^{(\bT)}_{c_\al c_\beta},
\end{split}\end{align}
where we used \eqref{e:sum_Pbound}, and $|\fc|=\OO(1)$.
For $k\geq 2$, we can bound \eqref{e:PUP2} as
\begin{align}\begin{split}\label{e:sumx1234}
    \eqref{e:PUP2} &\lesssim N^{-(k-1)\fb} \sum_{\al\neq \beta}|\wt G^{(\bT)}_{c_\al c_\beta}|\sum_{x_1, x_2,\cdots, x_{2m-2}\in \bT\atop x_{2m+1}, x_{2m+2},\cdots, x_{2k}\in \bT}|L^{(o)}_{i x_1}| 
   \cdots |L^{(o)}_{x_{2m-2} l_\al}||L^{(o)}_{l_\beta x_{2m+1}}| \cdots |L^{(o)}_{x_{2k} j}|\\
   &\lesssim
   N^{-(k-1)\fb} \sum_{\al\neq \beta}|\wt G^{(\bT)}_{c_\al c_\beta}|
   (d-1)^{\ell}(\ell(d-1)^\ell)^{k-1}\lesssim 
    \frac{1}{(d-1)^{7(k-1)\fb/8}}\sum_{\al\neq \beta}|\wt G^{(\bT)}_{c_\al c_\beta}|,
\end{split}\end{align}
where the first statement follows from \eqref{e:Vxm} and $|V_{xy}|\lesssim N^{-\fb}$;  the second statement follows from \eqref{e:sum_Pbound}; in the third statement we used $N^{\fb} \geq (d-1)^{20\ell}$.

The estimates \eqref{e:k=1Pxx} and \eqref{e:sumx1234} together lead to the following estimate for \eqref{e:GooY2} 
\begin{align}\label{e:firstbound}
   \eqref{e:GooY2}= \frac{\msc^{2\ell}(z)}{(d-1)^{\ell}}\sum_{\al\in \sfA_i\atop \beta\in \sfA_j}\wt G^{(\bT)}_{c_\al c_\beta}
    +\OO\left(\frac{1}{N^{7\fb/8}}\sum_{\al\neq \beta\in\qq{\mu}}|\wt G_{c_\al c_\beta}^{(\bT)}| +N^{-2}\right),
\end{align}
and by plugging \eqref{e:firstbound} back into \eqref{e:switching_Pi} we conclude that
\begin{align}\begin{split}\label{e:IIGU2}
    &\frac{1}{Z_\cF}\sum_{\bfi}\bE\left[\bm1(\cG\in \Omega) I(\cF,\cG) |G^{(o)}_{ij}|^2  \right]\lesssim J_1+J_2+\OO(N^{-\fb}\bE[\Phi  ])\\
    &J_1:= \frac{1}{Z_{\cF^+}}\sum_{\bfi^+}\bE\left[\bm1(\cG, \wt \cG\in \Omega) I(\cF^+,\cG)  \left|\frac{1}{(d-1)^{\ell}}\sum_{\al\neq \beta\in\qq{\mu}}\wt G^{(\bT)}_{c_\al c_\beta} \right|^2  \right]\\
    &|J_2|\lesssim \frac{1}{Z_{\cF^+}}\sum_{\bfi^+}\bE\left[\bm1(\cG, \wt \cG\in \Omega) I(\cF^+,\cG) N^{-3\fb/4} \left(\sum_{\al\neq \beta\in\qq{\mu}}|\wt G^{(\bT)}_{c_\al c_\beta}|^2+ \Phi\right) \right].
\end{split}\end{align}

Next, we estimate $J_1$ and $J_2$ as in \eqref{e:IIGU2}. We need to express $\wt G_{c_\al c_\beta}^{(\bT)}$ in terms of the Green's function of the graph $\cG$. In this process, any term that can be bounded by $\OO(N^{-\oo(1)}\Phi)$ is considered negligible, since it contributes to an error $\bE\left[{\bm1(\cG\in \Omega)}N^{-\oo(1)}\Phi   \right]=\OO(N^{-\oo(1)}\bE[\Phi ])$.

Thanks to \Cref{l:diffG1} (combining \eqref{e:total_diffG1}, \eqref{e:total_diffG2} and \eqref{e:total_diffG3}), we have
\begin{align}\label{e:replacealpha}
    \tG^{(\bT)}_{c_\al c_\beta}=G^{(b_\al b_\beta)}_{c_\al c_\beta}+\cE_{\al \beta},
\end{align}
where 
\begin{align}\begin{split}\label{e:cEbound}
    |\cE_{\al \beta}|&\lesssim  \sum_{\gamma\in\qq{\mu}, x\in \cN_\gamma}(|G_{c_\al x}^{(\bT\bW)}|^2+|G_{c_\beta x}^{(\bT\bW)}|^2)+\sum_{\gamma\in\qq{\mu}\setminus\{\al,\beta\}}(|G^{(\bT b_\al b_\beta )}_{c_\al b_\gamma}|^2+ |G^{(\bT b_\al b_\beta)}_{c_\beta b_{\gamma} }|^2)\\
    &+\sum_{x\in \bT}(|G_{c_\al x}^{(b_\al b_\beta)}|^2+|G^{(b_\al b_\beta)}_{c_\beta x}|^2)+N^{-\fb }\Phi,
\end{split}\end{align}
and $N_\gamma=\{x\neq c_\gamma: x\sim b_\gamma\}\cup\{a_\gamma\}$.

In the following we show that for $\cG\in \Omega$, and $\al\neq \beta, \al'\neq \beta'$
\begin{align}\begin{split}\label{e:smallterm}
   &\phantom{{}={}}\frac{1}{Z_{\cF^+}}\sum_{\bfi^+}I(\cF^+,\cG) (|\cE_{\al \beta}|| G^{(b_{\al'} b_{\beta'})}_{c_{\al'} c_{\beta'}}|+|G^{(b_\al b_\beta)}_{c_\al c_\beta}||\cE_{\al' \beta'}|
    +|\cE_{\al \beta}|^2)\\
    &\lesssim N^{-\fb} \sum_{\gamma\in\{\al, \beta, \al',\beta'\}\atop x\sim b_\gamma, 
x\neq c_\gamma}  \frac{1}{Z_{\cF^+}}\sum_{\bfi^+}I(\cF^+,\cG) |G^{(b_\gamma)}_{c_\gamma x}|^2+\frac{\Phi}{N^{3\fb/4 }},
\end{split}\end{align}
and thus
\begin{align}\begin{split}\label{e:GGGxy}
    &\phantom{{}={}}\frac{1}{Z_{\cF^+}}\sum_{\bfi^+}I(\cF^+,\cG)(G^{(b_\al b_\beta)}_{c_\al c_\beta}+\cE_{\al \beta})(\overline G^{(b_{\al'} b_{\beta'})}_{c_{\al'} c_{\beta'}}+\overline\cE_{\al' \beta'}) \\
    &=\frac{1}{Z_{\cF^+}}\sum_{\bfi^+}I(\cF^+,\cG)G^{(b_\al b_\beta)}_{c_\al c_\beta}\overline G^{(b_{\al'} b_{\beta'})}_{c_{\al'} c_{\beta'}}+\OO\left(\sum_{\gamma\in\{\al, \beta, \al',\beta'\}\atop x\sim b_\gamma, 
x\neq c_\gamma} \frac{N^{-\fb}}{ Z_{\cF^+}}\sum_{\bfi^+}I(\cF^+,\cG) |G^{(b_\gamma)}_{c_\gamma x}|^2+ \frac{\Phi}{N^{3\fb/4 }}\right).
\end{split}\end{align}

To prove \eqref{e:smallterm}, we start by plugging in the bound of $\cE_{\alpha \beta}$ from \eqref{e:cEbound} into the left-hand side of \eqref{e:smallterm}, after which each term contains three Green's function entries as factors. We can bound one of them by $N^{-\fb} $ using \eqref{eq:local_law}, and the remaining two can be bounded by terms in the form $\{|G_{c_\gamma x}^{(b_\gamma)}|^2\}$ with ${\gamma\in \{\al, \beta, \al', \beta'\}}, x\sim b_\gamma, x\neq c_\gamma$ or can be bounded by $N^\fo\Phi$ using \eqref{e:Gest}.  In the following we estimate the following term from \eqref{e:smallterm}, and the other terms can be bounded in the same way, so we omit arguments about them.
\begin{align}\label{e:yizhong}
\frac{1}{Z_{\cF^+}}\sum_{\bfi^+}I(\cF^+,\cG)\sum_{\gamma\in\qq{\mu}, x\in \cN_\gamma}|G_{c_\al x}^{(\bT\bW)}|^2|G^{(b_{\al'}b_{\beta'})}_{c_{\al'}c_{\beta'}}|\lesssim 
\frac{1}{Z_{\cF^+}}\sum_{\bfi^+}I(\cF^+,\cG)\sum_{\gamma\in\qq{\mu}, x\in \cN_\gamma}N^{-\fb}|G_{c_\al x}^{(\bT\bW)}|^2,
\end{align}
where we bound $|G^{(b_{\al'}b_{\beta'})}_{c_{\al'}c_{\beta'}}|\lesssim N^{-\fb} $ by \eqref{eq:local_law}. 
If $\gamma\neq \al$, thanks to \eqref{e:Gest}, we have (ignore the indicator and the averaging over embeddings)
\begin{align}\label{e:yizhong2}
\eqref{e:yizhong}
\lesssim \frac{1}{Nd}\sum_{\gamma\in\qq{\mu}\setminus\{\al\}\atop x\in \cN_\gamma}N^{-\fb}\sum_{b_\al\sim c_\al}|G_{c_\al x}^{(\bT\bW)}|^2
\lesssim N^{-\fb}N^\fo \Phi\lesssim N^{-3\fb/4}\Phi.
\end{align}

Thus we can reduce \eqref{e:yizhong} to the case $\gamma=\al$ 
\begin{align}
    \label{e:GtGG}
\frac{1}{Z_{\cF^+}}\sum_{\bfi^+}I(\cF^+,\cG)N^{-\fb}\left(|G_{c_\al a_\al}^{(\bT \bW)}|^2+\sum_{
x\sim b_\alpha\atop
x\neq c_\alpha} |G^{(\bT \bW)}_{c_\alpha x}|^2\right).
\end{align} 
The terms involving $|G_{c_\al a_\al}^{(\bT \bW)}|^2$ can be bounded by the same way as in \eqref{e:yizhong2}. 
Next, we show that we can replace $G_{c_\al x}^{(\bT \bW)}$ in \eqref{e:GtGG} by $G_{c_\al x}^{(b_\al)}$.
\begin{align}\begin{split}\label{e:yizhong3}
    \eqref{e:GtGG}&= \frac{1}{Z_{\cF^+}}\sum_{\bfi^+}I(\cF^+,\cG)N^{-\fb}\sum_{
x\sim b_\alpha\atop
x\neq c_\alpha} |G^{(b_\al)}_{c_\alpha x}|^2+\cE,\\
|\cE|&\lesssim \frac{1}{Z_{\cF^+}}\sum_{\bfi^+}I(\cF^+,\cG)N^{-\fb}\sum_{
x\sim b_\alpha\atop
x\neq c_\alpha} |G^{(\bT\bW)}_{c_\alpha x}-G^{(b_\al)}_{c_\alpha x}| +\OO(N^{-3\fb/4}\Phi),
\end{split}\end{align}
where for the bound of $\cE$, we used that $|G^{(\bT\bW)}_{c_\alpha x}|,|G^{(b_\al)}_{c_\alpha x}|\lesssim 1$ from \eqref{eq:local_law}.

Thanks to \eqref{e:Greplace}, we can bound the difference $|G_{c_\al x}^{(\bT \bW)}-G_{c_\al x}^{(b_\al)}|$ by
\begin{align}\label{e:yizhong4}
\left|G^{(\bT \bW)}_{c_\alpha x}-G^{(b_\alpha)}_{c_\alpha x}\right|
\leq \sum_{\gamma\in\qq{\mu}\setminus\{\al\}}(|G^{(\bT b_\al )}_{c_\al b_\gamma}|^2+|G^{(\bT b_\al)}_{x b_{\gamma}}|^2)+\sum_{y\in \bT}(|G^{(b_\al)}_{c_\al y}|^2+|G^{(b_\al)}_{xy}|^2).
\end{align}
By plugging \eqref{e:yizhong4} into \eqref{e:yizhong3}, by the same argument as in \eqref{e:yizhong2}, we can bound $\cE$ in \eqref{e:yizhong3} as
\begin{align}\label{e:yizhong5} \frac{N^{-\fb}}{Z_{\cF^+}}\sum_{\bfi^+}I(\cF^+,\cG)\left(\sum_{\gamma\in\qq{\mu}\setminus\{\al\}}(|G^{(\bT b_\al )}_{c_\al b_\gamma}|^2+|G^{(\bT b_\al)}_{x b_{\gamma}}|^2)+\sum_{y\in \bT}(|G^{(b_\al)}_{c_\al y}|^2+|G^{(b_\al)}_{xy}|^2)\right)
    \lesssim N^{-3\fb/4}\Phi.
\end{align}
The claim \eqref{e:smallterm} follows from plugging \eqref{e:yizhong2}, \eqref{e:yizhong3} and \eqref{e:yizhong5} into \eqref{e:yizhong}.

For the first term on the right-hand side of \eqref{e:GGGxy}, we recall that $\al \neq \beta$ and $\al'\neq \beta'$. Then either $\{\al, \beta\}=\{\al', \beta'\}$, or some indices, say $\al,\al'$, only appears once (namely, $\al\neq \al', \beta'$ and $\al'\neq \al, \beta$). Then we can sum over $(b_\al, c_\al)$ and $(b_{\al'}, c_{\al'})$ separately (ignoring the indicator function)
\begin{align}\begin{split}\label{e:fbound1}
    &\phantom{{}={}}\frac{1}{(Nd)^2}
    \sum_{c_\al\sim b_\al}\sum_{c_{\al'}\sim b_{\al'}}  G^{(b_\al b_\beta)}_{c_\al c_\beta} \overline G^{(b_{\al'} b_{\beta'})}_{c_{\al'} c_{\beta'}} 
\\
    &\lesssim  \frac{1}{(Nd)^2}\sum_{c_\al\sim b_\al}\sum_{c_{\al'}\sim b_{\al'}} N^{-\fb/2}(|G^{( b_{\beta})}_{b_{\al} c_{\beta}}|+\Phi)|N^{-\fb/2}(|G^{( b_{\beta'})}_{b_{\al'} c_{\beta'}}|+\Phi)|\\
    &\lesssim  \frac{N^{-\fb}}{(Nd)^2}\sum_{c_\al\sim b_\al}\sum_{c_{\al'}\sim b_{\al'}}
    (|G^{( b_{\beta})}_{b_{\al} c_{\beta}}|^2+|G^{( b_{\beta'})}_{b_{\al'} c_{\beta'}}|^2+\Phi)    \lesssim N^{-\fb} \Phi,
\end{split}\end{align}
where in the first statement we used \eqref{e:sum_one_index}; the second statement follows from Cauchy-Schwarz inequality; and the last  statement follows from the Ward-identity bound \eqref{e:Gest}.
Otherwise if $\{\al, \beta\}=\{\al', \beta'\}$,  \eqref{e:Gccerror} gives
\begin{align}\label{e:fbound2}
    \frac{1}{Z_{\cF^+}}\sum_{\bfi^+}I(\cF^+,\cG) |G^{(b_\al b_\beta)}_{c_\al c_\beta}|^2\lesssim N^\fo\Phi.
\end{align}

We recall that  $J_1$ in \eqref{e:IIGU2} is obtained by averaging \eqref{e:GGGxy} over  $\alpha \neq \beta \in \qq{\mu}$ and  $\alpha' \neq \beta' \in \qq{\mu}$. By substituting \eqref{e:GGGxy}, \eqref{e:fbound1}, and \eqref{e:fbound2}, we conclude that:
\begin{align}\begin{split}\label{e:J1bound}
    J_1&=\sum_{\al\neq\beta\in\qq{\mu}}\frac{1}{(d-1)^{2\ell}Z_{\cF^+}}\sum_{\bfi^+}\bE\left[\bm1(\cG, \wt \cG\in \Omega) I(\cF^+,\cG)  |G^{(b_\al b_\beta)}_{c_\al c_\beta}|^2  \right]\\
    &+\frac{1}{Z_{\cF^+}}\sum_{\bfi^+}\bE\left[\bm1(\cG,\wt\cG\in \Omega) I(\cF^+,\cG) \left(\frac{(d-1)^{2\ell}}{N^{\fb}}\sum_{\al\in\qq{\mu}}\sum_{
x\sim b_\alpha,
x\neq a_\alpha} |G^{(b_\al)}_{c_\al x}|^2 +\frac{\Phi}{N^{\fb/2}}\right) \right],\\
&\lesssim \frac{1}{Z_{\cF^+}}\sum_{\bfi^+}\bE\left[\bm1(\cG,\wt\cG\in \Omega) I(\cF^+,\cG) \left(\frac{1}{N^{3\fb/4}}\sum_{
x\sim b_\alpha,
x\neq a_\alpha} |G^{(b_\al)}_{c_\al x}|^2 +N^\fo \Phi\right) \right].
\end{split}\end{align}
where in the last statement we used \eqref{e:fbound2} and the permutation invariance of the vertices, so that the expectation does not depend on $\al$. By the same argument we can also bound $J_2$ in \eqref{e:IIGU2} as, 
\begin{align}\label{e:J2bound}
    J_2\lesssim \frac{1}{N^{3\fb/4}Z_{\cF^+}}\sum_{\bfi^+}\bE\left[\bm1(\cG,\wt\cG\in \Omega) I(\cF^+,\cG) \left(\frac{1}{N^{3\fb/4}}\sum_{
x\sim b_\alpha,
x\neq a_\alpha} |G^{(b_\al)}_{c_\al x}|^2 +(d-1)^{2\ell}N^\fo \Phi\right) \right].
\end{align}

By plugging \eqref{e:J1bound} and \eqref{e:J2bound} into \eqref{e:IIGU2}, we conclude
\begin{align}\begin{split}\label{e:IIGU3}
    \eqref{e:IIGU2}
    &\lesssim \frac{1}{Z_{\cF^+}}\sum_{\bfi^+}\bE\left[\bm1(\cG,\wt\cG\in \Omega) I(\cF^+,\cG) \left(\frac{1}{N^{3\fb/4}}\sum_{
x\sim b_\alpha,
x\neq a_\alpha} |G^{(b_\al)}_{c_\al x}|^2 +N^\fo \Phi\right) \right] \\
  &\lesssim 
   \frac{1}{N^{3\fb/4} Z_{\cF^+}}\sum_{\bfi^+}\bE\left[\bm1(\cG\in \Omega) I(\cF^+,\cG) \sum_{
x\sim b_\alpha,
x\neq a_\alpha} |G^{(b_\al)}_{c_\al x}|^2 \right]+\bE[N^\fo\Phi ]\\
&\lesssim \frac{1}{N^{3\fb/4}(Nd)} \sum_{b_\al\sim c_\al}  \bE\left[ I(\{c_\al,b_\al\},\cG)\sum_{
x\sim b_\alpha,
x\neq a_\alpha}  |G^{(b_\al)}_{c_\al x}|^2 \right] +\bE[N^\fo\Phi ]\\
&\lesssim \frac{1}{N^{1+3\fb/4}} \sum_{o\in\qq{N}}   \bE\left[ I(\{i,o\},\cG) |G^{(o)}_{ij}|^2 \right] +\bE[N^\fo\Phi ],
\end{split}\end{align}
where in the second statement, we dropped the indicator function $\bm1(\wt \cG\in \Omega)$; in the third statement sum over $\bfi^+\setminus\{b_\al,c_\al\}$; for the last statement, we used the permutation invariance of the vertices, so that $G_{ij}^{(o)}$ and $G_{c_\al x}^{(b_\al)}$ have the same distribution. 

Thus \eqref{e:IIGU2} and \eqref{e:IIGU3} together leads to the following bound 
    \begin{align*}
\frac{1}{N} \sum_{o\in\qq{N}} \bE[I(\{i,o\},\cG) |G_{ij}^{(o)}|^2 ]\lesssim  \frac{1}{N^{1+3\fb/4}} \sum_{o\in\qq{N}} \bE[I(\{i,o\},\cG) |G_{ij}^{(o)}|^2 ]+\bE[N^\fo\Phi ],
\end{align*}
and the claim \eqref{e:sameasdisconnect} follows from rearranging.

\end{proof}

\begin{proof}[Proof of Proposition \ref{lem:task2}]
Thanks to \Cref{l:diffG1}, we have
\begin{align*}
    \tG^{(\bT)}_{c_\al c_\beta}=G^{(b_\al b_\beta)}_{c_\al c_\beta}+\cE_{\al \beta},
\end{align*}
where $\cE_{\al \beta}$ as in \eqref{e:cEbound}. 
The statement follows from showing 
\begin{align}\begin{split}\label{e:G2bound}
     &\sum_{\gamma\in \qq{\mu}}\sum_{x\in \cN_\gamma}\frac{1}{Z_{\cF^+}}\sum_{\bfi^+}\bE[I(\cF^+,\cG)\bm1(\cG\in \Omega)|G_{c_\al x}^{(\bT\bW)}|^2 ]\lesssim (d-1)^\ell \bE[N^\fo ].\\
     &\sum_{\gamma\in\qq{\mu}\setminus\{\al,\beta\}}\frac{1}{Z_{\cF^+}}\sum_{\bfi^+}\bE[I(\cF^+,\cG)\bm1(\cG\in \Omega)|G_{c_\al b_\gamma}^{(\bT b_\al b_\beta)}|^2 ]\lesssim (d-1)^\ell \bE[N^\fo ],\\
     &\sum_{x\in\bT}\frac{1}{Z_{\cF^+}}\sum_{\bfi^+}\bE[I(\cF^+,\cG)\bm1(\cG\in \Omega)|G_{c_\al x}^{( b_\al b_\beta)}|^2 ]\lesssim (d-1)^\ell \bE[N^\fo ],
\end{split}\end{align}
where  $N_\gamma=\{x\neq c_\gamma: x\sim b_\gamma\}\cup\{a_\gamma\}$.

In the following we prove the first statement in \eqref{e:G2bound}, the others are similar, so we omit their proofs.
If $x\in \cN_\gamma$ with $\gamma\neq \al$, or $x=a_\al$, we can first sum over $b_\al\sim c_\al$, and  \eqref{e:Gest} gives
\begin{align}\label{e:qinkuang1}
   \frac{1}{Z_{\cF^+}}\sum_{\bfi^+}I(\cF^+,\cG)\bm1(\cG\in \Omega)|G_{c_\al x}^{(\bT\bW)}|^2\lesssim N^\fo\Phi.
\end{align}
Otherwise $x\sim b_\al, x\neq c_\al$. 
Thanks to \eqref{eq:local_law}, we have  $||G^{(\bT \bW)}_{c_\alpha x}|^2-|G^{(b_\alpha)}_{c_\alpha x}|^2|\lesssim N^{-\fb}|G^{(\bT \bW)}_{c_\alpha x}-G^{(b_\alpha)}_{c_\alpha x}|$.
We recall the upper bound on $|G^{(\bT \bW)}_{c_\alpha x}-G^{(b_\alpha)}_{c_\alpha x}|$ from \eqref{e:yizhong4} and \eqref{e:yizhong5}, then it follows that
\begin{align}\label{e:qinkuang2}
 \frac{1}{Z_{\cF^+}}\sum_{\bfi^+}I(\cF^+,\cG)\bm1(\cG\in \Omega)|G^{(\bT \bW)}_{c_\alpha x}|^2
 =\frac{1}{Z_{\cF^+}}\sum_{\bfi^+}I(\cF^+,\cG)\bm1(\cG\in \Omega)|G^{(b_\alpha)}_{c_\alpha x}|^2
 +\OO(N^{-\fb/2}\Phi).
\end{align}
By combining \eqref{e:qinkuang1} and \eqref{e:qinkuang2}, we conclude
    \begin{align}\begin{split}\label{e:decomp}
    &\phantom{{}={}}\sum_{\gamma\in \qq{\mu}}\sum_{x\in \cN_\gamma}\frac{1}{Z_{\cF^+}}\sum_{\bfi^+}\bE[I(\cF^+,\cG)\bm1(\cG\in \Omega)|G_{c_\al x}^{(\bT\bW)}|^2 ]\\
    &\lesssim 
    \sum_{x\sim b_\al, x\neq c_\al}\frac{1}{Z_{\cF^+}}\sum_{\bfi^+}\bE[I(\cF^+,\cG)\bm1(\cG\in \Omega)|G_{c_\al x}^{(b_\al)}|^2 ]
    +(d-1)^{\ell}\bE[N^\fo\Phi ]
    \lesssim (d-1)^{\ell}\bE[N^\fo\Phi ].
    \end{split}\end{align}
where in the last line we used \eqref{e:sameasdisconnect} to bound the first term; The first claim in \eqref{e:G2bound} follows from combining  \eqref{e:qinkuang1} and \eqref{e:decomp}.
\end{proof}

\bibliography{ref}{}
\bibliographystyle{abbrv}

\end{document}